\documentclass{article}

\usepackage{arxiv}
 \usepackage{xcolor}
\usepackage[utf8]{inputenc} % allow utf-8 input
\usepackage[T1]{fontenc}    % use 8-bit T1 fonts
\usepackage{hyperref,bbm}       % hyperlinks
\usepackage{url}            % simple URL typesetting
\usepackage{booktabs}       % professional-quality tables
\usepackage{amsfonts}       % blackboard math symbols
\usepackage{amssymb}        % for \leqslant and other symbols
\usepackage{nicefrac}       % compact symbols for 1/2, etc.
\usepackage{microtype}      % microtypography
\usepackage{lipsum}
\usepackage{amsmath}
\usepackage{amsthm}
\usepackage{appendix}
\usepackage{indentfirst}

\numberwithin{equation}{section}
\usepackage{graphicx}
\numberwithin{equation}{section}

\newtheorem{theorem}{Theorem}[section]
\newtheorem{proposition}[theorem]{Proposition}
\newtheorem{remark}[theorem]{Remark}

\newtheorem{lemma}[theorem]{Lemma}
\newtheorem{corollary}[theorem]{Corollary}

\graphicspath{ {./images/} }

\title{Scaling limit of 1+1 dimensional directed polymer with power-law tail and spatial correlated noise }

\author{
 Junjie Cao \\
  Beijing Normal-Hong Kong Baptist University\\
  \texttt{caojunjie@bnbu.edu.cn} \\
   \And
Guanglin Rang \\
  Wuhan University\\
  \texttt{glrang.math@whu.edu.cn} \\
  \And
  \\
}

\begin{document}
\maketitle
\begin{abstract}
We study a $(1+1)$-dimensional directed polymer in a spatially correlated random environment generated by power-law tail variables:
    $\omega(i,x)=\sum_{y\in\mathbb Z}\psi_{y-x}\xi(i,y), \psi_y\sim \lambda_r |y|^{-r}, r\in(1/2,1)$,
where the variables $\xi(i,y)$ are i.i.d. and have a regularly varying right tail with exponent $\alpha>2$.
The spatial covariance of the environment has long-range decay with Hurst parameter
    $H=\frac32-r\in(1/2,1)$.
We identify the limiting fluctuations of the log-partition function in the intermediate disorder regime and show that the critical tail exponent is
    $\alpha_c=\frac{3}{H}=\frac{6}{3-2r}$.
When $\alpha>\alpha_c$, the model has the same scaling limits as the corresponding Gaussian spatially correlated polymer: if $\beta_NN^{H/2}\to\beta\in(0,\infty)$, the centered log-partition function converges to the logarithm of the solution of the stochastic heat equation driven by fractional spatial noise; if $\beta_NN^{H/2}\to0$, its normalized fluctuation converges to a centered Gaussian law. In the regime $2<\alpha\le\alpha_c$, at the scale $\beta_NN^{H/2}=\beta N^{H/2}/l(N^{3/2})$, the log-partition function still satisfies Gaussian fluctuation. The main ingredient is a truncation comparison argument adapted to long-range moving-average environments, together with an invariance principle for polynomial chaos. Due to the non-locality of the environments, we perform a far-near field analysis, as well as multiscale analysis, to prove that the truncated version does not change the log-partition function at the corresponding scales.
\end{abstract}

% keywords can be removed

\section{Introduction}

\subsection{Setup of the model}

 Directed polymers are stochastic models describing the trajectory of a polymer chain in a random environment, typically defined on a lattice. From a mathematical perspective, such a polymer can be represented as a directed random walk path \(S = (S_1, S_2, \dots, S_N)\) starting from the origin in \(\mathbb{Z}^d\), where each step is influenced by a disorder \(\omega(i, S_i)\). We denote \(\mathcal{S}_0^N = \{S \mid (S_0, S_1, S_2, \dots, S_N),\;\; S_0 = 0\}\). Here, \(\omega = \{\omega(i, x), (i, x) \in \mathbb{N} \times \mathbb{Z}^d\}\) is a family of real-valued random variables representing the disordered environment. In the following, \(\mathbb{P}_N\) and \(\mathbb{Q}\) denote the probability measures associated with the random walk \(S\) and the environmental variables \(\omega\), respectively, and \(\mathbb{E}_{\mathbb{P}_N}\), \(\mathbb{E}_{\mathbb{Q}}\) denote the corresponding expectations.
For a fixed environment \(\omega\), the \(N\)-step energy of a path \(S\) is given by  
\[
H_N^\omega(S) = \sum_{i=1}^N \omega(i, S_i),
\]  
and the associated random polymer measure is defined via the Gibbs form  
\[
\mathbb{P}_{N,\beta}^{\omega}(S) = \frac{1}{Z_N(\beta, \omega)} e^{\beta H_N^\omega(S)} \mathbb{P}_N(S),
\]  
where \(\beta\) is the inverse temperature, and the partition function is defined as  
\begin{equation}
Z_N(\beta, \omega) = \sum_{S \in \mathcal{S}_0^N} e^{\beta H_N^\omega(S)} \mathbb{P}_N(S) = \mathbb{E}_{\mathbb{P}_N}\left(e^{\beta H_N^\omega(S)}\right) \label{partition function}.
\end{equation}
Moreover, for any \(x \in \mathbb{Z}^d\), the partition function from the origin to \(x\) is defined as  
\begin{equation}
Z_N(\beta, \omega; x) = \sum_{S \in \mathcal{S}_0^N} e^{\beta H_N^\omega(S)} \mathbb{P}_N(S \mid S_N = x) = \mathbb{E}_{\mathbb{P}_N}\left(e^{\beta H_N^\omega(S)} \mathbbm {1}_{\{S_N = x\}}\right).
\end{equation}

As a fundamental model in statistical physics and probability theory, directed polymer in random environment (DPRE for short) has been greatly  developed over the past twenty years. The investigation of this model    aims at to look for new universal exponents which encodes some important information characterizing how a pure entropy behavior can be effected by the disorders,  especially by strong disorders. Although there has a widely accepted conjecture on the exponents, very few examples can confirm its validity in dimension $1+1$. Meanwhile, a rich variety of different regimes, which is called intermediate disorder regimes,  exists between the weak and strong disorder by
scaling the inverse temperature with the length of the
polymer $N$. For more details see \cite{2014AKQ1, 2010Intermediate,Caravenna2017}. It is in these regimes that one can establish the convergence of the DPRE to the Stochastic Heat Equation (SHE) and (via the
Cole-Hopf transformation) of the Kardar-Parisi-Zhang (KPZ) equation or Gaussian fluctuations for $d=1$ in \cite{2014AKQ1} and for $d=2$ in \cite{2023CFS}. 

These discoveries are established under the assumption of the existence of exponential moments of $\omega$. Dey and Zygours in \cite{DZ16} showed that the exponential moments condition can be relaxed to the heavy-tailed law case, which verifies the earlier conjectures (\cite{2014AKQ1}) that the six moments is needed for the intermediate scaling limit of the partition function. Berger and Lacoin in \cite{2021BerLac} considered the same scaling limit of the partition function of product form when the environmental variables admitting power-law tail distributions.  In this paper, we shall keep to explore this topic in the case where the environmental variables $\omega$
are spatially correlated, in contrast to the independent assumption adopted in previous studies. Correlated environmental variables in DPRE were first investigated in \cite{1989MHKZ}, where numerical simulations were performed to extract the critical exponents for the correlated case.
 Since then more and more correlated models are investigated by mathematicians and physicists. Lacoin in \cite{2011Lacoin} studies the influence of disorder with correlation to Brownian motion; Chu and Kardar in \cite{PhysRevE2016} simulate the distribution of the free energy of DPRE in the presence of spatially correlated noise.
 Shen et al in \cite{SHEN2021} studies the scaling limit of random walk in the occupation field, which exhibits a long-range correlation with respect to time-space points, given by a Poisson system of independent random walks on $\mathbb Z$. 
 The moving average time-space fields also are used as the environmental variables in DPRE. The corresponding scaling limit are considered in \cite{RANG20203408} for simple symmetrical random walk in spatially correlated environment, in \cite{GC23} for long-range random walk in the same environment as in \cite{RANG20203408} and in \cite{2024RWS} for the same random walk as in \cite{GC23} but with time correlated spatially independent Gaussian environment, respectively. Berger and his collaborators did a series works on the DPRE in correlated environment, see references cited in \cite{BerLeg2024}, where  a comprehensive investigation was implemented on a generalized Poland-Scheraga model.

Here,  we investigate the limiting behavior of the log-partition function in the setting where the environment is temporally independent but exhibits long-range spatial correlations, with the polymer displaying diffusive behavior, i.e., \(\nu = \frac{1}{2}\). As the tail probability of the environmental variables increases, the limiting behavior ranges from the Hopf-Cole solution of the generalized KPZ equation to a normal distribution. Before formally stating the theorems, we first introduce some assumptions and notations.

\textbf{(A)} The environment \( \omega = \omega(i, x), (i, x) \in \mathbb{N} \times \mathbb{Z} \) is taken to be of the form:

\begin{equation}
\omega(i, x) =
\begin{cases} 
\displaystyle \sum_{y=-\infty}^{\infty} \psi_{y-x} \xi(i, y) & r \in (1/2, 1), \\[1.2em]
\xi(i, x) & r = 1  
\end{cases}
\label{omega}
\end{equation}
where \( \psi_y > 0, \psi_y \sim \lambda_r |y|^{-r}, \lambda_r > 0 \), and \( \{ \xi(i, x); (i, x) \in \mathbb{N} \times \mathbb{Z} \} \) is a family of independent and identically distributed random variables.

\begin{remark}
  Here we only consider the case \( r \in (1/2, 1) \) fixed throughout the paper. For consistency with classical results in the case of i.i.d. environment variables \cite{DZ16}, we extend the definition to \( r = 1 \).
\end{remark} 

We assume that the cumulative distribution function \( F(x) := \mathbb{Q}(\xi \leq x) \) has a regularly varying right tail. In other words, we suppose:

$$
\bar{F}(x) := 1 - F(x) = x^{-\alpha} L(x) \quad \forall \, x > 0,
$$
for some \( \alpha > 0 \) and some slowly varying function \( L(x) \) (a slowly varying function satisfies \( \frac{L(tx)}{L(x)} \to 1 \) as \( x \to \infty \) for any \( t > 0 \)), with the requirement of \( \alpha r > 1 \). Moreover, \( L(\cdot) \) is chosen such that \( \bar{F}(x) \) remains bounded, thereby defining a proper cumulative distribution function. And, the left tail of the cumulative distribution function \( F(x) \) is bounded by the right tail.

\textbf{(B)} When \( \alpha\textgreater2 \), we require the random variables to satisfy:

$$
\mathbb{E}_\mathbb{Q}(\xi) = 0, \quad \mathbb{E}_\mathbb{Q}(\xi^2) = 1.
$$
For any \( r \in (\frac{1}{2}, 1) \), we have (see \cite{HOSKING1996261} for details)

\begin{equation}
\mathbb{E_Q}[\omega(i, x)\omega(j, z)] = \delta_{ij} \gamma(x - z) = \delta_{ij} \sum_{y=-\infty}^{\infty} \psi_y \psi_{y+x-z},   \label{gamma}
\end{equation}
where \( \delta_{i,j} \) is the Kronecker delta, and \( \gamma(k) \sim \lambda_r |k|^{1-2r} \) for sufficiently large integers \( k \); here, \( \lambda_r = c_r^2 \frac{\Gamma(2r-1)\Gamma(1-r)}{\Gamma(r)} \).

\begin{remark}
    Throughout this paper, the two parameters $r\in(1/2,1)$, \( H = 3/2 - r\in (1/2,1) \)  are fixed unless otherwise stated. Actually, $H$, called the Hurst exponent, corresponds to a fractional Gaussian field. The parameter \( c_r \)
  can be chosen so that \( \lambda_r = H(2H - 1) \) holds. Consequently, we have

\begin{equation}
\gamma([z]) \sim K(z) \quad \text{as } |z| \to \infty,  \label{gamma1}
\end{equation}
where \( [z] \) denotes the integer part of \( z \), and the function \( K \) is defined as

\begin{equation}
K(z) = H(2H - 1)|z|^{2H-2}. \label{gamma2}
\end{equation}
\end{remark}
For \( t > 1 \), define

$$
l(t) := \inf \{ x \mid \bar{F}(x) \leq 1/t \},
$$
From the definition of \( l(t) \), we have \( l(t) = t^{1/\alpha} L_0(t) \) as \( t \to \infty \), where \( L_0(t) \) is some slowly varying function. For any \( x \in \mathbb{R} \), define \( x_+ = \max \{ x, 0 \} \), \( x_- = \max \{ -x, 0 \} \). 

Now, under the above assumptions we have the two main results of this paper:

\begin{theorem}
\label{thm:main_log_partition}
Assume that conditions \textbf{(A)} and \textbf{(B)} hold. Let the tail exponent satisfy $\alpha > \frac{6}{3-2r}$. For a sequence of inverse temperatures $\beta_N$,  the log-partition function $\ln Z_N(\beta_N; \omega)$ exhibits the following two limit behaviors:

%\medskip
\noindent
\textbf{(a)} 
If $\beta_N N^{H/2} \to \beta \in (0, \infty)$ as $N \to \infty$, then the log-partition function centered by its multi-order polynomial drift converges in distribution to the logarithm of the continuous Wiener chaos:
$$
    \ln Z_N(\beta_N; \omega) - N \sum_{j=2}^{\lfloor 2/H \rfloor} \frac{\beta_N^j}{j!} \mathbb{E}_{\mathbb{Q}}[\bar{\xi}^j] \sum_{y=-\infty}^{\infty} \psi_y^j \stackrel{D}{\Longrightarrow} \ln Z_{\sqrt{2}\beta}(1, \cdot),
$$
where $Z_{\sqrt{2}\beta}(t, \cdot)$ is the unique solution to the following stochastic heat equation (SHE)
$$
    \partial_t u = \frac{1}{2} \Delta u + \sqrt{2}\beta u \dot{W}, \quad u(0,x) = \delta_0(x),
$$
driven by a centered Gaussian space-time noise $\dot{W}$ with covariance structure
$$
    \mathbb{E}_{\mathbb{Q}}[\dot{W}(t,x) \dot{W}(s,y)] = H(2H-1)\delta(t-s) |x-y|^{2H-2}.
$$

%\medskip
\noindent
\textbf{(b)} 
Let $\beta_N N^{H/2} \to 0$ as $N \to \infty$, and there exists a constant $a$ with $0 < a < A_{*}(\alpha, H)$, which is defined in \eqref{A_alpha_H}, such that 
\[
\beta_N N^{H/2} \gg N^{-a}.
\]
Then the normalized fluctuation converges to a Gaussian distribution:
$$
    \frac{1}{\beta_N N^{H/2}} \left( \ln Z_N(\beta_N; \omega) - N \sum_{j=2}^{p'} \frac{\beta_N^j}{j!} \mathbb{E}_{\mathbb{Q}}[\bar{\xi}^j] \sum_{y=-\infty}^{\infty} \psi_y^j \right) \stackrel{D}{\longrightarrow} \mathcal{N}(0, \sigma_H^2),
$$
where $p'=\left\lfloor\frac2H\right\rfloor$. \(\text{More generally, \(p'\) can be replaced by any integer }
p\le \lfloor2/H\rfloor
\text{ such that }
\frac{N\beta_N^{p+1}}{\beta_NN^{H/2}}\to0.\) The variance is given by
$$
    \sigma_H^2 = 2H(2H-1) \int_0^1 \int_{\mathbb{R}^2} \rho(t,x)\rho(t,y)|x-y|^{2H-2} \,dx \,dy \,dt,
$$
with $\rho(t,x)$ being the standard heat kernel.
\end{theorem}

\medskip
\begin{theorem}
\label{thm:main_gaussian}
Assume that conditions \textbf{(A)} and \textbf{(B)} hold. Let the tail exponent satisfy $\alpha \in (2, \frac{6}{3-2r}]$ and the sequence of inverse temperatures chosen as $\beta_N = \beta l(N^{3/2})^{-1}$. When the scaling $\beta_N N^{H/2} \to 0$ as $N \to \infty$ , then the normalized fluctuation of the log-partition function converges in distribution to a centered Gaussian distribution:
$$
    \frac{1}{\beta_N N^{H/2}} \left( \ln Z_N(\beta_N; \omega) - N \sum_{j=2}^{p'} \frac{\beta_N^j}{j!} \mathbb{E}_{\mathbb{Q}}[\bar{\xi}^j] \sum_{y=-\infty}^{\infty} \psi_y^j \right)  \stackrel{D}{\longrightarrow} \mathcal{N}(0, \sigma_H^2),
$$
where \(p'\) is any fixed integer satisfying
\[
        1\le p'<\alpha,
        \qquad\text{and}\quad
        \frac{N\beta_N^{p'+1}}{\beta_NN^{H/2}}\longrightarrow0 ~\text{as}~N\to\infty.
\]
The variance $\sigma_H^2$ is given by the integral:
$$
    \sigma_H^2 = 2H(2H-1) \int_0^1 \int_{\mathbb{R}^2} \rho(t,x)\rho(t,y)|x-y|^{2H-2} \,dx \,dy \,dt.
$$
\end{theorem}

\subsection{Strategy of the proof}
In our proof, a truncation procedure is involved along with several modified partition functions; therefore, a detailed exposition is necessary to clarify these steps. To ensure that the partition function is well-defined during the calculation, we need to truncate the large values of the i.i.d. family \(\{\xi(i,x)\}\). The truncated family is denoted by \(\{\bar{\xi}(i,x)\}\), and the corresponding random environment is denoted by \(\bar{\omega}(i,x)\).

For a given truncation level \(k_N > 0\),  the truncated family \(\{\bar{\xi}(i,x)\}\) is defined by
$$
\bar{\xi}(i,x) = \xi(i,x) \cdot \mathbbm{1}_{\{|\xi(i,x)| \le k_N\}} - \mathbb{E}_{\mathbb{Q}}[\xi(i,x) \cdot \mathbbm{1}_{\{|\xi(i,x)| \le k_N\}}],
$$
and the corresponding truncated random environment \(\bar{\omega}(i,x)\) is then given by
\begin{equation}
\bar{\omega}(i,x) = \sum_{y=-\infty}^{\infty} \psi_{y-x} \, \bar{\xi}(i, y). \label{truncated_omega}
\end{equation}
The choices of the truncation level \(k_N\) and the inverse temperature \(\beta_N\) are given directly by
\begin{equation} \label{k_N}
k_N = 
\begin{cases}
\beta_N^{-1}, & \alpha \in \left( \frac{6}{3-2r}, \infty \right), \\[10pt]
\beta_N^{-1} \cdot \dfrac{l \bigl(N^{\frac{3}{2}} (\ln N)^{\eta}\bigr)}{l\bigl(N^{\frac{3}{2}}\bigr)}, & \alpha \in \left(2, \frac{6}{3-2r}\right],
\end{cases}
\end{equation}
where \(\eta \in (1, \alpha)\), \(\beta \in [0, \infty)\) and
\begin{equation} \label{beta_N}
\beta_N = 
\begin{cases}
\beta N^{-\frac{H}{2}}, & \alpha \in \left( \frac{6}{3-2r}, \infty \right), \\[10pt]
\beta \, l\bigl(N^{\frac{3}{2}}\bigr)^{-1}, & \alpha \in \left(2, \frac{6}{3-2r}\right].
\end{cases}
\end{equation}

\begin{remark}
    The selection of $\beta_N$ is based on several considerations. Drawing on the choice of $\beta_N$ in \cite{DZ16}, we suppose that as the random environment transitions from exponential moments to finite moments, the scaling of $\beta_N$ remains valid up to a critical moment. Accordingly, the first candidate scaling is taken from \cite{RANG20203408}. The calculation of the critical moment relies on the entropy-energy balance argument presented in \cite{Biroli_2007}
\end{remark}

The partition function corresponding to the random environment \(\bar{\omega}(i,x)\) is also given in the form of \eqref{partition function} and denoted by \(Z_N(\beta_N;\bar{\omega}(i,x))\). Through some tedious calculations, we obtain that
\begin{equation}
\tau_N^{-1}|\ln Z_N(\beta_N;\omega) - \ln Z_N(\beta_N;\bar{\omega})| \xrightarrow{\mathbb{Q}} 0, \label{energy-entropy balance}
\end{equation}
 where $\tau_N$ characterizes the convergence rate. Consequently, the limit of the log-partition function \(\ln Z_N(\beta_N;\omega)\) for the original random environment as \(N \to \infty\) can be derived from the truncated log-partition function \(\ln Z_N(\beta_N;\bar{\omega})\).

The truncated normalized partition function can be linearized as:
\begin{equation}
Z_N(\beta_N; \bar{\omega}) e^{-N \bar{\lambda}(\beta_N)} = \mathbb{E}_{\mathbb{P}_N} \left( \prod_{i=1}^N \bigl( 1 + \beta_N \hat{\omega}(i, S_i) \bigr) \right), \label{truncation normalized partition function}
\end{equation}
where
$$
\hat{\omega}(i, x) = \frac{e^{\beta_N \bar{\omega}(i,x) - \bar{\lambda}(\beta_N)} - 1}{\beta_N},
$$
and
$$
\bar{\lambda}(\beta_N) = \ln \left( \mathbb{E_Q} \left( e^{\beta_N \bar{\omega}} \right) \right) = \sum_{y=-\infty}^{\infty} \ln \left( \mathbb{E_Q} \left( e^{\beta_N \psi_{y-x} \bar{\xi}(i,y)} \right) \right).
$$
Note that when \(\beta_N\) is sufficiently small (high temperature) one has \(\hat{\omega} \approx \bar{\omega}\). Hence we also consider the modified partition function of the linear form:
\begin{equation}
\mathfrak{Z}_N(\beta_N; \bar{\omega}) :=\mathbb{E}_{\mathbb{P}_N} \left( \prod_{i=1}^N \bigl( 1 + \beta_N \bar{\omega}(i, S_i) \bigr) \right).\label{modified partition function}
\end{equation}

\begin{remark}
    In fact, whether \(\hat{\omega} \approx \bar{\omega}\) holds as \(\beta_N \to 0\) depends on the value of \(\alpha\). A simple calculation shows that when \(\alpha > 2\), the error is negligible; while, when \(\alpha \le 2\), it is not easy to show this (see Lemma \ref{lemma:env_covariance_bounds}).
\end{remark}

By controlling the differences between \eqref{truncation normalized partition function} and \eqref{modified partition function}, we can get the limit of the normalized truncated partition function  by the modified partition function \(\mathfrak{Z}_N(\beta_N; \bar{\omega})\). The key point is to apply the Lindeberg principle for polynomial expansions to transform the environmental variables \(\bar{\omega}(i,x)\) into Gaussian variables \(\mu(i,x)\), and then use the hypercontractivity technique for multilinear polynomials, developed by Mossel et al.~\cite{MDO2005} to disordered systems, to control the error between \(\mathfrak{Z}_N(\beta_N; \bar{\omega})\) and \(\mathfrak{Z}_N(\beta_N; \mu)\). Finally, by employing the method of weighted U-statistics for Gaussian variables, the limiting distribution of \(\mathfrak{Z}_N(\beta_N; \mu)\) can be obtained. This method was adopted by  Caravenna et al.~\cite{Caravenna2017} in the independent case, and then by Chen--Gao~\cite{GC23} to spatially correlated case under exponential-moment assumptions. Here we develop this framework further under finite moment. 

In order to estimate the error between \(\hat{\omega}\) and \(\bar{\omega}\), we need the following lemma, whose proof is postponed to  the appendix.

\begin{lemma}\label{lem:cgf_expansion}
Assume conditions \((A)\) and \((B)\), and suppose that the right tail of
\(\xi\) satisfies
\[
        \mathbb Q(\xi>x)=x^{-\alpha}L(x),
        \qquad \alpha>2.
\]
For fixed \(q>0\), denote log-cumulant generating function by
\[
        \Lambda_N(q)
        :=
        \log\mathbb E_{\mathbb Q}
        \left[
        \exp\{q\beta_N\bar\omega(i,x)\}
        \right].
\]
Let \(\theta\in(2,\alpha)\) satisfy \(r\theta>1\), and put
\(p:=\lfloor\theta\rfloor\).  Assume \(\beta_N\to0\), \(k_N\beta_N\) is
bounded away from zero, and that for some
\(\varepsilon\in(0,\alpha-\theta)\),
\[
        \left(
        \beta_N^{p+1-\theta}
        +
        \beta_N^{\alpha-\theta-\varepsilon}
        \right)
        \exp\{q\beta_Nk_N\sup_y\psi_y\}
        \longrightarrow0 .
\]
Then, as \(N\to\infty\),
\[
        \frac1{\beta_N^\theta}
        \left|
        \Lambda_N(q)
        -
        \sum_{y\in\mathbb Z}
        \log\left(
        1+
        \sum_{j=1}^p
        \frac{(q\beta_N\psi_y)^j}{j!}
        \mathbb E_{\mathbb Q}[\bar\xi^j]
        \right)
        \right|
        \longrightarrow0 .
\]
\end{lemma}

  Up to this point, the limiting behavior has been described for the truncated
environment \(\bar\omega\).  To transfer this limit back to the original
environment \(\omega\), we must show that the removed tail
\(\Delta\omega=\omega-\bar\omega\) produces an error negligible at the
logarithmic scale as \eqref{energy-entropy balance}. This step is more
delicate than that in the i.i.d. case.  Indeed, a single removed variable
\(\Delta\xi(i,y):=\xi(i,y)-\bar\xi(i,y)\) contributes to many sites through the moving-average kernel
$$
        \Delta\omega(i,x)
        =
        \omega(i,x)-\bar\omega(i,x)
        =
        \sum_{y\in\mathbb Z}\psi_{y-x}\Delta\xi(i,y).
$$
Thus the error is not localized on the sites visited by the walk, that is, even when the
path is at \(S_i=x\), the removed variables \(\Delta\xi(i,y)\) for all
\(y\in\mathbb Z\) contribute through the weights \(\psi_{y-x}\).  Since our
goal is to compare the two log-partition functions, we rewrite the difference
as
\[
        \log Z_N(\beta_N;\omega)
        -
        \log Z_N(\beta_N;\bar\omega)
        =
        \log R_N\]
        with
        \(
        R_N:=
        \frac{Z_N(\beta_N;\omega)}
             {Z_N(\beta_N;\bar\omega)} .
\)

To control $R_N$, we decompose the path space according to the maximal spatial range
of the polymer as that in \cite{DZ16}.  That is, a subtle sequence of scales $\{h_j\}$ and hence a sequence of blockes of path space  \(\mathcal B_1,\mathcal A_2,\ldots,\mathcal A_m\) leads to 
an elementary bounds
\[
       | R_N-1|\le F_N+U_N.
\]
Here \(L_N\) is the outer-shell contribution under the truncated polymer measure, whereas \(U_N\) is the corresponding outer-shell contribution in the original environment.  See section 2 for more explanation. Thus the comparison is reduced to controlling the first block and showing that the outer shells have negligible contribution.

A useful observation is that, for every
set of paths \(A\),
\begin{equation}\label{eq:stationarity-partition-section2}
        \mathbb E_{\mathbb Q}[Z_N(\beta_N;\bar\omega;A)]
        =
        \mathbb P_N(A)\,
        \mathbb E_{\mathbb Q}[Z_N(\beta_N;\bar\omega)] .
\end{equation}
This suggests that if we can succeed in replacing the original environment $\omega$ by the truncated $\bar\omega$ in the numerator of the ratios, then it  can be controlled by the entropy of the random walk, whose path wanders far from the start point. Specifically, for example, we introduce
\[
        G_N(\delta)
        :=
        \left\{
        Z_N(\beta_N;\bar\omega)
        \ge
        \delta\,
        \mathbb E_{\mathbb Q}[Z_N(\beta_N;\bar\omega)]
        \right\}.
\]
On \(G_N(\delta)\),
\[
    \frac{Z_N(\beta_N;\omega;A)}
         {Z_N(\beta_N;\bar\omega)}
    \xrightarrow{\text{``replace''}}
    \frac{Z_N(\beta_N;\bar\omega;A)}
         {Z_N(\beta_N;\bar\omega)}
    \le
    \frac{1}{\delta}\,
    \frac{Z_N(\beta_N;\bar\omega;A)}
         {\mathbb E_{\mathbb Q}[Z_N(\beta_N;\bar\omega)]}.
\]where ``replace'' means using $Z_N(\beta_N;\bar\omega;A)$ to control $Z_N(\beta_N;\omega;A)$. Taking \(\mathbb Q\)-expectation gives
\[
        \mathbb E_{\mathbb Q}
        \left[
        \frac{Z_N(\beta_N;\bar\omega;A)}
             {Z_N(\beta_N;\bar\omega)}
        ;G_N(\delta)
        \right]
        \le
        \delta^{-1}\mathbb P_N(A).
\]

The way we use $Z_N(\beta_N;\bar\omega;A)$ to control $Z_N(\beta_N;\omega;A)$ is through far--near decomposition and exponential change of measure. We notice that
\[
\begin{aligned}
        Z_N(\beta_N;\omega;A)
        &=
        \int_A
        \exp\{\beta_N H_N^{\bar\omega}(S)\}
        \exp\{\beta_N\Delta H_N(S)\}
        \,\mathbb P_N(dS)                                      \\
        &=
        \int_A
        \exp\{\beta_N H_N^{\bar\omega}(S)\}
        \exp\{\beta_N\Delta H_N^{\mathrm{near},h}(S)\}
        \exp\{\beta_N\Delta H_N^{\mathrm{far},h}(S)\}
        \,\mathbb P_N(dS).
\end{aligned}
\]
It involves the perturbation \(e^{\beta_N\Delta H_N(S)}\) generated by the removed tail $\Delta\omega $.  The role of the far--near decomposition is to show that this perturbation is negligible. For a complete descriptions see Section 2.1--2.3.

In summary, the first block is controlled by combining the near-field
$L^1$ bound with the far-field $L^2$ estimate.  On the outer
shells, the same two estimates are combined with the random-walk entropy cost
\(\mathbb P_N(\mathcal A_j)\le Ce^{-I_{j-1}/2}\).

Finally, we present the strategy for determining the limit when \(\alpha > 2\). Define \(\bar{\zeta}(i,x) = \beta_N \hat{\omega}(i,x)\). Expanding the truncated normalized partition function gives
\begin{equation}\label{truncated_partition}
\begin{aligned}
Z_N(\beta_N; \bar{\omega}) \exp\left(-N \bar{\lambda}_N(\beta_N)\right) &= \mathbb{E}_{\mathbb{P}_N} \left[ \exp\left( \sum_{i=1}^N \beta_N \bar{\omega}(i, S_i) - \bar{\lambda}_N(\beta_N) \right) \right] \\
&= \mathbb{E}_{\mathbb{P}_N} \left[ \prod_{i=1}^N \left(1 + \bar{\zeta}(i, S_i)\right) \right] \\
&= 1 + \sum_{1 \leq i \leq N, \, x \in \mathbb{Z}} \bar{\zeta}(i, x) p_N(i, x) \\
&\quad + \sum_{1 \leq i_1 < i_2 \leq N, \, x_1, x_2 \in \mathbb{Z}} \bar{\zeta}(i_1, x_1) \bar{\zeta}(i_2, x_2) p_N(i_1, x_1) p_N(i_2, x_2) + \cdots,
\end{aligned}
\end{equation}
where \(p_N(i, x)\) is the transition probability of the random walk.
Clearly, \(\mathbb{E}_{\mathbb{Q}}[\bar{\zeta}(i,x)] = 0\), and by Lemma \ref{lem:cgf_expansion},
\[
\frac{\mathbb{E}_{\mathbb{Q}}\left[e^{\beta_N(\bar{\omega}(i,x) + \bar{\omega}(i,y))}\right]}{\mathbb{E}_{\mathbb{Q}}\left[e^{\beta_N \bar{\omega}(i,x)}\right] \mathbb{E}_{\mathbb{Q}}\left[e^{\beta_N \bar{\omega}(i,x)}\right]} - 1 \approx \beta_N^2 \mathbb{E}_{\mathbb{Q}}[\bar{\xi}^2] \gamma(x-y).
\]
For the variance of the second term of RHS of \eqref{truncated_partition}, we have:
\[
\begin{aligned}
\mathbb{V}\text{ar}_{\mathbb{Q}} \left[ \sum_{1 \leq i \leq N, \, x \in \mathbb{Z}} \bar{\zeta}(i,x) p_N(i,x) \right] &= \sum_{i=1}^N \mathbb{V}\text{ar}_{\mathbb{Q}} \left[ \sum_{x \in \mathbb{Z}} \bar{\zeta}(i,x) p_N(i,x) \right] \\
&\approx \sum_{1 \leq i \leq N, \, x, y \in \mathbb{Z}} \beta_N^2 \mathbb{E}_{\mathbb{Q}}[\bar{\xi}^2] \gamma(x-y) p_N(i,x) p_N(i,y).
\end{aligned}
\]
Combining the local limit theorem, we have \(p_N(i,x) \approx \frac{2}{\sqrt{N}} \rho\left(\frac{i}{N}, \frac{x}{\sqrt{N}}\right)\), where \(\rho(t,x)\) is the standard Gaussian heat kernel. Then (ignoring constants),
\[
\begin{aligned}
&\frac{1}{\beta_N N^{\frac{H}{2}}} \left[ Z_N(\beta_N; \bar{\omega}) \exp\left(-N \bar{\lambda}_N(\beta_N)\right) - 1 \right] \\
&= \frac{1}{N^{\frac{H}{2}}} \sum_{1 \leq i \leq N, \, x \in \mathbb{Z}} \hat{\omega}(i,x) \left( \sqrt{N} p_N(i,x) \right) \\
&\quad + \frac{\beta_N N^{\frac{H}{2}}}{(N^{\frac{H}{2}})^2} \sum_{1 \leq i_1 < i_2 \leq N, \, x_1, x_2 \in \mathbb{Z}} \hat{\omega}(i_1,x_1) \hat{\omega}(i_2,x_2) \left( \sqrt{N} p_N(i_1,x_1) \right) \left( \sqrt{N} p_N(i_2,x_2) \right) + \cdots \\
&\approx \sqrt{2} \int_0^1 \int_{\mathbb{R}} \rho(t,x) W(dt, dx) \\
&\quad + \beta_N N^{\frac{H}{2}} (\sqrt{2})^2 \int_{0 < t_1 < t_2 < 1} \int_{\mathbb{R}^2} \rho(t_1,x_1) \rho(t_2 - t_1, x_2 - x_1) W(dt_1, dx_1) W(dt_2, dx_2) + \cdots,
\end{aligned}
\]
where \(W(dt, dx)\) is a space–time white noise satisfying
\[
\mathbb{E}_{\mathbb{Q}}[W(dt_1, dx_1) W(dt_2, dx_2)] = \delta(t_1 - t_2) |x_1 - x_2|^{2H-2}.
\]

Note that when \(\alpha \in \left(2, \frac{6}{3-2r}\right]\), our choice of \(\beta_N\) behaves asymptotically as \(\mathcal{O}\left(N^{-\frac{3}{2\alpha}}\right)\). Consequently, the scaling \(\beta_N N^{\frac{H}{2}} \to 0\) holds, and all terms except the first one converge to zero. In this case, the limit of the truncated normalized partition function is determined entirely by the first term, which clearly follows a centered normal distribution.

\subsection{Organization of the paper}
The rest of the paper is organized as follows. In Section 2, we compare the original and truncated partition functions. Section 3 is devoted to analyzing the limit behavior of the modified partition function. In Section 4, we complete the proof of Theorem \ref{thm:main_log_partition}. Section 5 provides the proof of Theorem \ref{thm:main_gaussian}. Finally, some technical lemmas are deferred to Appendix A and Appendix B. Throughout the paper, all asymptotic notation is understood as \(N\to\infty\). For positive deterministic sequences \(a_N,b_N\), we write \(a_N=O(b_N)\) if \(a_N\le Cb_N\) for all large N, where \(C\) does not depend on \(N\), and \(a_N=o(b_N)\) if \(a_N/b_N\to0\). We write \(a_N\asymp b_N\) if both \(a_N=O(b_N)\) and \(b_N=O(a_N)\) hold.

\section{Comparing the original and truncated partition functions}
\label{sec:truncation-comparison}

In this section we are devoted to prove that the partition function in the original environment
\(\omega\) can be replaced, at a suitable scale, by the partition function
in the truncated environment \(\bar\omega\).  The original environment is given
by \eqref{omega}, and the truncated environment is given by
\eqref{truncated_omega}. Our main tool is a far–near field analysis, which is of interest in its own right and has potential applications to the limit theory for nonlinear functionals of moving average innovations with power-law distributions, e.g., see \cite{LIU2024104237} for a related context.  Recall 
\(
        H=\frac32-r\in\left(\frac12,1\right),
\)
and the truncation level is given in \eqref{k_N} by
\begin{equation}\label{eq:kN-section2}
k_N=
\begin{cases}
\beta_N^{-1},&
\alpha>\dfrac{3}{H},\\[1.2ex]
\displaystyle
\beta_N^{-1}\frac{l(N^{3/2}(\log N)^\eta)}{l(N^{3/2})},&
2<\alpha\le\dfrac{3}{H},
\end{cases}
\end{equation}
where \(\eta\in(1,\alpha)\).  

For every \(\mathcal A\subseteq\mathcal S_0^N\), set
\[
Z_N(\beta_N;\omega;\mathcal A)
=
\sum_{S\in\mathcal A}
\exp\{ \beta_N H_N^\omega(S)\}\mathbb P_N(S),
\]
and define \(Z_N(\beta_N;\bar\omega;\mathcal A)\) in the same way.  Let
\[
        \Delta\xi(i,y)
        :=
        \xi(i,y)-\bar\xi(i,y)
        =
        \xi(i,y)\mathbb I_{\{|\xi(i,y)|>k_N\}}
        -
        \mathbb E_{\mathbb Q}\!\left[
        \xi(i,y)\mathbb I_{\{|\xi(i,y)|>k_N\}}
        \right],
\]
so that
\[
        \Delta\omega(i,x):=\omega(i,x)-\bar\omega(i,x)
        =
        \sum_{y\in\mathbb Z}\psi_{y-x}\Delta\xi(i,y).
\]
We also write
\[
        \mu_N:=\mathbb E_{\mathbb Q}
        [\xi\mathbbm 1_{\{|\xi|>k_N\}}],
        \qquad
        q_N:=\mathbb Q(|\xi|>k_N),
        \qquad
        e_N:=\mathbb E_{\mathbb Q}
        [|\xi|\mathbbm 1_{\{|\xi|>k_N\}}].
\]
By applying Karamata's theorem to the tail, and by Potter bounds, there are slowly varying functions
\(L_1,L_2\) such that,
\[
        w_N:=k_N^{2-\alpha}L_2(k_N).
\]
Then we have
\begin{equation}\label{eq:tail-basic-bounds-section2}
        q_N\le Ck_N^{-\alpha}L(k_N),\qquad
        e_N+|\mu_N|\le Ck_N^{1-\alpha}L_1(k_N),
\end{equation}
and
\begin{equation}\label{eq:tail-second-bounds-section2}
        \mathbb E_{\mathbb Q}[(\Delta\xi)^2]\le Cw_N,
        \qquad
        k_Ne_N\le Cw_N.
\end{equation}
where \(L_2\) is chosen so that
\(L_1(t)\le C L_2(t)\) for all large \(t\).  This is reasonable because the sum
of two slowly varying functions is still slowly varying.

These estimates will be used on the removed tail.  To handle the perturbation arising from the far field, we next establish two lemmas.

\subsection{Near and far decompositions}

 Notice that, for all large \(R\),
\begin{equation}\label{eq:kernel-sums-section2}
        \sum_{|z|\le R}\psi_z\le CR^{1-r}=CR^{H-1/2},
        \qquad
        \sum_{|z|>R}\psi_z^2\le CR^{1-2r}=CR^{2H-2}.
\end{equation}
Moreover, there is a fixed constant \(C_\psi<\infty\) such that, for all large
\(h\), all \(|x|\le h\), and all \(|y|>2h\),
\begin{equation}\label{eq:kernel-comparability-section2}
        C_\psi^{-1}|y|^{-r}\le \psi_{y-x}\le C_\psi|y|^{-r}.
\end{equation}

For \(h\ge1\), we define
\[
        \mathcal B(h)
        :=
        \{S\in\mathcal S_0^N:\ |S_i|\le h,\ 1\le i\le N\}.
\]
For a path \(S\in\mathcal B(h)\), we decompose
\[
        \Delta\omega(i,S_i)
        =
        \Delta\omega^{\mathrm{near},h}(i,S_i)
        +
        \Delta\omega^{\mathrm{far},h}(i,S_i),
\]
where
\[
        \Delta\omega^{\mathrm{near},h}(i,x)
        :=
        \sum_{|y|\le2h}\psi_{y-x}\Delta\xi(i,y),
        \qquad
        \Delta\omega^{\mathrm{far},h}(i,x)
        :=
        \sum_{|y|>2h}\psi_{y-x}\Delta\xi(i,y).
\]
The far-field series is understood as an \(L^2(\mathbb Q)\)-limit of centered
independent variables. 

The near-field contribution of the decomposition satisfies
\begin{equation}\label{eq:near-path-bound-section2}
\begin{aligned}
\left|
\sum_{i=1}^N\Delta\omega^{\mathrm{near},h}(i,S_i)
\right|
&\le
C\sum_{i=1}^N\sum_{|y|\le2h}
|\xi(i,y)|\mathbb I_{\{|\xi(i,y)|>k_N\}}
        +CN|\mu_N|h^{1-r}                                      \\
&=:T_h .
\end{aligned}
\end{equation}
Consequently, by \eqref{eq:tail-basic-bounds-section2},
\begin{equation}\label{eq:near-first-moment-section2}
        \mathbb E_{\mathbb Q}[T_h]\le CNhe_N.
\end{equation}
In addition, we define the event
\begin{equation}\label{$A_h$}
A_h:=\{\exists\,1\le i\le N,\ |y|\le2h:\ |\xi(i,y)|>k_N\}.
\end{equation}
Then by the union bound,
\[
\mathbb Q(A_h)
\le
\sum_{i=1}^N\sum_{|y|\le2h}
\mathbb Q(|\xi(i,y)|>k_N)
=
N(4h+1)q_N
\le CNh q_N .
\]
Moreover, on \(A_h^c\), for all \(1\le i\le N\) and \(|y|\le2h\),
\[
\Delta\xi(i,y)=-\mu_N .
\]
Hence, for \(S\in\mathcal B(h)\),
\begin{equation}\label{eq:near-centered-only-section2}
\left|
\sum_{i=1}^N\Delta\omega^{\mathrm{near},h}(i,S_i)
\right|
\le
|\mu_N|
\sum_{i=1}^N\sum_{|y|\le2h}\psi_{y-S_i}  
\le
CN|\mu_N|h^{1-r},
\end{equation}
where the last inequality follows from \(|S_i|\le h\) and
\(\sum_{|z|\le3h}\psi_z\le Ch^{1-r}\).

\subsection{A one-site exponential estimate}

The far-field estimate will be performed under the product measure tilted by
\(e^{\beta_N H_N^{\bar\omega}(S)}\).  At a single site, this tilt has the
form \(d\mathbb Q_a=e^{a\bar\xi}\mathbb E[e^{a\bar\xi}]^{-1}d\mathbb Q\).

\begin{lemma}\label{lem:one-site-clipped-section2}
Fix \(C_0<\infty\).  There exists \(C<\infty\), depending only on \(C_0\) and
the tail bounds, such that the following holds for all large \(N\).  Let
\(a,u\in\mathbb R\) and \(b>0\) satisfy
\[
        C_0^{-1}|u|\le b\le C_0|u|,
        \qquad
        |a|\le C_0|u|,
        \qquad
        2bk_N\le1 .
\]
Define
\[
        \Delta^{(b)}
        :=
        \xi\mathbbm{1}_{\{|\xi|>k_N,\ 2b|\xi|\le1\}}
        -
        \mu_N,
        \qquad
        d\mathbb Q_a
        :=
        \frac{e^{a\bar\xi}}
        {\mathbb E_{\mathbb Q}[e^{a\bar\xi}]}\,d\mathbb Q .
\]
Then
\[
        \left|
        \mathbb E_{\mathbb Q_a}
        [e^{u\Delta^{(b)}}]-1
        \right|
        \le Cu^2w_N,
\]
and
\[
        \mathbb E_{\mathbb Q_a}
        \left[
        (e^{u\Delta^{(b)}}-1)^2
        \right]
        \le Cu^2w_N,
        \qquad
        \left|
        \mathbb E_{\mathbb Q_a}
        [e^{2u\Delta^{(b)}}]-1
        \right|
        \le Cu^2w_N .
\]
\end{lemma}

\begin{proof}
Throughout the proof \(C\) may change from line to line, but it depends only on \(C_0\) and the tail bounds. The case \(u=0\) being trivial, we assume \(u\ne0\).  Since
\(2bk_N\le1\), \(1/(2b)\) is at least \(k_N\). Hence
\[
\begin{aligned}
\mathbb E_{\mathbb Q}[\Delta^{(b)}]
&=
\mathbb E_{\mathbb Q}
[\xi\mathbbm 1_{\{|\xi|>k_N,\ 2b|\xi|\le1\}}]
-
\mathbb E_{\mathbb Q}
[\xi\mathbbm 1_{\{|\xi|>k_N\}}]                         \\
&=
-
\mathbb E_{\mathbb Q}
[\xi\mathbbm 1_{\{|\xi|>k_N,\ 2b|\xi|>1\}}].
\end{aligned}
\]
Therefore
\[
        |\mathbb E_{\mathbb Q}[\Delta^{(b)}]|
        \le Cb^{\alpha-1}L_1(1/b).
\]
Since \(b\asymp |u|\) and \(1/b\ge k_N\), the Potter bounds for the regularly
varying function \(x^{1-\alpha}L_1(x)\) give
\[
        |u|\,|\mathbb E_{\mathbb Q}[\Delta^{(b)}]|
        \le
        C|u|^\alpha L_1(1/|u|)
        \le
        Cu^2k_N^{2-\alpha}L_2(k_N)
        =
        Cu^2w_N.
\]
Moreover, using \((a-b)^2\le2a^2+2b^2\) and
\(|\mu_N|\le e_N\), we have
\[
\begin{aligned}
        \mathbb E_{\mathbb Q}[(\Delta^{(b)})^2]
        &\le
        2\mathbb E_{\mathbb Q}\!\left[
        \xi^2\mathbbm 1_{\{|\xi|>k_N,\ 2b|\xi|\le1\}}
        \right]
        +2\mu_N^2                                      \\
        &\le
        2\mathbb E_{\mathbb Q}\!\left[
        \xi^2\mathbbm 1_{\{|\xi|>k_N\}}
        \right]
        +2e_N^2
        \le C k_N^{2-\alpha}L_2(k_N)
        =
        Cw_N.
\end{aligned}
\]
Because \(|a|\le C_0|u|\asymp b\) and \(2bk_N\le1\), while
\(|\bar\xi|\le Ck_N\), the density \(d\mathbb Q_a/d\mathbb Q\) is uniformly
bounded above and below.  Thus
\[
        \mathbb E_{\mathbb Q_a}[(\Delta^{(b)})^2]\le Cw_N .
\]
Then, since
\[
        \mathbb E_{\mathbb Q_a}[\Delta^{(b)}]
        =
        \frac{
        \mathbb E_{\mathbb Q}[e^{a\bar\xi}\Delta^{(b)}]}
        {\mathbb E_{\mathbb Q}[e^{a\bar\xi}]},
\]
the uniform bound on the density gives
\[
\left|
\mathbb E_{\mathbb Q_a}[\Delta^{(b)}]
\right|
\le
C\left|
\mathbb E_{\mathbb Q}[e^{a\bar\xi}\Delta^{(b)}]
\right|.
\]
Writing \(e^{a\bar\xi}=1+(e^{a\bar\xi}-1)\), and using
\(|e^{a\bar\xi}-1|\le C|a|\,|\bar\xi|\) since \(|a\bar\xi|\le C'\), we obtain
\[
\begin{aligned}
\left|
\mathbb E_{\mathbb Q_a}[\Delta^{(b)}]
\right|
&\le
C\left(
|\mathbb E_{\mathbb Q}[\Delta^{(b)}]|
+
|a|\,
\mathbb E_{\mathbb Q}[|\bar\xi|\,|\Delta^{(b)}|]
\right).
\end{aligned}
\]
The first term is bounded by \(Cb^{\alpha-1}L_1(1/b)\).  For the second one,
\(|\bar\xi|\le Ck_N\), \(|a|\le C|u|\), and
\[
        \mathbb E_{\mathbb Q}[|\Delta^{(b)}|]
        \le
        \mathbb E_{\mathbb Q}[|\xi|\mathbb I_{\{|\xi|>k_N\}}]
        +|\mu_N|
        \le 2e_N .
\]
Therefore
\[
\left|
\mathbb E_{\mathbb Q_a}[\Delta^{(b)}]
\right|
\le
C\left(
b^{\alpha-1}L_1(1/b)
+
|u|\,k_Ne_N
\right).
\]
Multiplying by \(|u|\), using \(b\asymp |u|\) and then 
\eqref{eq:tail-second-bounds-section2}, we obtain
\[
        |u|\,
        \left|
        \mathbb E_{\mathbb Q_a}[\Delta^{(b)}]
        \right|
        \le Cu^2w_N .
\]
Finally \(u\Delta^{(b)}\) is uniformly bounded.  Indeed, on the event \(\{2b|\xi|\le1\}\), and since \(b\asymp |u|\),
\[
        |u|\,|\xi|\mathbbm 1_{\{|\xi|>k_N,\ 2b|\xi|\le1\}}\le C .
\]
Moreover \(2bk_N\le1\) and \(b\asymp |u|\) imply \(|u|\le C/k_N\), and
therefore
\[
        |u\mu_N|\le |u|e_N
        \le Ck_N^{-1}e_N
        \le Ck_N^{-\alpha}L_1(k_N)\le C .
\]
Hence \(|u\Delta^{(b)}|\le C\).  For \(|x|\le C\),
\[
        e^x-1=x+x^2\int_0^1(1-t)e^{tx}\,dt ,
\]
so that \(|e^x-1-x|\le Cx^2\).  Applying this with
\(x=u\Delta^{(b)}\) gives
\[
\left|
\mathbb E_{\mathbb Q_a}[e^{u\Delta^{(b)}}]-1
\right|
\le
|u|\,
\left|
\mathbb E_{\mathbb Q_a}[\Delta^{(b)}]
\right|
+
Cu^2\mathbb E_{\mathbb Q_a}[(\Delta^{(b)})^2]
\le Cu^2w_N .
\]
On the same bounded interval, \(|e^x-1|\le C|x|\).  Thus
\[
        \mathbb E_{\mathbb Q_a}
        [(e^{u\Delta^{(b)}}-1)^2]
        \le
        Cu^2\mathbb E_{\mathbb Q_a}[(\Delta^{(b)})^2]
        \le Cu^2w_N,
\]
and the estimate with \(e^{2u\Delta^{(b)}}\) is identical.
\end{proof}

\subsection{Far-field product comparison}
For \(h\ge1\), set
\[
        \Psi_h(y):=\sup_{|x|\le h}\psi_{y-x},
        \qquad |y|>2h,
\]
and define the far-field bad event
\begin{equation}\label{$D_h$}
        D_h
        :=
        \left\{
        (i,y)\in [1,N]\times (-\infty,-2h)\cup(2h,\infty):
        2\beta_N|\xi(i,y)|\Psi_h(y)>1
        \right\}.
\end{equation}
The event \(D_h\) is estimated by a union bound.  
\[
        \mathbb Q(D_h)
        \le
        N\sum_{|y|>2h}
        \mathbb Q\left(
        |\xi|>\frac{2}{\beta_N\Psi_h(y)}
        \right).
\]
By the Potter bounds for the slowly varying function, for \(\rho>0\) small that \(r(\alpha-\rho)>1\) and \(L(t)\le C_\rho t^\rho\),
\[
        \mathbb Q(|\xi|>t)\le C t^{-\alpha}L(t) \le C_\rho t^{-(\alpha-\rho)} .
\]
Using this with \(t=2/(\beta_N\Psi_h(y))\), and then using
\(\Psi_h(y)\le C|y|^{-r}\), gives
\begin{equation}\label{eq:Dh-bound-section2}
\begin{aligned}
\mathbb Q(D_h)
&\le
C_\rho N\beta_N^{\alpha-\rho}
\sum_{|y|>2h}\Psi_h(y)^{\alpha-\rho}                       \\
&\le
C_\rho N\beta_N^{\alpha-\rho}
\sum_{|y|>2h}|y|^{-r(\alpha-\rho)}
\le
C_\rho N\beta_N^{\alpha-\rho}h^{1-r(\alpha-\rho)} .
\end{aligned}
\end{equation}

We shall use the abbreviation
\begin{equation}\label{eq:dNh-section2}
        \mathfrak d_N(h):=\beta_N^2Nw_Nh^{2H-2}.
\end{equation}

\begin{lemma}\label{lem:far-product-section2}
Let \(h=h_N\) be a positive sequence with \(h_N\to\infty\); Assume that $ \mathfrak d_N(h)\le1 $ and, for all large \(N\),
\[
        2\beta_N k_N\Psi_h(y)\le1,
        \qquad(|y|>2h).
\]
Then, we have, uniformly over all paths \(S\in\mathcal B(h)\),
\begin{equation}\label{eq:far-L2-section2}
        \frac{
        \mathbb E_{\mathbb Q}\!\left[
        e^{\beta_NH_N^{\bar\omega}(S)}
        \left(
        e^{\beta_N\Delta H_N^{\mathrm{far},h}(S)}-1
        \right)^2;
        D_h^c
        \right]
        }
        {
        \mathbb E_{\mathbb Q}
        [e^{\beta_NH_N^{\bar\omega}(S)}]
        }
        \le C\mathfrak d_N(h),
\end{equation}
\begin{equation}\label{eq:far-L1-section2}
        \frac{
        \mathbb E_{\mathbb Q}\!\left[
        e^{\beta_NH_N^{\bar\omega}(S)}
        \left|
        e^{\beta_N\Delta H_N^{\mathrm{far},h}(S)}-1
        \right|;
        D_h^c
        \right]
        }
        {
        \mathbb E_{\mathbb Q}
        [e^{\beta_NH_N^{\bar\omega}(S)}]
        }
        \le C\sqrt{\mathfrak d_N(h)},
\end{equation}
and
\begin{equation}\label{eq:far-exp-section2}
        \frac{
        \mathbb E_{\mathbb Q}\!\left[
        e^{\beta_NH_N^{\bar\omega}(S)}
        e^{\beta_N\Delta H_N^{\mathrm{far},h}(S)};
        D_h^c
        \right]
        }
        {
        \mathbb E_{\mathbb Q}
        [e^{\beta_NH_N^{\bar\omega}(S)}]
        }
        \le 1+C\mathfrak d_N(h).
\end{equation}
Consequently, for every \(A\subseteq\mathcal B(h)\),
\begin{equation}\label{eq:far-partition-L1-section2}
\begin{aligned}
&\mathbb E_{\mathbb Q}\left[
\left|
Z_N(\beta_N;\bar\omega+\Delta\omega^{\mathrm{far},h};A)
-
Z_N(\beta_N;\bar\omega;A)
\right|;
D_h^c
\right]                                                      \\
&\qquad\le
C\sqrt{\mathfrak d_N(h)}\,
\mathbb E_{\mathbb Q}[Z_N(\beta_N;\bar\omega;A)]
\end{aligned}
\end{equation}
and
\begin{equation}\label{eq:far-partition-exp-section2}
        \mathbb E_{\mathbb Q}\left[
        Z_N(\beta_N;\bar\omega+\Delta\omega^{\mathrm{far},h};A);
        D_h^c
        \right]
        \le
        (1+C\mathfrak d_N(h))\,
        \mathbb E_{\mathbb Q}[Z_N(\beta_N;\bar\omega;A)] .
\end{equation}
\end{lemma}

\begin{proof}
Fix \(S\in\mathcal B(h)\).  The ratios in
\eqref{eq:far-L2-section2}--\eqref{eq:far-exp-section2} are expectations under
the probability measure
\[
        d\mathbb Q^S
        :=
        \frac{e^{\beta_NH_N^{\bar\omega}(S)}}
        {\mathbb E_{\mathbb Q}[e^{\beta_NH_N^{\bar\omega}(S)}]}
        \,d\mathbb Q .
\]
Indeed,
\[
        H_N^{\bar\omega}(S)
        =
        \sum_{i=1}^N\sum_{y\in\mathbb Z}
        \psi_{y-S_i}\bar\xi(i,y),
\]
and therefore the density factorizes as
\[
        e^{\beta_NH_N^{\bar\omega}(S)}
        =
        \prod_{i=1}^N\prod_{y\in\mathbb Z}
        \exp\{\beta_N\psi_{y-S_i}\bar\xi(i,y)\}.
\]
Since variables \(\xi(i,y)\) are independent under \(\mathbb Q\), they
remain independent under \(\mathbb Q^S\).  The marginal law at
\((i,y)\) is the one-site tilted law of Lemma
\ref{lem:one-site-clipped-section2} with tilt parameter
\[
        a_{i,y}:=\beta_N\psi_{y-S_i}.
\]

For the far-field contribution, the 
coefficient at \((i,y)\) is
\[
        u_{i,y}:=\beta_N\psi_{y-S_i},
        \qquad |y|>2h.
\]
We choose the parameter
\[
        b_y:=\beta_N\Psi_h(y).
\]
Since \(S\in\mathcal B(h)\), \(|S_i|\le h\).  Hence, by
\eqref{eq:kernel-comparability-section2}, uniformly in \(S\),
\[
        |u_{i,y}|
        =
        \beta_N\psi_{y-S_i}
        \le
        \beta_N\Psi_h(y)
        =
        b_y,
        \qquad\text{and}\quad
        b_y
        \le
        \beta_N C_\psi |y|^{-r}
        \le
        C_\psi^2 |u_{i,y}|
        \qquad(|y|>2h).
\]
Thus the condition \(C_0^{-1}|u|\le b\le C_0|u|\) in Lemma
\ref{lem:one-site-clipped-section2} holds with, for instance,
\(C_0=C_\psi^2\). This constant is fixed by the kernel and does not depend on
\(N,h,y\), or on the path \(S\). Finally, since the assumption
\(
        2\beta_Nk_N\Psi_h(y)\le1
\)
is exactly \(2b_yk_N\le1\),  therefore Lemma
\ref{lem:one-site-clipped-section2} can be applied to each far-field coordinate
\((i,y)\), with \(a=a_{i,y}\), \(u=u_{i,y}\), and \(b=b_y\).
Writing
\[
        \Delta_{i,y}^{(b_y)}
        :=
        \xi(i,y)\mathbb I_{\{|\xi(i,y)|>k_N,\ 2b_y|\xi(i,y)|\le1\}}
        -\mu_N,
\]
Then, by Lemma \ref{lem:one-site-clipped-section2}, we have,
\[
        \left|
        \mathbb E_{\mathbb Q^S}
        \left[e^{u_{i,y}\Delta_{i,y}^{(b_y)}}\right]-1
        \right|
        \le
        Cu_{i,y}^2w_N,
\]
and 
\[
        \left|
        \mathbb E_{\mathbb Q^S}
        \left[e^{2u_{i,y}\Delta_{i,y}^{(b_y)}}\right]-1
        \right|
        +
        \mathbb E_{\mathbb Q^S}
        \left[
        \left(e^{u_{i,y}\Delta_{i,y}^{(b_y)}}-1\right)^2
        \right]
        \le
        Cu_{i,y}^2w_N .
\]

For $R\geq 2h$, let
\[
        \Lambda_R
        :=
        \{(i,y):1\le i\le N,\ 2h<|y|\le R\},
\]
and write
\[
        X_{i,y}
        :=
        \exp\left\{
        u_{i,y}\Delta_{i,y}^{(b_y)}
        \right\}.
\]
Then, the preceding estimates give
\begin{equation}\label{eq:X-one-site-section2}
        \left|\mathbb E_{\mathbb Q^S}[X_{i,y}]-1\right|
        +
        \left|\mathbb E_{\mathbb Q^S}[X_{i,y}^2]-1\right|
        \le
        Cu_{i,y}^2w_N .
\end{equation}
 Our goal is to
control the difference, uniformly in \(S\in\mathcal{B}(h)\), of the finite product
\[
        P_R(S):=\prod_{(i,y)\in\Lambda_R}X_{i,y}
        =
        \exp\left\{
        \sum_{(i,y)\in\Lambda_R}
        u_{i,y}\Delta_{i,y}^{(b_y)}
        \right\},
\]
which is the \(R\)-truncated far-field exponential, with 1 in  \(L^2(\mathbb Q^S)\).

By \eqref{eq:kernel-sums-section2}, since \(S\in\mathcal B(h)\),
\[
\begin{aligned}
        \sum_{(i,y)\in\Lambda_R}u_{i,y}^2w_N
        &\le
        \beta_N^2w_N
        \sum_{i=1}^N\sum_{|y|>2h}\psi_{y-S_i}^2              \\
        &\le
        C\beta_N^2Nw_N\sum_{|z|>h}\psi_z^2\\
        &\le
        C\beta_N^2Nw_Nh^{2H-2}
        =
        C\mathfrak d_N(h).
\end{aligned}
\]
Since \(\mathfrak d_N(h)\le1\), the right-hand side is bounded by a fixed
constant.  Using the independence under
\(\mathbb Q^S\), \eqref{eq:X-one-site-section2} gives
\[
\begin{aligned}
\mathbb E_{\mathbb Q^S}[(P_R(S)-1)^2]
&=
\prod_{(i,y)\in\Lambda_R}\mathbb E_{\mathbb Q^S}[X_{i,y}^2]
-2\prod_{(i,y)\in\Lambda_R}\mathbb E_{\mathbb Q^S}[X_{i,y}]
+1                                                       \\
&\le
\left|
\prod_{(i,y)\in\Lambda_R}\mathbb E_{\mathbb Q^S}[X_{i,y}^2]-1
\right|
+2\left|
\prod_{(i,y)\in\Lambda_R}\mathbb E_{\mathbb Q^S}[X_{i,y}]-1
\right|                                                  \\
&\le
C\sum_{(i,y)\in\Lambda_R}u_{i,y}^2w_N
\le
C\mathfrak d_N(h),
\end{aligned}
\]
by the product-difference estimate, see Durrett \cite[Lemma 3.4.3]{Durrett_2019}.
In a similar way, we have
\[
        \mathbb E_{\mathbb Q^S}[P_R(S)]
        =
        \prod_{(i,y)\in\Lambda_R}\mathbb E_{\mathbb Q^S}[X_{i,y}]
        \le 1+C\mathfrak d_N(h),
\]
and
\[
        \mathbb E_{\mathbb Q^S}[P_R(S)^2]
        =
        \prod_{(i,y)\in\Lambda_R}\mathbb E_{\mathbb Q^S}[X_{i,y}^2]
        \le 1+C\mathfrak d_N(h).
\]

The product \(P_R(S)\) converges in \(L^2(\mathbb Q^S)\) as \(R\to\infty\).
Indeed, for \(R'>R\), set
\[
        P_{R,R'}(S)
        :=
        \prod_{(i,y)\in\Lambda_{R'}\setminus\Lambda_R}X_{i,y}.
\]
Applying the preceding \(L^2\) estimate to this tail product gives
\[
        \mathbb E_{\mathbb Q^S}[(P_{R,R'}(S)-1)^2]
        \le
        Cw_N\sum_{i=1}^N\sum_{R<|y|\le R'}
        \beta_N^2\psi_{y-S_i}^2 .
\]
The right-hand side tends to \(0\) as \(R,R'\to\infty\), because
\[
        \sum_{i=1}^N\sum_{|y|>2h}
        \beta_N^2\psi_{y-S_i}^2w_N<\infty .
\]
Moreover, \(P_R(S)\) and \(P_{R,R'}(S)\) are independent under
\(\mathbb Q^S\), and the preceding bound on
\(\mathbb E_{\mathbb Q^S}[P_R(S)^2]\) is uniform in \(R\).  Hence
\[
\begin{aligned}
\mathbb E_{\mathbb Q^S}[(P_{R'}(S)-P_R(S))^2]
&=
\mathbb E_{\mathbb Q^S}[P_R(S)^2]\,
\mathbb E_{\mathbb Q^S}[(P_{R,R'}(S)-1)^2]
\longrightarrow0 .
\end{aligned}
\]
Thus \(P_R(S)\) converges in \(L^2(\mathbb Q^S)\); denote the limit by
\(P_\infty(S)\).  Passing to the limit in the two estimates above gives
\begin{equation}\label{eq:Pinfty-section2}
        \mathbb E_{\mathbb Q^S}[(P_\infty(S)-1)^2]\le C\mathfrak d_N(h),
        \qquad
        \mathbb E_{\mathbb Q^S}[P_\infty(S)]\le1+C\mathfrak d_N(h).
\end{equation}

We now connect this product estimate with the quantities in the statement.  On
\(D_h^c\), for every \((i,y)\) with \(|y|>2h\),
\[
        2b_y|\xi(i,y)|\le1.
\]
Therefore, on \(D_h^c\),
\[
        \Delta_{i,y}^{(b_y)}
        =
        \xi(i,y)\mathbb I_{\{|\xi(i,y)|>k_N\}}-\mu_N
        =
        \Delta\xi(i,y),
\]
and hence
\[
        P_\infty(S)
        =
        \exp\left\{
        \beta_N\Delta H_N^{\mathrm{far},h}(S)
        \right\}
        \qquad\text{on }D_h^c.
\]
Since the ratios in \eqref{eq:far-L2-section2}--\eqref{eq:far-exp-section2} are
\(\mathbb Q^S\)-expectations, \eqref{eq:Pinfty-section2} yields
\[
\begin{aligned}
&\frac{
\mathbb E_{\mathbb Q}\!\left[
e^{\beta_NH_N^{\bar\omega}(S)}
\left(e^{\beta_N\Delta H_N^{\mathrm{far},h}(S)}-1\right)^2;
D_h^c
\right]}
{\mathbb E_{\mathbb Q}[e^{\beta_NH_N^{\bar\omega}(S)}]}     =
\mathbb E_{\mathbb Q^S}[(P_\infty(S)-1)^2;D_h^c]
\le C\mathfrak d_N(h),
\end{aligned}
\]
which is \eqref{eq:far-L2-section2}.  Similarly,
\[
\frac{
\mathbb E_{\mathbb Q}\!\left[
e^{\beta_NH_N^{\bar\omega}(S)}
e^{\beta_N\Delta H_N^{\mathrm{far},h}(S)};
D_h^c
\right]}
{\mathbb E_{\mathbb Q}[e^{\beta_NH_N^{\bar\omega}(S)}]}
=
\mathbb E_{\mathbb Q^S}[P_\infty(S);D_h^c]
\le
1+C\mathfrak d_N(h),
\]
which is \eqref{eq:far-exp-section2}.  Finally,
\eqref{eq:far-L1-section2} follows from \eqref{eq:far-L2-section2} by
Cauchy's inequality under \(\mathbb Q^S\).

It remains to pass from the fixed-path estimates to partition functions.  By
Tonelli's theorem and the definition of \(Z_N\),
\[
\begin{aligned}
&\mathbb E_{\mathbb Q}\left[
\left|
Z_N(\beta_N;\bar\omega+\Delta\omega^{\mathrm{far},h};A)
-
Z_N(\beta_N;\bar\omega;A)
\right|;
D_h^c
\right]                                                   \\
&\quad\le
\int_A
\mathbb E_{\mathbb Q}\left[
e^{\beta_NH_N^{\bar\omega}(S)}
\left|
e^{\beta_N\Delta H_N^{\mathrm{far},h}(S)}-1
\right|;
D_h^c
\right]\,\mathbb P_N(dS).
\end{aligned}
\]
Applying \eqref{eq:far-L1-section2} to the integrand gives
\[
\begin{aligned}
&\mathbb E_{\mathbb Q}\left[
\left|
Z_N(\beta_N;\bar\omega+\Delta\omega^{\mathrm{far},h};A)
-
Z_N(\beta_N;\bar\omega;A)
\right|;
D_h^c
\right]                                                   \\
&\quad\le
C\sqrt{\mathfrak d_N(h)}
\int_A
\mathbb E_{\mathbb Q}
[e^{\beta_NH_N^{\bar\omega}(S)}]\,\mathbb P_N(dS)          \\
&\quad=
C\sqrt{\mathfrak d_N(h)}
\mathbb E_{\mathbb Q}[Z_N(\beta_N;\bar\omega;A)],
\end{aligned}
\]
which is \eqref{eq:far-partition-L1-section2}.  Similarly,
\[
\begin{aligned}
&\mathbb E_{\mathbb Q}\left[
Z_N(\beta_N;\bar\omega+\Delta\omega^{\mathrm{far},h};A);
D_h^c
\right]                                                   \\
&\quad=
\int_A
\mathbb E_{\mathbb Q}\left[
e^{\beta_NH_N^{\bar\omega}(S)}
e^{\beta_N\Delta H_N^{\mathrm{far},h}(S)};
D_h^c
\right]\,\mathbb P_N(dS)                                  \\
&\quad\le
(1+C\mathfrak d_N(h))
\int_A
\mathbb E_{\mathbb Q}
[e^{\beta_NH_N^{\bar\omega}(S)}]\,\mathbb P_N(dS)          \\
&\quad=
(1+C\mathfrak d_N(h))
\mathbb E_{\mathbb Q}[Z_N(\beta_N;\bar\omega;A)],
\end{aligned}
\]
which is \eqref{eq:far-partition-exp-section2}.
\end{proof}

\subsection{The multiscale comparison}

For $N, m\in \mathbb N_0$ and a sequence $\{h_j\}_{j=0}^m$ with 
\[
        0=h_0<h_1<\cdots<h_m,
        \qquad
        h_{m-1}<N\le h_m,
\]
define a sequence of boxes
\[
        \mathcal B_j:=\mathcal B(h_j),
        \qquad
        \mathcal A_j:=\mathcal B_j\setminus\mathcal B_{j-1}.
\]
Since every \(N\)-step nearest-neighbor path satisfies \(|S_i|\le N\), the sets
\(\mathcal B_1,\mathcal A_2,\ldots,\mathcal A_m\) cover the whole path space.
By Levy's inequality and Hoeffding's inequality,
\begin{equation}\label{eq:rw-entropy-section2}
        \mathbb P_N(\mathcal B_j^c)\le4e^{-I_j/2},
        \qquad
        \mathbb P_N(\mathcal A_j)\le4e^{-I_{j-1}/2}
        \quad(j\ge2),
\end{equation}
where $
        I_j:=\frac{h_j^2}{N}, j=2,\dots,m.$ 

Recall
\[
        R_N:=
        \frac{Z_N(\beta_N;\omega)}
        {Z_N(\beta_N;\bar\omega)} 
\]
and 
\[
        G_N(\delta)
        :=
        \left\{
        Z_N(\beta_N;\bar\omega)
        \ge
        \delta\,
        \mathbb E_{\mathbb Q}[Z_N(\beta_N;\bar\omega)]
        \right\}, \quad \delta >0.
\]

Let
\[
        F_N:=
        \frac{
        \left|
        Z_N(\beta_N;\omega;\mathcal B_1)
        -
        Z_N(\beta_N;\bar\omega;\mathcal B_1)
        \right|}
        {Z_N(\beta_N;\bar\omega)},
\]
\[
        U_N:=
        \sum_{j=2}^m
        \frac{Z_N(\beta_N;\omega;\mathcal A_j)}
        {Z_N(\beta_N;\bar\omega)},
        \qquad
        L_N:=
        \sum_{j=2}^m
        \frac{Z_N(\beta_N;\bar\omega;\mathcal A_j)}
        {Z_N(\beta_N;\bar\omega)} .
\]
Since
\(\mathcal B_1,\mathcal A_2,\ldots,\mathcal A_m\) cover the whole path space,
\[
\begin{aligned}
R_N-1
&=
\frac{
Z_N(\beta_N;\omega;\mathcal B_1)
-
Z_N(\beta_N;\bar\omega;\mathcal B_1)}
{Z_N(\beta_N;\bar\omega)}                                      \\
&\quad+
\sum_{j=2}^m
\frac{
Z_N(\beta_N;\omega;\mathcal A_j)
-
Z_N(\beta_N;\bar\omega;\mathcal A_j)}
{Z_N(\beta_N;\bar\omega)}
\le F_N+U_N,
\end{aligned}
\]
because the terms \(-Z_N(\beta_N;\bar\omega;\mathcal A_j)\) are non-positive.
Similarly,
\[
\begin{aligned}
1-R_N
&=
\frac{
Z_N(\beta_N;\bar\omega;\mathcal B_1)
-
Z_N(\beta_N;\omega;\mathcal B_1)}
{Z_N(\beta_N;\bar\omega)}                                      \\
&\quad+
\sum_{j=2}^m
\frac{
Z_N(\beta_N;\bar\omega;\mathcal A_j)
-
Z_N(\beta_N;\omega;\mathcal A_j)}
{Z_N(\beta_N;\bar\omega)}
\le F_N+L_N,
\end{aligned}
\]
Hence
\begin{equation}\label{eq:ratio-reduction-section2}
        R_N-1\le F_N+U_N,
        \qquad
        1-R_N\le F_N+L_N .
\end{equation}

We next give two lemmas for later use. The precise choices of the scales \(h_j\) 
  and \(\tau_N\) will be
 specified till in the proof of Proposition \ref{truncated-error}.

\begin{lemma}\label{lem:first-block-section2}
Let \(\tau_N>0\).  Assume that, for all large \(N\), Lemma
\ref{lem:far-product-section2} can be applied at the scale \(h_1\).  Put
\[
        s_{1,N}:=C_{\mathrm{near}}\beta_NN|\mu_N|h_1^{1-r},
\]
where \(C_{\mathrm{near}}\) is fixed large enough that
\eqref{eq:near-centered-only-section2} implies
\[
        |\beta_N\Delta H_N^{\mathrm{near},h_1}(S)|
        \le s_{1,N},
        \qquad S\in\mathcal B_1,
\]
on \(A_{h_1}^c\).  When \(s_{1,N}\le1\), 
then, for every fixed \(\varepsilon,\delta>0\),
\begin{equation}\label{eq:first-block-prob-section2}
\begin{aligned}
\mathbb Q(F_N>\varepsilon\tau_N)
&\le
\mathbb Q(G_N(\delta)^c)
+
\mathbb Q(A_{h_1})
+
\mathbb Q(D_{h_1})                                     
+
\frac{C}{\delta\varepsilon\tau_N}
\left(s_{1,N}+\sqrt{\mathfrak d_N(h_1)}\right),
\end{aligned}
\end{equation}
where $A_h$ and $D_h$ are given in \eqref{$A_h$} and \eqref{$D_h$}, respectively.
\end{lemma}

\begin{proof}
For any set \(A\) of paths and any \(h\),
\[
\begin{aligned}
&Z_N(\beta_N;\omega;A)-Z_N(\beta_N;\bar\omega;A)             \\
&\quad =
\int_A
\exp\{\beta_NH_N^{\bar\omega}(S)\}
\left[
\exp\{\beta_N\Delta H_N^{\mathrm{near},h}(S)
      +\beta_N\Delta H_N^{\mathrm{far},h}(S)\}
-1
\right]\,\mathbb P_N(dS).
\end{aligned}
\]
This identity is the point at which the near--far decomposition enters the
comparison.

For the first block we take \(A=\mathcal B_1\) and \(h=h_1\).  On
\(A_{h_1}^c\), \eqref{eq:near-centered-only-section2} gives, uniformly in
\(S\in\mathcal B_1\),
\[
        |\beta_N\Delta H_N^{\mathrm{near},h_1}(S)|
        \le s_{1,N}.
\]
Hence, on \(A_{h_1}^c\cap D_{h_1}^c\),
write
\[
        X(S):=\beta_N\Delta H_N^{\mathrm{near},h_1}(S),
        \qquad
        Y(S):=\beta_N\Delta H_N^{\mathrm{far},h_1}(S).
\]
Then \(|X(S)|\le s_{1,N}\le1\).  Using the identity
\[
        e^{X+Y}-1=e^X(e^Y-1)+(e^X-1),
\]
we get
\[
\begin{aligned}
&\left|
\exp\{\beta_N\Delta H_N^{\mathrm{near},h_1}(S)
      +\beta_N\Delta H_N^{\mathrm{far},h_1}(S)\}
-1
\right|                                                     \\
&\qquad\le
e^{|X(S)|}
\left|e^{Y(S)}-1\right|
+
\left|e^{X(S)}-1\right|                                    \\
&\qquad\le
e^{s_{1,N}}
\left|e^{\beta_N\Delta H_N^{\mathrm{far},h_1}(S)}-1\right|
+
\left(e^{s_{1,N}}-1\right)\\
&\qquad  \le
C\left(
\left|e^{\beta_N\Delta H_N^{\mathrm{far},h_1}(S)}-1\right|
+s_{1,N}
\right).
\end{aligned}
\]
Multiplying by \(e^{\beta_NH_N^{\bar\omega}(S)}\) and integrating over
\(\mathcal B_1\), we obtain, on \(A_{h_1}^c\cap D_{h_1}^c\),
\[
\begin{aligned}
&\left|
Z_N(\beta_N;\omega;\mathcal B_1)
-
Z_N(\beta_N;\bar\omega;\mathcal B_1)
\right|                                                     \\
&\qquad\le
C\int_{\mathcal B_1}
e^{\beta_NH_N^{\bar\omega}(S)}
\left|e^{\beta_N\Delta H_N^{\mathrm{far},h_1}(S)}-1\right|
\mathbb P_N(dS)                                             \\
&\qquad\quad
+Cs_{1,N}
\int_{\mathcal B_1}
e^{\beta_NH_N^{\bar\omega}(S)}
\mathbb P_N(dS).
\end{aligned}
\]
After taking \(\mathbb Q\)-expectation with the indicator
\(A_{h_1}^c\cap D_{h_1}^c\), we have,
by \eqref{eq:far-L1-section2},
\[
\begin{aligned}
&\int_{\mathcal B_1}
\mathbb E_{\mathbb Q}\left[
e^{\beta_NH_N^{\bar\omega}(S)}
\left|e^{\beta_N\Delta H_N^{\mathrm{far},h_1}(S)}-1\right|;
D_{h_1}^c
\right]\mathbb P_N(dS)                                      \\
&\qquad\le
C\sqrt{\mathfrak d_N(h_1)}
\int_{\mathcal B_1}
\mathbb E_{\mathbb Q}
\left[e^{\beta_NH_N^{\bar\omega}(S)}\right]\mathbb P_N(dS)  \\
&\qquad=
C\sqrt{\mathfrak d_N(h_1)}
\mathbb E_{\mathbb Q}[Z_N(\beta_N;\bar\omega;\mathcal B_1)] .
\end{aligned}
\]
Hence
\begin{equation}\label{eq:first-block-expectation-section2}
\begin{aligned}
&\mathbb E_{\mathbb Q}\left[
\left|
Z_N(\beta_N;\omega;\mathcal B_1)
-
Z_N(\beta_N;\bar\omega;\mathcal B_1)
\right|;
A_{h_1}^c\cap D_{h_1}^c
\right]                                                      \\
&\qquad\le
C\left(s_{1,N}+\sqrt{\mathfrak d_N(h_1)}\right)
\mathbb E_{\mathbb Q}[Z_N(\beta_N;\bar\omega;\mathcal B_1)] .
\end{aligned}
\end{equation}
On \(G_N(\delta)\), the denominator in \(F_N\) is at least
\(\delta\mathbb E_{\mathbb Q}[Z_N(\beta_N;\bar\omega)]\).  Therefore
\[
\begin{aligned}
\mathbb Q(F_N>\varepsilon\tau_N)
&\le
\mathbb Q(G_N(\delta)^c)+\mathbb Q(A_{h_1})+\mathbb Q(D_{h_1}) \\
&\quad+
\mathbb Q\left(
\left|
Z_N(\beta_N;\omega;\mathcal B_1)
-
Z_N(\beta_N;\bar\omega;\mathcal B_1)
\right|
>
\delta\varepsilon\tau_N
\mathbb E_{\mathbb Q}[Z_N(\beta_N;\bar\omega)];
A_{h_1}^c\cap D_{h_1}^c
\right)                                                       \\
&\le
\mathbb Q(G_N(\delta)^c)+\mathbb Q(A_{h_1})+\mathbb Q(D_{h_1}) \\
&\quad+
\frac{
\mathbb E_{\mathbb Q}\left[
\left|
Z_N(\beta_N;\omega;\mathcal B_1)
-
Z_N(\beta_N;\bar\omega;\mathcal B_1)
\right|;
A_{h_1}^c\cap D_{h_1}^c
\right]}
{\delta\varepsilon\tau_N
\mathbb E_{\mathbb Q}[Z_N(\beta_N;\bar\omega)]}.
\end{aligned}
\]
Finally, using \eqref{eq:first-block-expectation-section2} and the identity \eqref{eq:stationarity-partition-section2},
\[
        \mathbb E_{\mathbb Q}[Z_N(\beta_N;\bar\omega;\mathcal B_1)]
        =
        \mathbb P_N(\mathcal B_1)
        \mathbb E_{\mathbb Q}[Z_N(\beta_N;\bar\omega)]
        \le
        \mathbb E_{\mathbb Q}[Z_N(\beta_N;\bar\omega)],
\]
which gives \eqref{eq:first-block-prob-section2}.
\end{proof}

\begin{lemma}\label{lem:outer-shell-section2}
Let \(\tau_N>0\).  Assume that, for all large \(N\), Lemma
\ref{lem:far-product-section2} can be applied at the scales
\(h_2,\ldots,h_m\).  Then, for every fixed \(\varepsilon,\delta>0\),
\begin{equation}\label{eq:upper-shell-section2}
\begin{aligned}
\mathbb Q(U_N>\varepsilon\tau_N)
&\le
\mathbb Q(G_N(\delta)^c)
+
\sum_{j=2}^m
\left[
C\frac{\beta_NNh_je_N}{I_{j-1}}
+
\mathbb Q(D_{h_j})
\right]                                                     \\
&\quad
+
\frac{C}{\delta\varepsilon\tau_N}
\sum_{j=2}^m
\exp\{-I_{j-1}/4+C\mathfrak d_N(h_j)\},
\end{aligned}
\end{equation}
and
\begin{equation}\label{eq:lower-shell-section2}
\begin{aligned}
\mathbb Q(L_N>\varepsilon\tau_N)
&\le
\mathbb Q(G_N(\delta)^c)
+
\frac{C}{\delta\varepsilon\tau_N}
\sum_{j=2}^m e^{-I_{j-1}/2}.
\end{aligned}
\end{equation}
\end{lemma}

\begin{proof}
Fix \(j\ge2\). If \(S\in\mathcal A_j\), then
\[
        e^{\beta_NH_N^\omega(S)}
        =
        e^{\beta_NH_N^{\bar\omega}(S)}
        e^{\beta_N\Delta H_N^{\mathrm{near},h_j}(S)}
        e^{\beta_N\Delta H_N^{\mathrm{far},h_j}(S)} .
\]
Recall $T_h$ in \eqref{eq:near-path-bound-section2}. Set
\[
        E_j:=\{\beta_NT_{h_j}\le I_{j-1}/4\}.
\]
On \(E_j\), the near-field bound \eqref{eq:near-path-bound-section2} gives
\[
        e^{\beta_N\Delta H_N^{\mathrm{near},h_j}(S)}
        \le e^{\beta_NT_{h_j}}\le e^{I_{j-1}/4},
        \qquad S\in\mathcal A_j .
\]
Therefore, on \(E_j\cap D_{h_j}^c\),
\[
        Z_N(\beta_N;\omega;\mathcal A_j)
        \le
        e^{I_{j-1}/4}
        Z_N(\beta_N;\bar\omega+\Delta\omega^{\mathrm{far},h_j};
        \mathcal A_j).
\]
By \eqref{eq:near-first-moment-section2} and Markov's inequality,
\[
        \mathbb Q(E_j^c)
        \le
        C\frac{\beta_NNh_je_N}{I_{j-1}}.
\]
Taking \(\mathbb Q\)-expectation and using 
\eqref{eq:far-partition-exp-section2}, we obtain
\[
\begin{aligned}
&\mathbb E_{\mathbb Q}\left[
Z_N(\beta_N;\omega;\mathcal A_j);
E_j\cap D_{h_j}^c
\right]                                                       \\
&\qquad\le
e^{I_{j-1}/4}
\mathbb E_{\mathbb Q}\left[
Z_N(\beta_N;\bar\omega+\Delta\omega^{\mathrm{far},h_j};
\mathcal A_j);
D_{h_j}^c
\right]                                                       \\
&\qquad\le
e^{I_{j-1}/4}
(1+C\mathfrak d_N(h_j))
\mathbb E_{\mathbb Q}[Z_N(\beta_N;\bar\omega;\mathcal A_j)] .
\end{aligned}
\]
By identity \eqref{eq:stationarity-partition-section2} and \eqref{eq:rw-entropy-section2},
\[
        \mathbb E_{\mathbb Q}
        [Z_N(\beta_N;\bar\omega;\mathcal A_j)]
        =
        \mathbb P_N(\mathcal A_j)
        \mathbb E_{\mathbb Q}[Z_N(\beta_N;\bar\omega)]
        \le
        4e^{-I_{j-1}/2}
        \mathbb E_{\mathbb Q}[Z_N(\beta_N;\bar\omega)] .
\]
Hence
\begin{equation}\label{eq:shell-original-expectation-section2}
\begin{aligned}
&\mathbb E_{\mathbb Q}\left[
Z_N(\beta_N;\omega;\mathcal A_j);
E_j\cap D_{h_j}^c
\right]                                                       \\
&\qquad\le
C\exp\{-I_{j-1}/4+C\mathfrak d_N(h_j)\}
\mathbb E_{\mathbb Q}[Z_N(\beta_N;\bar\omega)] .
\end{aligned}
\end{equation}
We now remove \(G_N(\delta)^c\), the near-field bad events \(E_j^c\), and the
far-field bad events \(D_{h_j}\). Let
\[
        \mathcal G:=
        G_N(\delta)\cap
        \bigcap_{j=2}^m E_j\cap
        \bigcap_{j=2}^m D_{h_j}^c .
\]
Then
\[
        \mathbb Q(U_N>\varepsilon\tau_N)
        \le
        \mathbb Q(\mathcal G^c)
        +
        \mathbb Q(U_N>\varepsilon\tau_N;\mathcal G),
\]
and, by the union bound and the estimate on \(E_j^c\),
\[
        \mathbb Q(\mathcal G^c)
        \le
        \mathbb Q(G_N(\delta)^c)
        +
        \sum_{j=2}^m
        \left[
        C\frac{\beta_NNh_je_N}{I_{j-1}}
        +
        \mathbb Q(D_{h_j})
        \right].
\]
On \(G_N(\delta)\),
\[
        Z_N(\beta_N;\bar\omega)
        \ge
        \delta\mathbb E_{\mathbb Q}[Z_N(\beta_N;\bar\omega)].
\]
Hence, on \(\mathcal G\),
\[
        U_N
        =
        \sum_{j=2}^m
        \frac{Z_N(\beta_N;\omega;\mathcal A_j)}
        {Z_N(\beta_N;\bar\omega)}
        \le
        \frac{
        \sum_{j=2}^m
        Z_N(\beta_N;\omega;\mathcal A_j)}
        {\delta\mathbb E_{\mathbb Q}
        [Z_N(\beta_N;\bar\omega)]}.
\]
Therefore Markov's inequality gives
\[
\begin{aligned}
&\mathbb Q(U_N>\varepsilon\tau_N;\mathcal G)                  \\
&\qquad\le
\frac{1}{
\delta\varepsilon\tau_N
\mathbb E_{\mathbb Q}[Z_N(\beta_N;\bar\omega)]}
\sum_{j=2}^m
\mathbb E_{\mathbb Q}\left[
Z_N(\beta_N;\omega;\mathcal A_j);
E_j\cap D_{h_j}^c
\right].
\end{aligned}
\]
Using \eqref{eq:shell-original-expectation-section2} in the last display gives
\[
\begin{aligned}
\mathbb Q(U_N>\varepsilon\tau_N)
&\le
\mathbb Q(G_N(\delta)^c)
+
\sum_{j=2}^m
\left[
C\frac{\beta_NNh_je_N}{I_{j-1}}
+
\mathbb Q(D_{h_j})
\right]                                                     \\
&\quad
+
\frac{C}{\delta\varepsilon\tau_N}
\sum_{j=2}^m
\exp\{-I_{j-1}/4+C\mathfrak d_N(h_j)\}.
\end{aligned}
\]
This proves \eqref{eq:upper-shell-section2}.

For \(L_N\), each
\(j\ge2\), by identity \eqref{eq:stationarity-partition-section2} and \eqref{eq:rw-entropy-section2} give
\begin{equation}\label{eq:shell-bar-expectation-section2}
        \mathbb E_{\mathbb Q}
        [Z_N(\beta_N;\bar\omega;\mathcal A_j)]
        =
        \mathbb P_N(\mathcal A_j)
        \mathbb E_{\mathbb Q}[Z_N(\beta_N;\bar\omega)]
        \le
        4e^{-I_{j-1}/2}
        \mathbb E_{\mathbb Q}[Z_N(\beta_N;\bar\omega)] .
\end{equation}
On \(G_N(\delta)\), Markov's inequality applied to the sum gives
\[
\begin{aligned}
\mathbb Q(L_N>\varepsilon\tau_N)
&\le
\mathbb Q(G_N(\delta)^c)
+
\frac{1}{\delta\varepsilon\tau_N}
\sum_{j=2}^m
\frac{
\mathbb E_{\mathbb Q}
[Z_N(\beta_N;\bar\omega;\mathcal A_j)]
}{
\mathbb E_{\mathbb Q}[Z_N(\beta_N;\bar\omega)]
}                                                          \\
&=
\mathbb Q(G_N(\delta)^c)
+
\frac{1}{\delta\varepsilon\tau_N}
\sum_{j=2}^m \mathbb P_N(\mathcal A_j).
\end{aligned}
\]
Equivalently, using \eqref{eq:shell-bar-expectation-section2} in the last
display gives
\eqref{eq:lower-shell-section2}.
\end{proof}

\begin{proposition}\label{truncated-error}
Assume conditions \((A)\) and \((B)\), and the positivity condition
\begin{equation}\label{eq:positivity-section2}
        \lim_{\delta\downarrow0}\limsup_{N\to\infty}
        \mathbb Q\left(
        Z_N(\beta_N;\bar\omega)
        <
        \delta\,
        \mathbb E_{\mathbb Q}[Z_N(\beta_N;\bar\omega)]
        \right)=0 .
\end{equation}
Let \(k_N\) be as in \eqref{eq:kN-section2}.  If \(\alpha>3/H\), assume
\[
        \sup_{N\ge1}\beta_NN^{H/2}<\infty .
\]
Put
\[
        \kappa:=\frac{\alpha H}{2}-\frac32>0
\]
and define
\begin{equation}
A_*(\alpha, H) := \sup_{1/2 < p < \min\{1,\, 1/2 + \kappa\}} \min\left\{
\frac12 + \kappa - p\left(H - \frac12\right),\,
\frac14 + \frac{\kappa}{2} + p(1 - H)
\right\}
\label{A_alpha_H}
\end{equation}
Then, for every \(0<a<A_*(\alpha,H)\),
\[
        N^a\left|
        \log Z_N(\beta_N;\omega)
        -
        \log Z_N(\beta_N;\bar\omega)
        \right|
        \xrightarrow{\mathbb Q}0 .
\]
If \(2<\alpha\le3/H\), assume
\[
        \beta_N=\beta\, l(N^{3/2})^{-1},
        \qquad \beta>0,
\]
and set
\[
        a_N:=\beta_NN^{H/2}.
\]
Then
\[
        \frac{
        \log Z_N(\beta_N;\omega)
        -
        \log Z_N(\beta_N;\bar\omega)}
        {a_N}
        \xrightarrow{\mathbb Q}0 .
\]
\end{proposition}

\begin{proof}
We first derive the upper bound using the lemma proved earlier.  Fix a positive sequence \(\tau_N\), and suppose that
the scales \(h_1,\ldots,h_m\) are such that Lemmas
\ref{lem:first-block-section2} and \ref{lem:outer-shell-section2} can be
applied.  For fixed \(\varepsilon,\delta>0\), set
\[
        \lambda_N(\varepsilon)
        :=
        \frac14\min\{1,\varepsilon\tau_N\}.
\]
Since
\[
        |\log x|\le2|x-1|,
        \quad\text{for}\,\, |x-1|\le\frac12,
\]
we have
\[
        \left\{
        |\log R_N|>\varepsilon\tau_N
        \right\}
        \subseteq
        \left\{
        |R_N-1|>2\lambda_N(\varepsilon)
        \right\}.
\]
Indeed, if \(|R_N-1|\le1/2\), then
\(|\log R_N|\le2|R_N-1|\); if \(|R_N-1|>1/2\), then the right-hand event holds
because \(2\lambda_N(\varepsilon)\le1/2\).  By
\eqref{eq:ratio-reduction-section2},
\[
        \{|R_N-1|>2\lambda_N(\varepsilon)\}
        \subseteq
        \{F_N>\lambda_N(\varepsilon)\}
        \cup
        \{U_N>\lambda_N(\varepsilon)\}
        \cup
        \{L_N>\lambda_N(\varepsilon)\}.
\]
Indeed, on the part where \(R_N\ge1\), the first inequality in
\eqref{eq:ratio-reduction-section2} gives
\(F_N+U_N>2\lambda_N(\varepsilon)\), hence either
\(F_N>\lambda_N(\varepsilon)\) or
\(U_N>\lambda_N(\varepsilon)\).  On the part where \(R_N<1\), the second
inequality gives \(F_N+L_N>2\lambda_N(\varepsilon)\), hence either
\(F_N>\lambda_N(\varepsilon)\) or \(L_N>\lambda_N(\varepsilon)\).
Applying Lemmas \ref{lem:first-block-section2} and
\ref{lem:outer-shell-section2} with the comparison scale in those lemmas
replaced by \(\lambda_N(\varepsilon)\), and with the lemma-parameter
\(\varepsilon\) there equal to \(1\), gives
\begin{equation}\label{eq:log-reduction-section2}
\begin{aligned}
&\mathbb Q\left(
        |\log Z_N(\beta_N;\omega)
        -
        \log Z_N(\beta_N;\bar\omega)|
        >
        \varepsilon\tau_N
        \right)                                             \\
&\quad\le
C\mathbb Q(G_N(\delta)^c)
+\mathbb Q(A_{h_1})
+\sum_{j=1}^m\mathbb Q(D_{h_j})
+C\sum_{j=2}^m\frac{\beta_NNh_je_N}{I_{j-1}}                 \\
&\qquad
+\frac{C}{\delta\lambda_N(\varepsilon)}
\left[
s_{1,N}
+\sqrt{\mathfrak d_N(h_1)}
+\sum_{j=2}^m
\exp\{-I_{j-1}/4+C\mathfrak d_N(h_j)\}
+\sum_{j=2}^m e^{-I_{j-1}/2}
\right].
\end{aligned}
\end{equation}
The term \(\mathbb Q(G_N(\delta)^c)\) is kept until the end and is removed by the
positivity assumption \eqref{eq:positivity-section2}; all other terms in
\eqref{eq:log-reduction-section2} will be shown to vanish for the chosen
scales.

\medskip
\noindent\textbf{The regime \(\alpha>3/H\).}

Assume \(\sup_N\beta_NN^{H/2}<\infty\) and \(k_N=\beta_N^{-1}\).
Notice 
\(
        \kappa=\frac{\alpha H}{2}-\frac32>0.
\)
Choose
\[
        p\in\left(\frac12,\min\left\{1,\frac12+\kappa\right\}\right)
\]
and choose \(a>0\) such that
\begin{equation}\label{eq:a-choice-light-section2}
        a<
        \min\left\{
        \frac12+\kappa-p\left(H-\frac12\right),
        \frac14+\frac{\kappa}{2}+p(1-H)
        \right\}.
\end{equation}
Set \(\tau_N=N^{-a}\).  Define
\[
        p_1=p,\qquad
        p_j=2p_{j-1}-\frac12+\frac{\kappa}{2},\quad j\ge2,
\]
and
\[
        h_j=\lfloor N^{p_j}\rfloor .
\]
We stop at the first \(m\) such that \(p_m\ge1\), replacing \(h_m\) by \(N\)
if necessary.  Then \(m\) is fixed constant independent of \(N\).

We first check that these scales are admissible in Lemmas
\ref{lem:first-block-section2} and \ref{lem:outer-shell-section2}.  Since
\(h_1\to\infty\), all scales \(h_j\) tend to infinity.  Moreover
\(\beta_Nk_N=1\), and therefore, uniformly in \(1\le j\le m\),
\[
        \sup_{|y|>2h_j}2\beta_Nk_N\Psi_{h_j}(y)
        \le
        Ch_j^{-r}
        \le
        Ch_1^{-r}
        \to0.
\]
Thus the condition in Lemma \ref{lem:far-product-section2} holds at
every scale \(h_j\).  The tail bounds and \(k_N=\beta_N^{-1}\) also give
\[
        \beta_N^2Nw_N\le N^{-1/2-\kappa+o(1)},
        \qquad
        \beta_NN|\mu_N|\le N^{-1/2-\kappa+o(1)} .
\]
Since \(2H-2<0\), it follows that
\[
        \sup_{1\le j\le m}\mathfrak d_N(h_j)
        \le
        N^{-1/2-\kappa+p(2H-2)+o(1)}
        \to0,
\]
where replacing the last scale by \(N\) only decreases \(h_j^{2H-2}\).  Also
\[
        s_{1,N}
        \le
        N^{-1/2-\kappa+p(H-\frac12)+o(1)}
        \to0.
\]
Hence the hypotheses of Lemmas \ref{lem:first-block-section2} and
\ref{lem:outer-shell-section2} hold for all large \(N\).

We now verify that all terms in the right-hand side of
\eqref{eq:log-reduction-section2} vanish,  except for
\(\mathbb Q(G_N(\delta)^c)\), whose $0$ limit will be stated in the end of the proof.  Since \(\tau_N=N^{-a}\to0\), for
large \(N\),
\[
        \lambda_N(\varepsilon)=\frac{\varepsilon}{4}N^{-a}.
\]

The first term satisfies
\[
      \mathbb Q(A_{h_1})
      \le
      \sum_{i=1}^N\sum_{|y|\le2h_1}
      \mathbb Q(|\xi|>k_N)
       \le Nh_1q_N
        \le
        C N^{1+p-\alpha H/2+o(1)}
        =
        C N^{p-1/2-\kappa+o(1)}
        \to0,
\]
because \(p<1/2+\kappa\).  Next,
\[
\begin{aligned}
        \frac{s_{1,N}}{\tau_N}
        &\le
        C N^a
        \beta_NN|\mu_N|h_1^{H-1/2}                         \\
        &\le
        N^{a-\frac12-\kappa+p(H-\frac12)+o(1)}
        \to0
\end{aligned}
\]
by the first inequality in \eqref{eq:a-choice-light-section2}.  Moreover,
\[
\begin{aligned}
        \frac{\sqrt{\mathfrak d_N(h_1)}}{\tau_N}
        &\le
        C N^a
        \left[
        \beta_N^2Nw_Nh_1^{2H-2}
        \right]^{1/2}                                      \\
        &\le
        N^{a-\frac14-\frac{\kappa}{2}+p(H-1)+o(1)}
        \to0
\end{aligned}
\]
by the second inequality in \eqref{eq:a-choice-light-section2}.

For \(j\ge2\),
\[
\begin{aligned}
        \frac{\beta_NNh_je_N}{I_{j-1}}
        &\le
        C N^{1/2-\kappa+p_j-2p_{j-1}+o(1)}
        =
        C N^{-\kappa/2+o(1)}
        \to0.
\end{aligned}
\]
If the last scale is replaced by \(h_m=N\), the same bound still holds.  Indeed,
the stopping rule for $p_m$ gives
\[
        2p_{m-1}-\frac12+\frac{\kappa}{2}\ge1,
\]
and therefore
\[
        \frac{\beta_NNh_me_N}{I_{m-1}}
        \le
        C N^{1/2-\kappa+1-2p_{m-1}+o(1)}
        \le
        C N^{-\kappa/2+o(1)}.
\]
The bad far-field events are summable:
choosing \(\rho>0\) sufficiently small in \eqref{eq:Dh-bound-section2},
\[
\begin{aligned}
        \sum_{j=1}^m\mathbb Q(D_{h_j})
        &\le
        C\sum_{j=1}^m
        N\beta_N^{\alpha-\rho}h_j^{1-r(\alpha-\rho)}
        \to0 .
\end{aligned}
\]
Indeed, at \(\rho=0\) the worst exponent is
\[
        1-\frac{\alpha H}{2}+p(1-\alpha r)
        =
        -\frac12-\kappa+p(1-\alpha r)<0,
\]
since \(\alpha r>1\).  The same strict negativity remains valid for all small
\(\rho>0\).

Finally we control the remaining two sums in
\eqref{eq:log-reduction-section2}.  Here
\[
        I_1=N^{2p-1}\to\infty,
        \qquad
        I_{j-1}\ge I_1\quad (j\ge2),
\]
and the preceding estimate gives
\[
        \sup_{1\le j\le m}\mathfrak d_N(h_j)\to0.
\]
Therefore, for all large \(N\), uniformly over \(j\ge2\),
\[
        C\mathfrak d_N(h_j)\le \frac{I_1}{8}
        \le \frac{I_{j-1}}{8}.
\]
Consequently
\[
        \exp\{-I_{j-1}/4+C\mathfrak d_N(h_j)\}
        \le
        \exp\{-I_{j-1}/8\}
        \le
        \exp\{-I_1/8\}.
\]
Since \(I_1=N^{2p-1}\gg a\log N=\log(1/\tau_N)\), we obtain
\[
\begin{aligned}
\frac1{\tau_N} \sum_{j=2}^m \exp\{-I_{j-1}/4 + C\mathfrak d_N(h_j)\}
&\le C m \exp\{a\log N - I_1/8\} \to 0, \\
\frac1{\tau_N} \sum_{j=2}^m e^{-I_{j-1}/2}
&\le C m \exp\{a\log N - I_1/2\} \to 0.
\end{aligned}
\]
It remains only to optimize over the first scale \(h_1=N^p\).  For
a fixed \(p\in(1/2,\min\{1,1/2+\kappa\})\), the two restrictions in
\eqref{eq:a-choice-light-section2} say that \(a\) can be chosen below
\[
        \Phi(p):=
        \min\left\{
        \frac12+\kappa-p\left(H-\frac12\right),
        \frac14+\frac{\kappa}{2}+p(1-H)
        \right\}.
\]
The first function inside the minimum is decreasing in \(p\), while the second
one is increasing in \(p\).  They are equal when
\[
        \frac12+\kappa-p\left(H-\frac12\right)
        =
        \frac14+\frac{\kappa}{2}+p(1-H),
        \qquad\text{i.e.}\qquad
        p=\frac12+\kappa .
\]
If \(0<\kappa\le1/2\), taking \(p\uparrow1/2+\kappa\) gives
\[
        \sup_p\Phi(p)
        =
        \left(\frac12+\kappa\right)\left(\frac32-H\right).
\]
If \(\kappa\ge1/2\), on \((1/2,1)\) the minimum is maximized by taking \(p\uparrow1\), which gives
\[
        \sup_p\Phi(p)
        =
        \frac14+\frac{\kappa}{2}+1-H
        =
        \frac54-H+\frac{\kappa}{2}.
\]
Thus
\[
A_*(\alpha,H)=
\begin{cases}
\left(\dfrac12+\kappa\right)\left(\dfrac32-H\right),
&0<\kappa\le\dfrac12,\\[1.2ex]
\dfrac54-H+\dfrac{\kappa}{2},
&\kappa\ge\dfrac12.
\end{cases}
\]

\medskip
\noindent\textbf{The regime \(2<\alpha\le3/H\).}

Now
\[
        \beta_N=\beta\,l(N^{3/2})^{-1},
        \qquad
        k_N=
        \beta_N^{-1}
        \frac{l(N^{3/2}(\log N)^\eta)}{l(N^{3/2})}.
\]
We compare at the scale
\[
        \tau_N=a_N:=\beta_NN^{H/2}.
\]
Definition of the tail scale \(l\), the Karamata's theorem and the Potter bounds give
\begin{equation}\label{eq:heavy-tail-scales-section2}
        q_N\le C N^{-3/2}(\log N)^{-\eta},
\end{equation}
\begin{equation}\label{eq:heavy-first-moment-section2}
        \beta_Ne_N
        \le
        C N^{-3/2}
        (\log N)^{-\eta(\alpha-1)/\alpha+o(1)},
\end{equation}
and
\begin{equation}\label{eq:heavy-second-moment-section2}
        \beta_N^2w_N
        \le
        C N^{-3/2}
        (\log N)^{-\eta(\alpha-2)/\alpha+o(1)}.
\end{equation}

Choose
\[
        h_j=\lfloor \sqrt N(\log N)^{\theta_j}\rfloor,
\]
where
\[
        \frac12<\theta_1<\eta,
        \qquad
        \theta_j=2\theta_{j-1}+c,
        \qquad
        0<c<\eta\frac{\alpha-1}{\alpha}.
\]
Stop at the first \(m=m_N\) for which \(h_m\ge N\), and replace \(h_m\) by
\(N\) if necessary.  Then
\(m_N=O(\log\log N)\).

We again start by checking the hypotheses of Lemmas
\ref{lem:first-block-section2} and \ref{lem:outer-shell-section2}.  Since
\(h_1\to\infty\), all scales \(h_j\) tend to infinity. We have
\[
        \beta_Nk_N
        =
        \frac{l(N^{3/2}(\log N)^\eta)}{l(N^{3/2})}
        \le
        C(\log N)^{\eta/\alpha+o(1)}.
\]
Therefore, uniformly in \(1\le j\le m_N\),
\[
        \sup_{|y|>2h_j}\beta_Nk_N\Psi_{h_j}(y)
        \le
        C(\log N)^{\eta/\alpha+o(1)}
        N^{-r/2}(\log N)^{-r\theta_1}
        \to0.
\]
Thus the condition in Lemma \ref{lem:far-product-section2} holds at
all scales \(h_j\).  Moreover, since \(2H-2<0\) and \(h_j\ge h_1\),
\[
        \sup_{1\le j\le m_N}\mathfrak d_N(h_j)
        \le
        C N^{H-3/2+o(1)}
        (\log N)^{
        -\eta(\alpha-2)/\alpha
        +\theta_1(2H-2)
        +o(1)}
        \to0.
\]
Finally,
\[
        s_{1,N}
        \le
        C N^{H/2-3/4+o(1)}
        (\log N)^{
        \theta_1(H-\frac12)
        -\eta(\alpha-1)/\alpha
        +o(1)}
        \to0.
\]
Thus Lemmas \ref{lem:first-block-section2} and
\ref{lem:outer-shell-section2} can be applied for all large \(N\), uniformly
over \(1\le j\le m_N\).

We now verify the right-hand side of \eqref{eq:log-reduction-section2}.  Recall
that
\[
        \lambda_N(\varepsilon)
        =
        \frac14\min\{1,\varepsilon a_N\}.
\]

The first term is
\[
\mathbb Q(A_{h_1})
      \le
      \sum_{i=1}^N\sum_{|y|\le2h_1}
      \mathbb Q(|\xi|>k_N)
       \le 
        Nh_1q_N
        \le
        C(\log N)^{\theta_1-\eta}
        \to0.
\]
Next,
\[
\begin{aligned}
        \frac{s_{1,N}}{a_N}
        &=
        \frac{\beta_NN|\mu_N|h_1^{H-1/2}}
        {\beta_NN^{H/2}}                                    \\
        &\le
        C N^{-3/4+3/(2\alpha)+o(1)}
        (\log N)^{
        \theta_1(H-\frac12)
        -\eta(\alpha-1)/\alpha
        +o(1)}
        \to0,
\end{aligned}
\]
because of \(\alpha>2\).  Then,
\[
\begin{aligned}
        \frac{\mathfrak d_N(h_1)}{a_N^2}
        &=
        \frac{\beta_N^2Nw_Nh_1^{2H-2}}
        {\beta_N^2N^H}
        =
        w_NN^{1-H}h_1^{2H-2}                                 \\
        &\le
        C N^{3/\alpha-3/2+o(1)}
        (\log N)^{
        -\eta(\alpha-2)/\alpha
        +\theta_1(2H-2)
        +o(1)}
        \to0,
\end{aligned}
\]
again since \(\alpha>2\).  Thus
\[
        a_N^{-1}\sqrt{\mathfrak d_N(h_1)}\to0.
\]

For \(j\ge2\), first suppose that the scale has not been replaced by \(N\).
Then \(h_j\asymp \sqrt N(\log N)^{\theta_j}\) and
\[
        I_{j-1}=\frac{h_{j-1}^2}{N}
        \asymp
        (\log N)^{2\theta_{j-1}}.
\]
Using \eqref{eq:heavy-first-moment-section2},
\[
\begin{aligned}
        \frac{\beta_NNh_je_N}{I_{j-1}}
        &\le
        C
        \frac{
        N\sqrt N(\log N)^{\theta_j}
        N^{-3/2}(\log N)^{-\eta(\alpha-1)/\alpha+o(1)}
        }
        {(\log N)^{2\theta_{j-1}}}                          \\
        &=
        C(\log N)^{
        \theta_j-2\theta_{j-1}
        -\eta(\alpha-1)/\alpha
        +o(1)}                                               \\
        &=
        C(\log N)^{
        c-\eta(\alpha-1)/\alpha
        +o(1)}
        \to0 .
\end{aligned}
\]
If the last scale has been replaced by \(h_m=N\), let
\(\widetilde h_m:=\sqrt N(\log N)^{\theta_m}\) denote the unreplaced scale.
At the stopping time \(N\le\widetilde h_m\), while the denominator is still
\(I_{m-1}=h_{m-1}^2/N\).  Hence
\[
        \frac{\beta_NNh_me_N}{I_{m-1}}
        =
        \frac{\beta_NN^2e_N}{I_{m-1}}
        \le
        \frac{\beta_NN\widetilde h_me_N}{I_{m-1}},
\]
and the preceding estimate, with \(j=m\), applies to the right-hand side.
Since \(m_N=O(\log\log N)\), the summed near-field shell error satisfies
\[
        \sum_{j=2}^{m_N}
        \frac{\beta_NNh_je_N}{I_{j-1}}
        \le
        C(\log\log N)
        (\log N)^{
        c-\eta(\alpha-1)/\alpha
        +o(1)}
        \to0.
\]

The far-field bad events also vanish.  By \eqref{eq:Dh-bound-section2}, for
small enough \(\rho>0\),
\[
\begin{aligned}
        \sum_{j=1}^{m_N}\mathbb Q(D_{h_j})
        &\le
        Cm_N\,N\beta_N^{\alpha-\rho}
        h_1^{1-r(\alpha-\rho)}                              \\
        &\le
        C(\log\log N)
        N^{-\alpha r/2+\rho(3/(2\alpha)+r/2)+o(1)}
        \to0.
\end{aligned}
\]
Here we use \(\alpha r>1\), and then choose \(\rho\) so small that the polynomial exponent remains negative.

It remains to check the remaining two sums.  By the earlier bound
\[
        \sup_{1\le j\le m_N}\mathfrak d_N(h_j)\to0,
\]
and since \(I_{j-1}\ge I_1\) for \(j\ge2\), with
\[
        I_1=\frac{h_1^2}{N}\asymp(\log N)^{2\theta_1}\to\infty,
\]
we have, uniformly in \(j\ge2\), for all large \(N\),
\[
        C\mathfrak d_N(h_j)\le \frac{I_1}{8}\le\frac{I_{j-1}}8.
\]
Thus
\[
        \exp\{-I_{j-1}/4+C\mathfrak d_N(h_j)\}
        \le
        \exp\{-I_{j-1}/8\}
        \le
        \exp\{-I_1/8\}.
\]
If
\(2<\alpha<3/H\), then
\[
        \log(1/a_N)
        =
        \left(
        \frac{3}{2\alpha}-\frac H2+o(1)
        \right)\log N.
\]
If \(\alpha=3/H\), then \(a_N\) is slowly varying:
\[
        \log(1/a_N)=o(\log N).
\]
Since \(\theta_1>1/2\) and \(m_N=O(\log\log N)\), we have,
\[
        I_1\gg\log(1/a_N)+\log m_N .
\]
Hence,
\[
\begin{aligned}
\frac1{a_N} \sum_{j=2}^{m_N} \exp\{-I_{j-1}/4 + C\mathfrak d_N(h_j)\}
&\le \exp\{\log(1/a_N) + \log m_N - I_1/8\} \to 0, \\
\frac1{a_N} \sum_{j=2}^{m_N} e^{-I_{j-1}/2}
&\le \exp\{\log(1/a_N) + \log m_N - I_1/2\} \to 0.
\end{aligned}
\]
Thus, with
\(\tau_N=a_N\), all terms in
\eqref{eq:log-reduction-section2} vanish in the whole range
\(2<\alpha\le3/H\).

\medskip
\noindent\textbf{Final estimate}

In either regime, \eqref{eq:log-reduction-section2} and the estimates above
give, for every fixed \(\varepsilon,\delta>0\),
\[
        \limsup_{N\to\infty}
        \mathbb Q\left(
        |\log Z_N(\beta_N;\omega)
        -
        \log Z_N(\beta_N;\bar\omega)|
        >
        \varepsilon\tau_N
        \right)
        \le
        C\limsup_{N\to\infty}\mathbb Q(G_N(\delta)^c).
\]
By the positivity assumption \eqref{eq:positivity-section2}, the right-hand
side tends to zero after sending \(\delta\downarrow0\).  This completes
the proof of Proposition \ref{truncated-error}.
\end{proof}

\section{The scaling limit of the modified partition function}
  In this section, we establish the limit of the modified partition function corresponding to the truncated and centered environment $\bar{\omega}$ defined in \eqref{truncated_omega}. We generally adopt the  methodology of \cite{GC23}, but with some differences in the specific implementation.
  \subsection{Modified partition function and estimations for transition densities }
  The modified partition function with respect to the truncated environment $\bar{\omega}$ is given by
$$
    \mathfrak{Z}_N(\beta_N; \bar{\omega}) := \mathbb{E}_{\mathbb{P}_N} \left[ \prod_{i=1}^N (1 + \beta_N \bar{\omega}(i, S_i)) \right].
$$
Expanding this product along the simple random walk paths and taking the expectation $\mathbb{E}_{\mathbb{P}_N}$ yields a discrete chaos expansion:
$$
    \mathfrak{Z}_N(\beta_N; \bar{\omega}) = 1 + \sum_{k=1}^N \beta_N^k \sum_{1 \le i_1 < \dots < i_k \le N} \mathbb{E}_{\mathbb{P}_N} \left[ \prod_{j=1}^k \bar{\omega}(i_j, S_{i_j}) \right].
$$

The discrete $k$-point time simplex and scaled spatial grid are defined by
$$
    \Delta_k^{(N)} := \left\{ \mathbf{t} = (t_1, \dots, t_k) \in \left(\frac{1}{N}\mathbb{N}\right)^k : 0 < t_1 < \dots < t_k \le 1 \right\}, \quad \mathcal{X}_N := N^{-1/2}\mathbb{Z}.
$$
For $(\mathbf{t}, \mathbf{x})\in (\mathbf{t}, \mathbf{x}) \in \Delta_k^{(N)} \times \mathcal{X}_N^k $,  let 
$$p_N^k(\mathbf{t}, \mathbf{x}) := \mathbb{P}_N(S_{N t_1} = N^{1/2}x_1, \dots, S_{N t_k} = N^{1/2}x_k)
$$
be the $k$-step transition probability, and denote the environment tensor by $\bar{\omega}_N^k(\mathbf{t}, \mathbf{x}) := \prod_{j=1}^k \bar{\omega}(N t_j, N^{1/2}x_j)$. The expansion can be rewritten as:
\begin{equation}\label{eq:discrete_chaos}
    \mathfrak{Z}_N(\beta_N; \bar{\omega}) = 1 + \sum_{k=1}^N \beta_N^k \sum_{(\mathbf{t}, \mathbf{x}) \in \Delta_k^{(N)} \times \mathcal{X}_N^k} p_N^k(\mathbf{t}, \mathbf{x}) \bar{\omega}_N^k(\mathbf{t}, \mathbf{x}). 
\end{equation}

To convert the sum over $\Delta_k^{(N)} \times \mathcal{X}_N^k$ into an integral over the continuous simplex $\Delta_k(1) = \{0 < t_1 < \dots < t_k \le 1\}$, we assign each grid point a space-time cell $\mathcal{C}_N^k(\mathbf{t}, \mathbf{x})$ of volume $(2N^{-3/2})^k$, extending $p_N^k$ and $\bar{\omega}_N^k$ as step functions. Introducing the spatially normalized density $\hat{p}_N^k(\mathbf{t}, \mathbf{x}) := (N^{1/2}/2)^k p_N^k(\mathbf{t}, \mathbf{x})$ and defining its symmetrization $\psi_N^k(\mathbf{t}, \mathbf{x}) := \frac{1}{k!} \sum_{\pi \in \mathcal{S}_k} \hat{p}_N^k(\pi \mathbf{t}, \pi \mathbf{x}) \mathbbm{1}_{\Delta_k(1)}(\pi \mathbf{t})$, we obtain the integral representation:
\begin{equation}\label{integral-representation}
    \mathfrak{Z}_N(\beta_N; \bar{\omega}) = 1 + \sum_{k=1}^N (\beta_N N)^k \int_{[0,1]^k} \int_{\mathbb{R}^k} \psi_N^k(\mathbf{t}, \mathbf{x}) \bar{\omega}_N^k(\mathbf{t}, \mathbf{x}) \, \mathrm{d}\mathbf{t} \, \mathrm{d}\mathbf{x}.
\end{equation}

Before analyzing the $L^2$-convergence, we need to compute the covariance gap between the truncated environment $\bar{\omega}$ and the original environment $\omega$. The original covariance is $\gamma(z) = \mathbb{E}_{\mathbb{Q}}[\omega(i,x)\omega(i,x+z)]$ as defined in \eqref{gamma}. For the truncated noise $\bar{\xi}(i,y)$, its variance $\sigma_{\bar{\xi}}^2$ is less than $\sigma_\xi^2 = 1$ with
$$
    1 - \sigma_{\bar{\xi}}^2 = \mathbb{E}_{\mathbb{Q}}[\xi^2 \mathbbm{1}_{\{|\xi| > k_N\}}] + (\mathbb{E}_{\mathbb{Q}}[\xi \mathbbm{1}_{\{|\xi| > k_N\}}])^2 \le C k_N^{2-\alpha} L(k_N).
$$
Consequently, the absolute difference between the truncated spatial covariance $\bar{\gamma}(z)$ and the original $\gamma(z)$ is uniformly bounded:
\begin{equation}\label{eq:covariance_gap}
    \sup_{z \in \mathbb{Z}} |\gamma(z) - \bar{\gamma}(z)| \le (1 - \sigma_{\bar{\xi}}^2) \sum_{u} |\psi_u| |\psi_{u+z}| \le C k_N^{2-\alpha} L(k_N).
\end{equation}
Since $k_N \to \infty$ for any $\alpha > 2$, this difference vanishes asymptotically. Recalling from \eqref{gamma1}-\eqref{gamma2} that $\gamma(N^{1/2}z) \sim N^{1/2-r}K(z)$, the rescaled truncated covariance function satisfies:
\begin{equation}\label{eq:rescaled_cov}
    N^{r-1/2} \bar{\gamma}(N^{1/2}z) \longrightarrow K(z) \quad \text{as } N \to \infty.
\end{equation}

With the asymptotic behavior \eqref{eq:rescaled_cov} of the rescaled covariance, we can uniformly bound the covariance integrals. Since $\bar{\gamma}([z])$ is bounded for small $|z|$ and decays proportionally to $K(z)$ as $|z| \to \infty$, there exists a universal constant $C > 0$ such that $N^{r-1/2}\bar{\gamma}(N^{1/2}z) \le C K(z)$ for all $z \in \mathbb{R}$ and $N \ge 1$. 

The $k$-point rescaled truncated covariance function is defined as $\bar{\gamma}_N^{(k)}(\mathbf{x}-\mathbf{y}) := \prod_{i=1}^k N^{r-1/2}\bar{\gamma}(N^{1/2}(x_i-y_i))$. For any test function $\varphi \in \mathcal{L}_K^k([0,1]^k \times \mathbb{R}^k)$, we can bound the multi-dimensional integral as follows:

\begin{align*}
&\left| \int_{[0,1]^k} \int_{\mathbb{R}^{2k}} \varphi(\mathbf{t}, \mathbf{x}) \bar{\gamma}_N^{(k)}(\mathbf{x} - \mathbf{y}) \varphi(\mathbf{t}, \mathbf{y}) \, \mathrm{d}\mathbf{t} \, \mathrm{d}\mathbf{x} \, \mathrm{d}\mathbf{y} \right| \\
&\le \int_{[0,1]^k} \int_{\mathbb{R}^{2k}} |\varphi(\mathbf{t}, \mathbf{x})| \bar{\gamma}_N^{(k)}(\mathbf{x} - \mathbf{y}) |\varphi(\mathbf{t}, \mathbf{y})| \, \mathrm{d}\mathbf{t} \, \mathrm{d}\mathbf{x} \, \mathrm{d}\mathbf{y} \\
&\le C^k \int_{[0,1]^k} \int_{\mathbb{R}^{2k}} |\varphi(\mathbf{t}, \mathbf{x})| \prod_{i=1}^k K(x_i - y_i) |\varphi(\mathbf{t}, \mathbf{y})| \, \mathrm{d}\mathbf{t} \, \mathrm{d}\mathbf{x} \, \mathrm{d}\mathbf{y} \\
&= C^k |||\varphi|||_{\mathcal{L}_K^k}^2,
\end{align*}
where we define the fractional kernel norm by
\begin{equation*}
|||\varphi|||_{\mathcal{L}_K^k}^2 := \int_{[0,1]^k} \int_{\mathbb{R}^{2k}} \varphi(\mathbf{t}, \mathbf{x}) \prod_{i=1}^k K(x_i - y_i) \varphi(\mathbf{t}, \mathbf{y}) \, \mathrm{d}\mathbf{t} \, \mathrm{d}\mathbf{x} \, \mathrm{d}\mathbf{y}.
\end{equation*}
Therefore, we have the general bound:
\begin{equation} \label{eq:general_bound_35}
\left| \int_{[0,1]^k} \int_{\mathbb{R}^{2k}} \varphi(\mathbf{t}, \mathbf{x}) \bar{\gamma}_N^{(k)}(\mathbf{x} - \mathbf{y}) \varphi(\mathbf{t}, \mathbf{y}) \, \mathrm{d}\mathbf{t} \, \mathrm{d}\mathbf{x} \, \mathrm{d}\mathbf{y} \right| \le C^k |||\varphi|||_{\mathcal{L}_K^k}^2.
\end{equation}
If the function $\varphi$ satisfies
\begin{equation*}
\sup_{\mathbf{t} \in [0,1]^k} \int_{\mathbb{R}^k} |\varphi(\mathbf{t}, \mathbf{x})| \, \mathrm{d}\mathbf{x} \le A < \infty,
\end{equation*}
then by applying the Hardy-Littlewood inequality for fractional integration along with H\"older's inequality, there exists a positive constant $A_H$ such that:
\begin{align}
|||\varphi|||_{\mathcal{L}_K^k}^2 &\le A_H \int_{[0,1]^k} \left( \int_{\mathbb{R}^k} |\varphi(\mathbf{t}, \mathbf{x})|^{\frac{2}{3-2r}} \, \mathrm{d}\mathbf{x} \right)^{3-2r} \mathrm{d}\mathbf{t} \nonumber \\
&\le A_H A^{2-2r} \int_{[0,1]^k} \int_{\mathbb{R}^k} |\varphi(\mathbf{t}, \mathbf{x})|^{2r} \, \mathrm{d}\mathbf{x} \, \mathrm{d}\mathbf{t}. \label{eq:norm_bound_36}
\end{align}

The following three lemmas give some estimates of $\psi_N^k$. We omit the proof of Lemma \ref{lem:3.1} (\cite{2014AKQ1,GC23,RANG20203408}), and provide alternative proofs for Lemma \ref{lem:LK-tightness} and Lemma \ref{lem:3.2}

\begin{lemma}\label{lem:3.1}
Assume that the conditions \textbf{(A)} and \textbf{(B)} hold with $r \in (1/2, 1)$. Let $\beta_N$ be a sequence of inverse temperatures satisfying $\beta_N N^{H/2} \to \beta \in [0, \infty)$ as $N \to \infty$. Then there exists a positive constant $C$ such that for any $N \ge 1$ and $k \ge 1$,
\begin{equation} \label{eq:rho_k_bound}
\| \rho_k \|_{\mathcal{L}_K^k}^2 \le \frac{C^k \Gamma^k(H)}{\Gamma(kH + 1)}
\end{equation}
and
\begin{equation} \label{eq:psi_k_bound}
k! \| \psi_N^k \|_{\mathcal{L}_K^k}^2 \le \frac{C^k \Gamma^k(H)}{\Gamma(kH + 1)},
\end{equation}
where
\begin{equation*}
\rho_k(\mathbf{t}, \mathbf{x}) := \prod_{i=1}^k \rho(t_i - t_{i-1}, x_i - x_{i-1}), \quad (\mathbf{t}, \mathbf{x}) \in \Delta_k(1) \times \mathbb{R}^k.
\end{equation*}
In particular, since $H > 0$, the bounds in \eqref{eq:rho_k_bound} and \eqref{eq:psi_k_bound} guarantee the summability of the following series:
\begin{equation} \label{eq:rho_summable}
\sum_{k=1}^\infty (\sqrt{2}\beta)^{2k} \| \rho_k \|_{\mathcal{L}_K^k}^2 < \infty,
\end{equation}
and
\begin{equation} \label{eq:psi_summable}
\lim_{l \to \infty} \limsup_{N \to \infty} \sum_{k \ge l} (\sqrt{2}\beta_N N^{H/2})^{2k} k! \| \psi_N^k \|_{\mathcal{L}_K^k}^2 = 0.
\end{equation}
\end{lemma}

\begin{lemma}\label{lem:LK-tightness}
Assume that conditions \textbf{(A)} and \textbf{(B)} hold, and let
\(H=\frac32-r\in(\frac12,1)\). For \(k\ge1\), set \(t_0=0\) and
\(s_i=t_i-t_{i-1}\). Then
\[
\lim_{\delta\downarrow0}\limsup_{N\to\infty}
\left\|\psi_N^k
{\mathbbm 1}_{\{\min_{1\le i\le k}s_i\le\delta\}}
\right\|_{\mathcal L_K^k}=0,
\]
and
\[
\lim_{M\to\infty}\limsup_{N\to\infty}
\left\|\psi_N^k
{\mathbbm 1}_{\{\max_{1\le i\le k}|x_i|>M\}}
\right\|_{\mathcal L_K^k}=0.
\]
The same estimates hold with \(\psi_N^k\) replaced by \(\tilde\rho_k\).
\end{lemma}

\begin{proof}
It is enough to work on the ordered
simplex \(\Delta_k(1)\); the symmetrization only changes the estimates by a
constant depending on \(k\).

We first consider the time part. The Hardy--Littlewood estimate
and the one-dimensional random-walk heat-kernel bound used in Lemma
\ref{lem:3.1} give, with the indicator inserted into the same computation,
\[
\left\|\psi_N^k
{\mathbbm 1}_{\{\min_i s_i\le\delta\}}
\right\|_{\mathcal L_K^k}^2
\le
C_{H,k}
\int_{\Delta_k(1)}
{\mathbbm 1}_{\{\min_i s_i\le 2\delta\}}
\prod_{i=1}^k s_i^{H-1}\,d\mathbf t+o(1).
\]
Since \(H>0\),
\[
\int_{\Delta_k(1)}\prod_{i=1}^k s_i^{H-1}\,d\mathbf t
=
\frac{\Gamma(H)^k}{\Gamma(kH+1)}<\infty,
\]
and hence the right-hand side converges to zero as \(\delta\downarrow0\).

For the spatial tail, the same localized version of the estimate in Lemma
\ref{lem:3.1} gives
\[
\left\|\psi_N^k
{\bf 1}_{\{\max_i |x_i|>M\}}
\right\|_{\mathcal L_K^k}^2
\le
C_{H,k}
P_N\left(\max_{0\le j\le N}|S_j|>M\sqrt N\right)^{2H}
\int_{\Delta_k(1)}
\prod_{i=1}^k s_i^{H-1}\,d\mathbf t+o(1).
\]
By the martingale maximal inequality, see Lawler and Limic
\cite[Theorem A.2.5]{lawler2010random},
\[
P_N\left(\max_{0\le j\le N}|S_j|>M\sqrt N\right)\longrightarrow0
\qquad\text{as }M\to\infty,
\]
uniformly in \(N\). Therefore the spatial-tail estimate also vanishes.

The proof for \(\tilde\rho_k\) is identical, replacing the rescaled
random-walk transition density by the Gaussian heat kernel. The estimates
\[
\|g_s\|_\infty\le Cs^{-1/2},\qquad \int_{\mathbb R}g_s(x)\,dx=1,
\]
give the same time-part bound, and the spatial tail follows from the
reflection principle for Brownian motion. This proves the lemma.
\end{proof}

\begin{lemma}\label{lem:3.2}
Assume that conditions \textbf{(A)} and \textbf{(B)} hold. Then for every
\(k\ge1\),
\[
\lim_{N\to\infty}\|\psi_N^k-\tilde\rho_k\|_{\mathcal L_K^k}^2=0,
\]
where
\[
\tilde\rho_k(\mathbf t,\mathbf x):=\mathrm{Sym}\{\rho_k(\mathbf t,\mathbf x)\}.
\]
\end{lemma}

\begin{proof}
Fix \(\delta>0\) and \(M<\infty\), and set
\[
E_{\delta,M}:=
\left\{(\mathbf t,\mathbf x)\in\Delta_k(1)\times\mathbb R^k:
\min_{1\le i\le k}(t_i-t_{i-1})>\delta,\ 
\max_{1\le i\le k}|x_i|\le M
\right\}.
\]
On \(E_{\delta,M}\), all time increments are bounded away from zero and all
spatial variables stay in a compact set. Hence the local central limit
theorem for the one-dimensional simple symmetric random walk implies
\[
\sup_{(\mathbf t,\mathbf x)\in E_{\delta,M}}
|\psi_N^k(\mathbf t,\mathbf x)-\tilde\rho_k(\mathbf t,\mathbf x)|
\longrightarrow0 .
\]
For the local central limit theorem, see Lawler and Limic
\cite[Proposition 2.5.3 and Corollary 2.5.4]{lawler2010random}.

Moreover, \(\psi_N^k\) and \(\tilde\rho_k\) are uniformly bounded on
\(E_{\delta,M}\). Since \(2H-2\in(-1,0)\), the kernel \(K\) is locally
integrable. Therefore, by dominated convergence,
\[
\lim_{N\to\infty}
\|(\psi_N^k-\tilde\rho_k){\bf 1}_{E_{\delta,M}}\|_{\mathcal L_K^k}=0.
\]

By the triangle inequality,
\[
\begin{aligned}
\|\psi_N^k-\tilde\rho_k\|_{\mathcal L_K^k}
&\le
\|(\psi_N^k-\tilde\rho_k){\bf 1}_{E_{\delta,M}}\|_{\mathcal L_K^k}  \\
&\quad+
\|\psi_N^k{\bf 1}_{E_{\delta,M}^c}\|_{\mathcal L_K^k}
+
\|\tilde\rho_k{\bf 1}_{E_{\delta,M}^c}\|_{\mathcal L_K^k}.
\end{aligned}
\]
Since
\[
E_{\delta,M}^c
\subset
\{\min_i(t_i-t_{i-1})\le\delta\}
\cup
\{\max_i |x_i|>M\},
\]
Lemma \ref{lem:LK-tightness} gives
\[
\lim_{\delta\downarrow0}\lim_{M\to\infty}\limsup_{N\to\infty}
\left(
\|\psi_N^k{\bf 1}_{E_{\delta,M}^c}\|_{\mathcal L_K^k}
+
\|\tilde\rho_k{\bf 1}_{E_{\delta,M}^c}\|_{\mathcal L_K^k}
\right)=0.
\]
Combining this with the convergence on \(E_{\delta,M}\) yields
\[
\lim_{N\to\infty}\|\psi_N^k-\tilde\rho_k\|_{\mathcal L_K^k}=0.
\]
The proof is complete.
\end{proof}

\subsection{Gaussian Comparison and Hypercontractivity}

To establish the limiting distribution of the modified partition function $\mathfrak{Z}_{N}(\beta_{N};\bar{\omega})$, we still rely on the Gaussian comparison argument based on the Lindeberg principle for polynomial chaos developed in \cite{MDO2005, Caravenna2017,GC23}. This approach involves comparing the truncated environment $\bar{\omega}$ with a Gaussian field $\mu$ that matches the correlation structure of the original environment $\omega$.

Let $\{\eta(i, x); (i, x) \in \mathbb{N} \times \mathbb{Z}\}$ be a family of i.i.d. standard Gaussian random variables, independent of the original environment variables $\{\xi(i, x)\}$. We define a Gaussian environment $\{\mu(i, x)\}$:
\begin{equation}
\mu(i, x) = \sum_{y=-\infty}^{\infty} \psi_{y-x} \eta(i, y). \label{eq:gauss_env}
\end{equation}
By construction, $\{\mu(i, x)\}$ is a centered Gaussian field with covariance $\mathbb{E}[\mu(i, x)\mu(j, z)] = \delta_{ij} \gamma(x-z)$. Analogous to \eqref{eq:discrete_chaos}, we define the discrete Gaussian chaos:
\begin{equation}
\mathfrak{Z}_{N}(\beta_{N}; \mu) = 1 + \sum_{k=1}^{N} \beta_{N}^{k} \sum_{(\mathbf{t},\mathbf{x}) \in \Delta_{k}^{(N)} \times \mathcal{X}_{N}^{k}} p_{N}^{k}(\mathbf{t},\mathbf{x}) \mu_{N}^{k}(\mathbf{t},\mathbf{x}), \label{eq:gauss_chaos}
\end{equation}
where $\mu_{N}^{k}(\mathbf{t},\mathbf{x}) := \prod_{j=1}^{k} \mu(Nt_{j}, N^{1/2}x_{j})$. Our goal is to show that $\mathfrak{Z}_{N}(\beta_{N}; \bar{\omega})$ and $\mathfrak{Z}_{N}(\beta_{N}; \mu)$ have the same asymptotic distribution.
Following the framework in \cite{MDO2005}, we define the truncated chaos of order $m$:
\begin{equation}
\mathfrak{Z}_{N}^{\le m}(\beta_{N}; \bar{\omega}) := 1 + \sum_{k=1}^{m} \beta_{N}^{k} \sum_{(\mathbf{t},\mathbf{x}) \in \Delta_{k}^{(N)} \times \mathcal{X}_{N}^{k}} p_{N}^{k}(\mathbf{t},\mathbf{x}) \bar{\omega}_{N}^{k}(\mathbf{t},\mathbf{x}),
\end{equation}
and let $\mathfrak{Z}_{N}^{\le m}(\beta_{N}; \mu)$ be defined similarly.

To control the error between the truncated environment $\bar{\omega}$ and the Gaussian environment $\mu$ from the variable exchanges, we utilize the hypercontractivity of the noise ensemble. For $1 \le p \le q < \infty$ and $\tau \in (0, 1)$, a random variable $X$ is said to be $(p, q, \tau)$-hypercontractive if for all $a \in \mathbb{R}$:
\begin{equation}
\|a + \tau X\|_q \le \|a + X\|_p, \label{eq:hyper_def}
\end{equation}
and it is known that if $\mathbb{E}[X] = 0$ and $\mathbb{E}[|X|^q] < \infty$ for $q > 2$, then $X$ is $(2, q, \tau)$-hypercontractive with $\tau = \frac{\|X\|_2}{2(q-1)^{1/2}\|X\|_q}$. In our setting, since the truncated underlying noise $\bar{\xi}$ is centered and possesses uniformly bounded moments of order $q$ for any $2 < q < \alpha$, the ensemble satisfies this condition with a contraction parameter $\tau$ uniformly bounded away from zero.

Generally, let $\mathcal{X} = \{\mathcal{X}_1, \dots, \mathcal{X}_n\}$ be an orthonormal ensemble of independent random variables. For a multilinear polynomial $Q(x) = \sum_{\sigma} c_{\sigma} x_{\sigma}$ over $\mathcal{X}$, where $\sigma = (\sigma_1, \dots, \sigma_n)$ is a multi-index and $x_{\sigma} := \prod_{i=1}^n \mathcal{X}_{i, \sigma_i}$, the ensemble $\mathcal{X}$ is $(p, q, \tau)$-hypercontractive if:
\begin{equation}
\| T_{\tau} Q \|_q \le \| Q \|_p, \label{eq:operator_hyper}
\end{equation}
for every such polynomial $Q$, where the operator $T_{\tau}$ is defined by $(T_{\tau} Q)(x) := \sum_{\sigma} \tau^{|\sigma|} c_{\sigma} x_{\sigma}$.

\begin{lemma}\label{lem:gaussian_inv}
Assume that conditions \textbf{(A)} and \textbf{(B)} hold.
Assume moreover that \(\alpha>2\). Choose and fix
$q\in(2,\min\{\alpha,3\})$ such that \(\mathbb E_{\mathbb Q}[|\xi|^q]<\infty\).
We also use the parameter choices for \(\beta_N\) in \eqref{beta_N}
and \(k_N\) in \eqref{k_N}. Then, for every fixed \(m\ge1\),
the truncated chaos expansions \(\mathfrak Z_N^{\le m}(\beta_N;\overline\omega)\)
and \(\mathfrak Z_N^{\le m}(\beta_N;\mu)\) possess the same limiting
distribution as \(N\to\infty\). That is, for any test function
\(f\in C_b^{(3)}(\mathbb R)\),
\begin{equation}
\lim_{N\to\infty}
\left|
\mathbb E_{\mathbb Q}
\left[f\left(\mathfrak Z_N^{\le m}(\beta_N;\overline\omega)\right)\right]
-
\mathbb E_{\mathbb Q}
\left[f\left(\mathfrak Z_N^{\le m}(\beta_N;\mu)\right)\right]
\right|=0 .
\label{eq:inv_limit}
\end{equation}
\end{lemma}

\begin{proof}
For a fixed \(f\in C_b^{(3)}(\mathbb R)\), let
\[
C_f:=
\max\{\|f\|_\infty,\|f'\|_\infty,\|f''\|_\infty,\|f'''\|_\infty\}.
\]
To estimate the difference in \eqref{eq:inv_limit}, we consider the
discrete space-time set \(T=\mathbb N\times\mathbb Z\). To ensure that the
sums are well-defined, we first restrict the indices to a finite block
\[
T_N:=\{(i,x_{i,m})\in T:1\le i\le N,\ 1\le m\le N\}.
\]

The environment vectors for each time step \(i\in\{1,\dots,N\}\)
are
\[
\overline{\boldsymbol\omega}_i
:=
(1,\overline\omega(i,x_{i,1}),\dots,\overline\omega(i,x_{i,N})),
\]
\[
\boldsymbol{\mu}_i
:=
(1,\mu(i,x_{i,1}),\dots,\mu(i,x_{i,N})).
\]
An interpolating sequence \(X^{(j)}\) to facilitate the
variable-by-variable replacement are defined by
\[
X^{(j)}
=(X_1^{(j)},\dots,X_N^{(j)})
:=
(\overline{\boldsymbol\omega}_1,\dots,\overline{\boldsymbol\omega}_j,
\boldsymbol{\mu}_{j+1},\dots,\boldsymbol{\mu}_N),
\qquad j=0,1,\dots,N.
\]
Thus \(X^{(N)}\) corresponds to the truncated environment \(\overline\omega\),
while \(X^{(0)}\) corresponds to the Gaussian environment \(\mu\).
For each \(k\in\{1,\dots,N\}\), we write
\[
X_k^{(j)} = \{X_{k,0}^{(j)},\dots,X_{k,N}^{(j)}\},
\qquad X_{k,0}^{(j)}\equiv1,
\]
where \(X_{k,m}^{(j)}\) represents the noise value at the space-time point
\((k,x_{k,m}), m=1,2,\dots,N\).

For any multi-index \(\sigma=(\sigma_1,\dots,\sigma_N)\in\{0,1,\dots,N\}^N\),
let \(|\sigma|\) denote the number of non-zero entries. Let
\(\mathbf m_\sigma=(i_1,\dots,i_{|\sigma|})\) be the temporal indices where
\(\sigma_{i_\ell}\ne0\), and let
\[
\mathbf x_\sigma
=
(x_{i_1,\sigma_{i_1}},\dots, x_{i_{|\sigma|},\sigma_{i_{|\sigma|}}})
\]
be the corresponding coordinates. The truncated Gaussian chaos can
be expressed as a multilinear polynomial:
\begin{equation}
\mathfrak Z_N^{\le m}(\beta_N;\mu)
=
1+\sum_{k=1}^m\sum_{\sigma:|\sigma|=k} c_\sigma\prod_{i=1}^N X_{i,\sigma_i}^{(0)}
:=Q(X^{(0)}),
\label{eq:chaos_poly}
\end{equation}
where the coefficients are determined by the discrete transition probabilities:
\[
c_\sigma = \beta_N^{|\sigma|} p_N^{|\sigma|}
\left(\frac{\mathbf m_\sigma}{N},\mathbf x_\sigma\right).
\]

By defining \(Q(X^{(j)})\) analogously for \(j=1,\dots,N\), we separate the terms
where \(\sigma_j=0\) from those where \(\sigma_j>0\):
\begin{equation}
\begin{aligned}
Q(X^{(j)})
&=
\sum_{\substack{|\sigma|\le m\\ \sigma_j=0}}
c_\sigma\prod_{k\ne j}X_{k,\sigma_k}^{(j)}
+
\sum_{\substack{|\sigma|\le m\\ \sigma_j>0}}
c_\sigma X_{j,\sigma_j}^{(j)} \prod_{k\ne j}X_{k,\sigma_k}^{(j)} \\
&=:\widetilde Q_j+R_j .
\end{aligned}
\label{eq:poly_decomp}
\end{equation}
Here \(\widetilde Q_j\) is independent of the \(j\)-th time layer, while
\(R_j\) captures the local fluctuation driven by \(\overline\omega\) at
step \(j\). Similarly, the adjacent interpolating configuration is
\(\widetilde Q_j+S_j\), where
\[
S_j
=
\sum_{\substack{|\sigma|\le m\\ \sigma_j>0}}
c_\sigma\mu(j,x_{j,\sigma_j}) \prod_{\substack{k\ne j\\1\le k\le N}}X_{k,\sigma_k}^{(j)} .
\]
By summing over all \(N\) replacement steps, we obtain the telescoping sum
\begin{align}
&f\left(\mathfrak Z_N^{\le m}(\beta_N;\overline\omega)\right)
-
f\left(\mathfrak Z_N^{\le m}(\beta_N;\mu)\right) \nonumber\\
&\qquad
=
\sum_{j=1}^N \left(f(Q(X^{(j)}))-f(Q(X^{(j-1)}))\right) \nonumber\\
&\qquad
=
\sum_{j=1}^N \left(f(\widetilde Q_j+R_j)-f(\widetilde Q_j+S_j)\right).
\label{eq:telescoping}
\end{align}

Since \(f\in C_b^{(3)}(\mathbb R)\), there exists \(C_{f,q}<\infty\) such that,
for all \(x,y\in\mathbb R\),
\begin{equation}
\left| f(x+y)-f(x)-f'(x)y-\frac12 f''(x)y^2 \right| \le C_{f,q}|y|^q .
\label{eq:q_taylor}
\end{equation}
Indeed, if \(|y|\le1\), the usual third-order Taylor formula gives a bound
by \(C_f|y|^3\le C_f|y|^q\), since \(q\le3\). If \(|y|>1\), the boundedness
of \(f,f'\), and \(f''\) gives a bound by \(C_f(1+|y|+|y|^2)\le C_{f,q}|y|^q\),
since \(q>2\).

Applying \eqref{eq:q_taylor} to \(R_j\) and \(S_j\), taking expectation,
and using the triangle inequality, we get
\begin{align}
&\left|
\mathbb E_{\mathbb Q}
\left[f(\widetilde Q_j+R_j)-f(\widetilde Q_j+S_j)\right]
\right|
\nonumber\\
&\qquad
\le
\left|
\mathbb E_{\mathbb Q}
\left[f'(\widetilde Q_j)(R_j-S_j)\right]
\right|
+
\frac12
\left|
\mathbb E_{\mathbb Q}
\left[f''(\widetilde Q_j)(R_j^2-S_j^2)\right]
\right|
+
C_{f,q}\mathbb E_{\mathbb Q}\left[|R_j|^q+|S_j|^q\right].
\label{eq:local_taylor_bound}
\end{align}
The first-order terms vanish. Indeed, the product
\(\prod_{k\ne j}X_{k,\sigma_k}^{(j)}\) is independent of both
\(\overline\omega(j,\cdot)\) and \(\mu(j,\cdot)\), and both fields are centered.
Hence
\[
\mathbb E_{\mathbb Q}[f'(\widetilde Q_j)R_j] = \mathbb E_{\mathbb Q}[f'(\widetilde Q_j)S_j] = 0 .
\]

We now compare the second-order terms. Let
\[
\sigma_N^2:=\mathbb E_{\mathbb Q}[\overline\xi(i,x)^2].
\]
we have the exact covariance identity \(\overline\gamma(z)=\sigma_N^2\gamma(z)\).
By the covariance gap estimate \eqref{eq:covariance_gap},
\[
\varepsilon_N:=|1-\sigma_N^2| \le C k_N^{2-\alpha}L_2(k_N).
\]
Because \(q<\alpha\), Potter's bound for slowly varying functions gives,
for all large \(N\),
\[
\varepsilon_N \le C_q k_N^{2-q}.
\]
Using the lower bound \(k_N\ge cN^{H/2}\), we obtain
\begin{equation}
\varepsilon_N\le C_q N^{-H(q-2)/2}.
\label{eq:eps_absorb}
\end{equation}
Expanding \(R_j^2\) and \(S_j^2\), and using the independence of \(\widetilde Q_j\)
from the \(j\)-th time layer, gives
\begin{align}
&\mathbb E_{\mathbb Q} \left[f''(\widetilde Q_j)R_j^2\right]
- \mathbb E_{\mathbb Q} \left[f''(\widetilde Q_j)S_j^2\right]
\nonumber\\
&= \sum_{\substack{\sigma_j^1>0,\ \sigma_j^2>0\\ |\sigma^1|\le m,\ |\sigma^2|\le m}}
c_{\sigma^1}c_{\sigma^2}
\mathbb E_{\mathbb Q} \left[ f''(\widetilde Q_j)
\left(\prod_{k\ne j}X_{k,\sigma_k^1}^{(j)}\right)
\left(\prod_{k\ne j}X_{k,\sigma_k^2}^{(j)}\right) \right] \nonumber\\
&\qquad\times
\left( \mathbb E_{\mathbb Q} [\overline\omega(j,x_{j,\sigma_j^1})
\overline\omega(j,x_{j,\sigma_j^2})]
- \mathbb E_{\mathbb Q} [\mu(j,x_{j,\sigma_j^1}) \mu(j,x_{j,\sigma_j^2})] \right).
\label{eq:second_order_expansion}
\end{align}
By \(\overline\gamma=\sigma_N^2\gamma\), the expansion above is exactly
\[
\mathbb E_{\mathbb Q} \left[f''(\widetilde Q_j)R_j^2\right]
= \sigma_N^2 \mathbb E_{\mathbb Q} \left[f''(\widetilde Q_j)S_j^2\right].
\]
Therefore,

\begin{equation}
\left| \mathbb E_{\mathbb Q} \left[f''(\widetilde Q_j)(R_j^2-S_j^2)\right] \right|
\le C_f\varepsilon_N\,\mathbb E_{\mathbb Q}[S_j^2].
\label{eq:second_order_absorbed_local}
\end{equation}
Combining \eqref{eq:local_taylor_bound}, \eqref{eq:second_order_absorbed_local},
and the vanishing first-order terms, we have
\begin{equation}
\left| \mathbb E_{\mathbb Q} \left[f(\widetilde Q_j+R_j)-f(\widetilde Q_j+S_j)\right] \right|
\le
C_{f,q}\mathbb E_{\mathbb Q}\left[|R_j|^q+|S_j|^q\right]
+ C_f\varepsilon_N\,\mathbb E_{\mathbb Q}[S_j^2].
\label{eq:local_q_plus_second}
\end{equation}

We next control the \(q\)-th moments in \eqref{eq:local_q_plus_second}.
Temporarily truncate the spatial moving average at range \(M>0\), i.e.,
\[
\mu^{(M)}(i,x)=\sum_{|y|\le M}\psi_y\eta(i,x+y),
\quad
\overline\omega^{(M)}(i,x) =\sum_{|y|\le M}\psi_y\overline\xi(i,x+y).
\]
Let \(R_{j,M}\) and \(S_{j,M}\) be obtained from \(R_j\) and \(S_j\) by
replacing \(\overline\omega,\mu\) with \(\overline\omega^{(M)},\mu^{(M)}\).
These are multilinear polynomials of degree at most \(m\) over a finite
collection of independent random variables.

Since \(q<\alpha\) and \(\mathbb E_{\mathbb Q}[|\xi|^q]<\infty\),
the truncated noise inherits this moment bound, yielding
\[
    \sup_{N\ge1}\mathbb E_{\mathbb Q}[|\overline\xi|^q]<\infty .
\]
Moreover, \(\sigma_N\to1\). After normalizing the non-Gaussian variables
by \(\widehat\xi_N:=\overline\xi/\sigma_N\) for all large \(N\), the
finite independent ensemble generated by \(\widehat\xi_N\) and \(\eta\) is
\((2,q,\tau_q)\)-hypercontractive. According to \cite{MDO2005},
the constant \(\tau_q\) can be explicitly given by
\[
    \tau_q = \min \left\{
    \frac{1}{2^{q/2} \mathbb{E}_{\mathbb{Q}}[|\widehat{\xi}_N|^q]^{1/q}},
    \frac{1}{2^{q/2} \mathbb{E}_{\mathbb{Q}}[|\eta|^q]^{1/q}}
    \right\}.
\]
The \(\tau_q > 0\)
is bounded away from zero and is independent of \(N,j\), and \(M\).
By the hypercontractive estimate for multilinear polynomials of degree at
most \(m\), 
\cite[Propositions 3.11 and 3.12]{MDO2005},
\begin{equation}
\mathbb E_{\mathbb Q} \left[|R_{j,M}|^q+|S_{j,M}|^q\right]
\le
\tau_q^{-mq} \left( \mathbb E_{\mathbb Q}[R_{j,M}^2]^{q/2}
+ \mathbb E_{\mathbb Q}[S_{j,M}^2]^{q/2} \right).
\label{eq:q_hyper_M}
\end{equation}
Letting \(M\to\infty\), using \(L^2\)-convergence of the finite-range
moving averages and Fatou's lemma, yields
\begin{equation}
\mathbb E_{\mathbb Q} \left[|R_j|^q+|S_j|^q\right]
\le
\tau_q^{-mq} \left( \mathbb E_{\mathbb Q}[R_j^2]^{q/2}
+ \mathbb E_{\mathbb Q}[S_j^2]^{q/2} \right).
\label{eq:q_hyper}
\end{equation}

We now proceed to the second-moment estimates. This part follows the
original proof. Let \(\mathbf t=(t_1,\dots,t_k)\in\Delta_k^{(N)}\), and
write \(j/N\in\mathbf t\) if \(j/N\) is one of the components of
\(\mathbf t\). Expanding the multilinear polynomial for \(R_j\), and using
independence across different time steps, the cross terms between different
chaos orders vanish. Therefore,
\begin{align}
\mathbb E_{\mathbb Q}[R_j^2]
&=
\sum_{k=1}^m \beta_N^{2k}
\sum_{\substack{\mathbf t\in\Delta_k^{(N)}\\ j/N\in\mathbf t}}
\sum_{\mathbf x^1,\mathbf x^2\in\mathcal X_N^k}
p_N^k(\mathbf t,\mathbf x^1) p_N^k(\mathbf t,\mathbf x^2) \nonumber\\
&\qquad\qquad \times
\mathbb E_{\mathbb Q} \left[ \prod_{i=1}^k
X_{Nt_i,N^{1/2}x_i^1}^{(j)} X_{Nt_i,N^{1/2}x_i^2}^{(j)} \right].
\label{eq:Rj_second_expansion}
\end{align}

To bound this discrete sum, we express it as an integral over
the continuous space-time simplex. Let
\[
        \mathcal I_j^{(N)}(\mathbf t)
        =
        \sum_{i=1}^k
        \mathbf 1_{I_j}(t_i),
        \qquad
        I_j:=((j-1)/N,j/N],
\]
Using the symmetrized density
\(\psi_N^k\), and the factor \(k!\) from extending the simplex to
\([0,1]^k\), we obtain
\begin{align}
\mathbb E_{\mathbb Q}[R_j^2]
&\le
C\sum_{k=1}^m (\beta_NN^{H/2})^{2k}k!
\int_{[0,1]^k} \int_{\mathbb R^{2k}}
|\psi_N^k(\mathbf t,\mathbf x)|\, \mathcal I_j^{(N)}(\mathbf t) \nonumber\\
&\qquad\qquad \times
\left|\overline\gamma_N^{(k)}(\mathbf x-\mathbf y)\right|
|\psi_N^k(\mathbf t,\mathbf y)|
\,d\mathbf x\,d\mathbf y\,d\mathbf t .
\label{eq:Rj_indicator_integral}
\end{align}

Applying the generalized fractional-kernel bound, it gives:
\begin{align}
\mathbb E_{\mathbb Q}[R_j^2]
&\le
C\sum_{k=1}^m(\beta_NN^{H/2})^{2k}k!
\int_{[0,1]^k}\int_{\mathbb R^{2k}}
|\psi_N^k(\mathbf t,\mathbf x)|\,
\mathcal I_j^{(N)}(\mathbf t) \nonumber\\
&\qquad\qquad \times
\left(\prod_{i=1}^kK(x_i-y_i)\right)
|\psi_N^k(\mathbf t,\mathbf y)|
\,d\mathbf x\,d\mathbf y\,d\mathbf t .
\label{eq:Rj_local_integral}
\end{align}
The spatial integrations give a factor \(C^k\),
where the constant \(C\) is universal; the remaining time integral is bounded by
\[
\int_{0<t_1<\cdots<t_k<1}
\mathcal I_j^{(N)}(\mathbf t)
\prod_{a=1}^k(t_a-t_{a-1})^{H-1}\,d\mathbf t ,
\qquad t_0:=0 .
\]
Since
\[
        \mathcal I_j^{(N)}(\mathbf t)
        =
        \sum_{i=1}^k
        \mathbf 1_{I_j}(t_i),
        \qquad
        I_j:=((j-1)/N,j/N],
\]
we fix \(i\).  Put \(t_i=s\).  The variables before \(s\) give
\[
\int_{0<t_1<\cdots<t_{i-1}<s}
\prod_{a=1}^{i-1}(t_a-t_{a-1})^{H-1}
(s-t_{i-1})^{H-1}\,dt_1\cdots dt_{i-1}
=
\frac{\Gamma(H)^i}{\Gamma(iH)}s^{iH-1},
\]
and the variables after \(s\) give
\[
\int_{s<t_{i+1}<\cdots<t_k<1}
\prod_{a=i+1}^k(t_a-t_{a-1})^{H-1}
\,dt_{i+1}\cdots dt_k
=
\frac{\Gamma(H)^{k-i}}{\Gamma((k-i)H+1)}
(1-s)^{(k-i)H}.
\]
Therefore
\[
\begin{aligned}
\int_{0<t_1<\cdots<t_k<1}
\mathbf 1_{I_j}(t_i)
\prod_{a=1}^k(t_a-t_{a-1})^{H-1}\,d\mathbf t\le
\frac{\Gamma(H)^k}
{\Gamma(iH)\Gamma((k-i)H+1)}
\int_{I_j}s^{iH-1}\,ds .
\end{aligned}
\]
For all \(i\ge1\), since \(0<s\le1\), \(
s^{iH-1}\le s^{H-1}.
\) Thus
\[
\int_{I_j}s^{iH-1}\,ds
\le
\int_{0}^{1/N}s^{H-1}\,ds
=
\frac1H N^{-H}.
\]
Summing over \(i=1,\dots,k\), we obtain
\[
\begin{aligned}
\int_{0<t_1<\cdots<t_k<1}
\mathcal I_j^{(N)}(\mathbf t)
\prod_{a=1}^k(t_a-t_{a-1})^{H-1}\,d\mathbf t\le
\frac{N^{-H}}{H}
\Gamma(H)^k
\sum_{i=1}^k
\frac{1}
{\Gamma(iH)\Gamma((k-i)H+1)} .
\end{aligned}
\]
Hence
\[
\mathbb E_{\mathbb Q}[R_j^2]
\le
CN^{-H}\sum_{k=1}^m
(\beta_NN^{H/2})^{2k} B_{k,H},
\]
where
\[
B_{k,H}
:=
\frac{C^k\Gamma(H)^k}{H}
\sum_{\ell=1}^k
\frac{1}
{\Gamma(\ell H)\Gamma((k-\ell)H+1)} .
\]
Because \(\beta_NN^{H/2}\) is bounded, and \(m\) is
fixed, consequently,
\begin{equation}
\sup_{1\le j\le N}\mathbb E_{\mathbb Q}[R_j^2] \le C_mN^{-H},
\qquad
\sup_{1\le j\le N}\mathbb E_{\mathbb Q}[S_j^2] \le C_mN^{-H}.
\label{eq:variance_sup_bound}
\end{equation}

To evaluate the sum of the variances over \(j=1,\dots,N\), we retain
\(\mathcal I_j^{(N)}\) in \eqref{eq:Rj_indicator_integral}. For every
\(\mathbf t\in[0,1]^k\),
\[
\sum_{j=1}^N\mathcal I_j^{(N)}(\mathbf t)
=
\sum_{j=1}^N\sum_{i=1}^k \mathbf 1_{((j-1)/N,j/N]}(t_i) = k.
\]
Starting from \eqref{eq:Rj_local_integral}, since \(k\le m\), Lemma \ref{lem:3.1} yields
\begin{equation}
\sum_{j=1}^N\mathbb E_{\mathbb Q}[R_j^2]
\le
Cm\sum_{k=1}^m k(\sqrt{2}\beta_NN^{H/2})^{2k}
\frac{A_H^k\Gamma^k(H)}{\Gamma(kH+1)} \le C_m .
\label{eq:variance_sum_bound}
\end{equation}
Similarly,
\begin{equation}
\sum_{j=1}^N\mathbb E_{\mathbb Q}[S_j^2]\le C_m .
\label{eq:S_variance_sum_bound}
\end{equation}

Combine \eqref{eq:eps_absorb},\eqref{eq:q_hyper},\eqref{eq:variance_sup_bound},\eqref{eq:variance_sum_bound}, and \eqref{eq:S_variance_sum_bound}
\begin{align}
&\left| \mathbb E_{\mathbb Q} \left[
f\left(\mathfrak Z_N^{\le m}(\beta_N;\overline\omega)\right) \right]
- \mathbb E_{\mathbb Q} \left[
f\left(\mathfrak Z_N^{\le m}(\beta_N;\mu)\right) \right] \right| \nonumber\\
&\qquad
\le C_{f,q}\tau_q^{-mq} \sum_{j=1}^N
\left( \mathbb E_{\mathbb Q}[R_j^2]^{q/2} + \mathbb E_{\mathbb Q}[S_j^2]^{q/2} \right)
+ C_f\varepsilon_N\sum_{j=1}^N\mathbb E_{\mathbb Q}[S_j^2] \nonumber\\
&\qquad
\le C_{f,m,q}\tau_q^{-mq} \left( \sup_{1\le j\le N}
\bigl(\mathbb E_{\mathbb Q}[R_j^2] + \mathbb E_{\mathbb Q}[S_j^2]\bigr) \right)^{(q-2)/2} \nonumber\\
&\qquad\qquad \times \sum_{j=1}^N
\bigl(\mathbb E_{\mathbb Q}[R_j^2] + \mathbb E_{\mathbb Q}[S_j^2]\bigr)
+ C_{f,m}\varepsilon_N \nonumber\\
&\qquad
\le C_{f,m,q}\tau_q^{-mq}N^{-H(q-2)/2} + C_{f,m}N^{-H(q-2)/2} \nonumber\\
&\qquad
= O\left(N^{-H(q-2)/2}\right).
\label{eq:final_error_alpha_gt_2}
\end{align}
The right-hand side converges to zero because \(q>2\) and \(H>0\). This
proves \eqref{eq:inv_limit}.

Finally, removing the finite-block restriction \(T_N\) follows from the same
square-summability estimates in Lemma \ref{lem:3.1}, uniformly in the finite
block. The proof is complete.
\end{proof}

\subsection{Scaling Limit of the Modified Partition Function}

By the Gaussian invariance principle established in Lemma \ref{lem:gaussian_inv}, we now present the main convergence result for the modified partition function $\mathfrak{Z}_N(\beta_N; \bar{\omega})$.

\begin{theorem}\label{thm:scaling_limit}
Assume that conditions \textbf{(A)} and \textbf{(B)} hold. Let the tail exponent satisfy $\alpha > \frac{6}{3-2r}$. If the inverse temperature sequence $\beta_N$ satisfies the thermodynamic scaling $\beta_N N^{H/2} \to \beta \in [0, \infty)$ as $N \to \infty$. Then the modified partition function $\mathfrak{Z}_N(\beta_N; \bar{\omega})$ converges in distribution to the continuous Wiener chaos $Z_{\sqrt{2}\beta}(1, \cdot)$ given by:
\begin{equation}
Z_{\sqrt{2}\beta}(1, \cdot) = 1 + \sum_{k=1}^{\infty} (\sqrt{2}\beta)^{k} \int_{\Delta_{k}(1)} \int_{\mathbb{R}^k} \tilde{\rho}_{k}(\mathbf{t}, \mathbf{x}) W(\mathrm{d}\mathbf{t}, \mathrm{d}\mathbf{x}).
\end{equation}
Here, the notations
\begin{itemize}
    \item $\Delta_k(1) = \{0 < t_1 < \dots < t_k \le 1\}$ denotes the standard $k$-dimensional time simplex;
    \item $\tilde{\rho}_{k}(\mathbf{t}, \mathbf{x})$ is the symmetrized Gaussian heat kernel introduced in Lemma \ref{lem:3.2};
    \item $W(\mathrm{d}\mathbf{t}, \mathrm{d}\mathbf{x})$ is a centered Gaussian space-time white noise, whose spatial covariance structure is  the fractional kernel $K(z) = H(2H-1)|z|^{2H-2}$ defined in Section 1.
\end{itemize}
\end{theorem}

\begin{proof}
For a fixed truncation level $m \ge 1$, we decompose the modified partition function into the truncated chaos expansion and its corresponding remainder:
$$
\mathfrak{Z}_N(\beta_N; \bar{\omega}) = \mathfrak{Z}_N^{\le m}(\beta_N; \bar{\omega}) + \mathfrak{R}_N^{>m}(\beta_N; \bar{\omega}).
$$
By Lemma \ref{lem:3.1}, there exists a constant $C$ such that for all $N \ge 1$, the second moment of the remainder is uniformly bounded:
\begin{equation}\label{L2summability}
\mathbb{E}_{\mathbb{Q}} \left[ \left| \mathfrak{R}_N^{>m}(\beta_N; \bar{\omega}) \right|^2 \right] \le \sum_{k=m+1}^N (\sqrt{2}\beta_N N^{H/2})^{2k} \frac{C^k \Gamma^k(H)}{\Gamma(kH+1)}.
\end{equation}
Because $H > 0$, the sum of this series converges to zero as $m \to \infty$ uniformly in $N$. Thus, the remainder $\mathfrak{R}_N^{>m}$ vanishes in $L^2(\mathbb{Q})$ as the truncation level increases.

In the Gaussian environment $\mu$, the truncated partition function $\mathfrak{Z}_N^{\le m}(\beta_N; \mu)$ is a multi-linear polynomial of a centered Gaussian field.
$$
\mathfrak{Z}_N^{\le m}(\beta_N; \mu) = 1 + \sum_{k=1}^m (\beta_N N)^{k} \int_{[0,1]^k} \int_{\mathbb{R}^k} \psi_N^k(\mathbf{t}, \mathbf{x}) \mu_N^k(\mathbf{t}, \mathbf{x}) \,\mathrm{d}\mathbf{t} \,\mathrm{d}\mathbf{x}.
$$

We define the corresponding truncated continuous Wiener chaos:
$$
Z_{\sqrt{2}\beta}^{\le m}(1, \cdot) = 1 + \sum_{k=1}^m (\sqrt{2}\beta)^{k} \int_{\Delta_{k}(1)} \int_{\mathbb{R}^k} \tilde{\rho}_{k}(\mathbf{t}, \mathbf{x}) W(\mathrm{d}\mathbf{t}, \mathrm{d}\mathbf{x}).
$$

Adapting the convergence framework for discrete Gaussian polynomial chaos from Lemma 3.5 of \cite{GC23} to our setting, we embed the discrete Gaussian field $\mu$ into a continuous fractional white noise framework. To estimate the $L^2(\mathbb{Q})$-distance between the $k$-th discrete chaos term and its continuous counterpart, we apply the Itô isometry for multiple fractional Wiener integrals on the simplex $\Delta_k(1)$. Adding and subtracting $(\sqrt{2}\beta_N N^{H/2})^k \tilde{\rho}_k$ to separate the scaling error from the kernel approximation error, and using the elementary inequality $\|aX - bY\|^2 \le 2a^2\|X-Y\|^2 + 2(a-b)^2\|Y\|^2$, we obtain:
\begin{align*}
&\mathbb{E}_{\mathbb{Q}} \left[ \left| (\beta_N N)^k \int_{\Delta_k(1)} \int_{\mathbb{R}^k} \psi_N^k \mu_N^k \,\mathrm{d}\mathbf{t} \,\mathrm{d}\mathbf{x} - (\sqrt{2}\beta)^k \int_{\Delta_k(1)} \int_{\mathbb{R}^k} \tilde{\rho}_k \,\mathrm{d}W \right|^2 \right] \\
&= \mathbb{E}_{\mathbb{Q}} \left[ \left| \int_{\Delta_k(1) \times \mathbb{R}^k} \left( (\sqrt{2}\beta_N N^{H/2})^k \psi_N^k - (\sqrt{2}\beta)^k \tilde{\rho}_k \right) \mathrm{d}W \right|^2 \right] \\
&= k! \left\| (\sqrt{2}\beta_N N^{H/2})^k \psi_N^k - (\sqrt{2}\beta)^k \tilde{\rho}_k \right\|_{\mathcal{L}_K^k}^2 \\
&\le 2k! (\sqrt{2}\beta_N N^{H/2})^{2k} \|\psi_N^k - \tilde{\rho}_k\|_{\mathcal{L}_K^k}^2 + 2k! \left( (\sqrt{2}\beta_N N^{H/2})^k - (\sqrt{2}\beta)^k \right)^2 \|\tilde{\rho}_k\|_{\mathcal{L}_K^k}^2.
\end{align*}

Since $\beta_N N^{H/2} \to \beta$ as $N \to \infty$, the coefficient $\left( (\sqrt{2}\beta_N N^{H/2})^k - (\sqrt{2}\beta)^k \right)^2$ vanishes. Furthermore, Lemma \ref{lem:3.1} ensures $\|\tilde{\rho}_k\|_{\mathcal{L}_K^k}^2$ is finite. Thus, the second term is a higher-order infinitesimal. We can extract a constant $C(k, \beta)$ such that the total error is asymptotically dominated by the first term:
$$
\le C(k, \beta) \|\psi_N^k - \tilde{\rho}_k\|_{\mathcal{L}_K^k}^2.
$$

By Lemma \ref{lem:3.2}, $\|\psi_N^k - \tilde{\rho}_k\|_{\mathcal{L}_K^k} \to 0$ as $N \to \infty$. Consequently, each discrete term converges in $L^2(\mathbb{Q})$ to its corresponding multiple stochastic integral, which implies convergence in distribution for any fixed $m \ge 1$:
\begin{equation}\label{truncated Gaussian to Wiener}
\mathfrak{Z}_N^{\le m}(\beta_N; \mu) \xrightarrow{d} Z_{\sqrt{2}\beta}^{\le m}(1, \cdot) \quad \text{as } N \to \infty.
\end{equation}

Finally, for any test function $f \in C_b^{(3)}(\mathbb{R})$ with $C_f = \max\{\|f\|_\infty,\|f'\|_\infty, \|f''\|_\infty, \|f'''\|_\infty\}$, the triangle inequality yields:
\begin{align*}
\left| \mathbb{E}_{\mathbb{Q}}[f(\mathfrak{Z}_N(\beta_N; \bar{\omega}))] - \mathbb{E}_{\mathbb{Q}}[f(Z_{\sqrt{2}\beta})] \right| &\le \left| \mathbb{E}_{\mathbb{Q}}[f(\mathfrak{Z}_N(\beta_N; \bar{\omega}))] - \mathbb{E}_{\mathbb{Q}}[f(\mathfrak{Z}_N^{\le m}(\beta_N; \bar{\omega}))] \right| \\
&\quad + \left| \mathbb{E}_{\mathbb{Q}}[f(\mathfrak{Z}_N^{\le m}(\beta_N; \bar{\omega}))] - \mathbb{E}_{\mathbb{Q}}[f(\mathfrak{Z}_N^{\le m}(\beta_N; \mu))] \right| \\
&\quad + \left| \mathbb{E}_{\mathbb{Q}}[f(\mathfrak{Z}_N^{\le m}(\beta_N; \mu))] - \mathbb{E}_{\mathbb{Q}}[f(Z_{\sqrt{2}\beta}^{\le m})] \right| \\
&\quad + \left| \mathbb{E}_{\mathbb{Q}}[f(Z_{\sqrt{2}\beta}^{\le m})] - \mathbb{E}_{\mathbb{Q}}[f(Z_{\sqrt{2}\beta})] \right|. 
\end{align*}

We analyze the four terms as $N \to \infty$ and then $m \to \infty$:
First, by the $L^2$-summability established in \eqref{L2summability}, the first term is bounded by $C_f ( \mathbb{E}_{\mathbb{Q}}[| \mathfrak{R}_N^{>m}(\beta_N; \bar{\omega}) |^2] )^{1/2}$, which vanishes as $m \to \infty$ uniformly in $N$.
Second, Lemma \ref{lem:gaussian_inv} guarantees the second term converges to $0$ as $N \to \infty$.
Third, \eqref{truncated Gaussian to Wiener} above ensures the third term converges to $0$ as $N \to \infty$.
Finally, the fourth term vanishes as $m \to \infty$ since the continuous Wiener chaos $Z_{\sqrt{2}\beta}(1, \cdot)$ is rigorously constructed as the $L^2$-limit of its finite truncations $Z_{\sqrt{2}\beta}^{\le m}$. So we complete the proof.
\end{proof}

\section{Proof of Theorem 1.3}
To complete the final step of the environment transition, we need to estimate the $L^2$-error induced by replacing the non-linear environment variable $\hat{\omega}(i,x)$ with the linear truncated environment $\bar{\omega}(i,x)$. We define the local environment difference as:
\begin{equation}\label{eq:env_diff_def}
\Delta \hat{\omega}(i, x) := \hat{\omega}(i, x) - \bar{\omega}(i, x) = \frac{e^{\beta_N \bar{\omega}(i, x) - \bar{\lambda}(\beta_N)} - 1}{\beta_N} - \bar{\omega}(i, x).
\end{equation}
We establish the following upper bounds for the environment covariances.

\begin{lemma}\label{lemma:env_covariance_bounds}
Assume that conditions \textbf{(A)} and \textbf{(B)} hold, with
\(\alpha>2\).  Recall \( \beta_N \) and \(k_N\) in \eqref{beta_N} and \eqref{k_N}, and fix $\theta\in(2,\alpha\wedge4)$.
Define
\[
        \Delta\hat\omega(i,x):=\hat\omega(i,x)-\bar\omega(i,x),
        \qquad
        \hat\omega(i,x)
        :=
        \frac{
        e^{\beta_N\bar\omega(i,x)-\bar\lambda(\beta_N)}-1}
        {\beta_N}.
\]
Then there exists a sequence \(\varepsilon_N\downarrow0\) such
that, uniformly in \(i\) and \(x,z\in\mathbb Z\),
\begin{equation}\label{eq:cov_hat_omega}
        \left|
        \mathbb E_{\mathbb Q}
        [\hat\omega(i,x)\hat\omega(i,z)]
        \right|
        \le C\gamma(x-z),
\end{equation}
and
\begin{equation}\label{eq:cov_delta_omega}
        \left|
        \mathbb E_{\mathbb Q}
        [\Delta\hat\omega(i,x)\Delta\hat\omega(i,z)]
        \right|
        \le \varepsilon_N\gamma(x-z).
\end{equation}
More precisely, one may take
\[
        \varepsilon_N=C(\beta_N^{\theta-2}+\beta_N^2).
\]
\end{lemma}

\begin{proof}
Throughout the proof \(C\) may change from line to line, but it is independent
of \(N,i,x,z\).  Put
\[
        X:=\bar\xi(i,0).
\]
For every fixed \(C_0<\infty\), we claim
\begin{equation}\label{eq:local-exp-moment-heavy}
        \sup_{|u|\le C_0\beta_N}
        \mathbb E_{\mathbb Q}
        \left[
        |X|^\theta e^{|u||X|}
        \right]\le C .
\end{equation}
Indeed, \(|X|\le |\xi|\mathbf 1_{\{|\xi|\le k_N\}}+|c_N|\), where
\[
        c_N:=\mathbb E_{\mathbb Q}
        [\xi\mathbf 1_{\{|\xi|\le k_N\}}],
\]
and \(\sup_N |c_N|<\infty\).  Hence it is enough to bound the corresponding
quantity with \(|\xi|\mathbf 1_{\{|\xi|\le k_N\}}\).  On
\(\{|\xi|\le\beta_N^{-1}\}\), \eqref{eq:local-exp-moment-heavy} is bounded by a constant because
\(\theta<\alpha\).  On
\(\{\beta_N^{-1}<|\xi|\le k_N\}\), by discretization, for a small
\(\rho>0\) with \(\theta+\rho<\alpha\),
\[
\begin{aligned}
&\mathbb E_{\mathbb Q}
\left[
|\xi|^\theta e^{C_0\beta_N|\xi|}
\mathbf 1_{\{\beta_N^{-1}<|\xi|\le k_N\}}
\right]                                                     \\
&\quad\le
C
\sum_{1\le m\le C\beta_N k_N}
e^{C_0(m+1)}
\left(\frac{m+1}{\beta_N}\right)^\theta
\mathbb Q\left(|\xi|>\frac{m}{\beta_N}\right)                \\
&\quad\le
C\beta_N^{\alpha-\theta-\rho}
\sum_{1\le m\le C\beta_N k_N}
e^{C_0(m+1)}m^{\theta-\alpha+\rho}.
\end{aligned}
\]
If \(k_N=\beta_N^{-1}\), then \(\beta_Nk_N=1\), and the last display is
bounded.  In the truncation,
\[
        \beta_Nk_N
        =
        \frac{l(N^{3/2}(\log N)^\eta)}{l(N^{3/2})}
        =
        (\log N)^{\eta/\alpha+o(1)}.
\]
Since \(\eta<\alpha\), the exponential factor in the last display is
\(N^{o(1)}\), whereas
\(\beta_N^{\alpha-\theta-\rho}=N^{-c+o(1)}\) for some \(c>0\).
This proves \eqref{eq:local-exp-moment-heavy}.

For \(s\in\mathbb R\), define
\[
        m_N(s):=\mathbb E_{\mathbb Q}[e^{sX}],
        \qquad
        U_s:=\frac{e^{sX}}{m_N(s)}-1,
        \qquad
        V_s:=U_s-sX .
\]
The variables \(U_s\) and \(V_s\) are centered.  We shall use the following estimates: whenever \(|s|,|t|\le C_0\beta_N\),
\begin{equation}\label{eq:one-site-U-bound}
        \left|\mathbb E_{\mathbb Q}[U_sU_t]\right|
        \le C|st|,
\end{equation}
and
\begin{equation}\label{eq:one-site-V-bound}
        \left|\mathbb E_{\mathbb Q}[V_sV_t]\right|
        \le C\beta_N^{\theta-2}|st|.
\end{equation}
To prove \eqref{eq:one-site-U-bound}, let
\[
        \ell_N(s):=\log m_N(s).
\]
By \eqref{eq:local-exp-moment-heavy} with \(\theta>2\),
\[
        \sup_{|u|\le C_0\beta_N} |\ell_N''(u)|\le C.
\]
Therefore
\[
\begin{aligned}
\left|
\ell_N(s+t)-\ell_N(s)-\ell_N(t)
\right|
&=
\left|
\int_0^s\int_0^t \ell_N''(u+v)\,dv\,du
\right|                                                     \\
&\le C|st|.
\end{aligned}
\]
Since
\[
        \mathbb E_{\mathbb Q}[U_sU_t]
        =
        \frac{m_N(s+t)}{m_N(s)m_N(t)}-1
        =
        \exp\{\ell_N(s+t)-\ell_N(s)-\ell_N(t)\}-1,
\]
\eqref{eq:one-site-U-bound} follows.

For \eqref{eq:one-site-V-bound}, use the Taylor estimate
\[
        |e^u-1-u|\le C|u|^{\theta/2}e^{u_+},
        \qquad 2<\theta\le4.
\]
Since \(\mathbb E_{\mathbb Q}[X]=0\), \eqref{eq:local-exp-moment-heavy}
implies
\[
        |m_N(s)-1|
        \le C|s|^{\theta/2},
        \qquad |s|\le C_0\beta_N,
\]
 hence \(m_N(s)\) is bounded above and below by positive constants for all
large \(N\).  Writing
\[
        V_s
        =
        \frac{e^{sX}-1-sX}{m_N(s)}
        -
        \frac{m_N(s)-1}{m_N(s)}(1+sX),
\]
we get, again by \eqref{eq:local-exp-moment-heavy},
\[
        \mathbb E_{\mathbb Q}[V_s^2]\le C|s|^\theta .
\]
Consequently, by Cauchy's inequality,
\[
        |\mathbb E_{\mathbb Q}[V_sV_t]|
        \le C|s|^{\theta/2}|t|^{\theta/2}
        \le C\beta_N^{\theta-2}|st|,
\]
because \(|s|,|t|\le C_0\beta_N\).  This proves
\eqref{eq:one-site-V-bound}.

We now pass to the moving-average environment.  For
\(x\in\mathbb Z\), put
\[
        s_y^x:=\beta_N\psi_{y-x}.
\]
Since \(\sup_y\psi_y<\infty\), \(|s_y^x|\le C\beta_N\). The following computations
 are first performed with \(y\) restricted to a finite set.  The bounds are
uniform in the finite set, and the infinite-volume identities follow by
\(L^2(\mathbb Q)\)-convergence, because
\[
        \sum_y \mathbb E_{\mathbb Q}[(U_{s_y^x})^2]
        \le C\beta_N^2\sum_y\psi_{y-x}^2<\infty .
\]
By independence of the variables \(\{\xi(i,y)\}_{y\in\mathbb Z}\) and \eqref{cor:variance_estimate}
\[
        e^{\beta_N\bar\omega(i,x)-\bar\lambda(\beta_N)}
        =
        \prod_{y\in\mathbb Z}(1+U_{s_y^x})
\]
in \(L^2(\mathbb Q)\), where $U_{s_y^x}$ depends only on $\bar{\xi}(i,y)$.  Thus
\[
        \beta_N\hat\omega(i,x)
        =
        \sum_{\emptyset\ne I\subset\mathbb Z}
        \prod_{y\in I}U_{s_y^x},
\]
and, since
\[
        \beta_N\bar\omega(i,x)=\sum_y s_y^x X_y,
        \qquad X_y:=\bar\xi(i,y),
\]
we also have, recalling $\Delta\hat\omega(i,x):=\hat\omega(i,x)-\bar\omega(i,x)$, 
\[
        \beta_N\Delta\hat\omega(i,x)
        =
        \sum_y V_{s_y^x}
        +
        \sum_{\substack{\emptyset\ne I\subset\mathbb Z\\ |I|\ge2}}
        \prod_{y\in I}U_{s_y^x}.
\]
Here \(I\subset\mathbb Z\) means that \(I\) is finite.

We first prove the covariance bound for \(\hat\omega\).  Expanding gives
\[
\begin{aligned}
&\beta_N^2
\mathbb E_{\mathbb Q}
[\hat\omega(i,x)\hat\omega(i,z)]                         \\
&\quad =
\sum_{\emptyset\ne I,J\subset\mathbb Z}
\mathbb E_{\mathbb Q}
\left[
\prod_{y\in I}U_{s_y^x}
\prod_{y\in J}U_{s_y^z}
\right].
\end{aligned}
\]
 Factors with different \(y\)-indices are independent and each
\(U_{s_y^x}\) is centered.  When \(I\ne J\), it is zero. Hence, only the
terms with \(I=J\) remain:
\[
        \beta_N^2
        \mathbb E_{\mathbb Q}
        [\hat\omega(i,x)\hat\omega(i,z)]
        =
        \sum_{\emptyset\ne I\subset\mathbb Z}
        \prod_{y\in I}
        \mathbb E_{\mathbb Q}
        [U_{s_y^x}U_{s_y^z}] .
\]
Taking absolute values and using \eqref{eq:one-site-U-bound},
\[
\begin{aligned}
&\beta_N^2
\left|
\mathbb E_{\mathbb Q}
[\hat\omega(i,x)\hat\omega(i,z)]
\right|                                                     \\
&\quad\le
\sum_{\emptyset\ne I\subset\mathbb Z}
\prod_{y\in I}
C\beta_N^2\psi_{y-x}\psi_{y-z}                              \\
&\quad=
\prod_y
\left(
1+
C\beta_N^2\psi_{y-x}\psi_{y-z}
\right)-1                                                   \\
&\quad\le
C\beta_N^2
\sum_y\psi_{y-x}\psi_{y-z}
=
C\beta_N^2\gamma(x-z),
\end{aligned}
\]
where the last inequality follows from
\[
        \prod_y(1+a_y)-1
        \le
        \exp\!\left\{\sum_y a_y\right\}-1
        \le
        C\sum_y a_y
\]
with \(a_y=C\beta_N^2\psi_{y-x}\psi_{y-z}\) and \(\sum_y a_y\le C\beta_N^2\gamma(0)\to0\), uniformly in \(x,z\).
Dividing by \(\beta_N^2\) gives \eqref{eq:cov_hat_omega}.

For the difference, by
\eqref{eq:one-site-V-bound},
\[
\begin{aligned}
&\beta_N^{-2}
\sum_y
\left|
\mathbb E_{\mathbb Q}
[V_{s_y^x}V_{s_y^z}]
\right|                                                     \\
&\quad\le
C\beta_N^{-2}\beta_N^{\theta-2}
\sum_y |s_y^x s_y^z|
=
C\beta_N^{\theta-2}\gamma(x-z).
\end{aligned}
\]
The terms of degree at least two give, by \eqref{eq:one-site-U-bound},
\[
\begin{aligned}
&\beta_N^{-2}
\sum_{\substack{I\subset\mathbb Z\\ |I|\ge2}}
\prod_{y\in I}
C\beta_N^2\psi_{y-x}\psi_{y-z}                              \\
&\quad\le
C\beta_N^{-2}
\left[
\exp\{C\beta_N^2\gamma(x-z)\}
-1-C\beta_N^2\gamma(x-z)
\right]                                                     \\
&\quad\le
C\beta_N^2\gamma(x-z)^2
\le
C\beta_N^2\gamma(0)\gamma(x-z).
\end{aligned}
\]
Combining the last two estimates gives
\[
        \left|
        \mathbb E_{\mathbb Q}
        [\Delta\hat\omega(i,x)\Delta\hat\omega(i,z)]
        \right|
        \le
        C(\beta_N^{\theta-2}+\beta_N^2)\gamma(x-z).
\]
Since \(\theta>2\) and \(\beta_N\to0\), the right-hand side has the form
\(\varepsilon_N\gamma(x-z)\) with \(\varepsilon_N\downarrow0\).  This proves
\eqref{eq:cov_delta_omega}.
\end{proof}

With the covariance bounds established in Lemma \ref{lemma:env_covariance_bounds}, we are now ready to evaluate the $L^2(\mathbb{Q})$-distance between the truncated normalized partition function $Z_N(\beta_N; \bar{\omega})e^{-N\bar{\lambda}(\beta_N)}$ and the modified partition function $\mathfrak{Z}_N(\beta_N; \bar{\omega})$.

Let us express the truncated normalized partition function via its discrete chaos expansion:
$$
Z_N(\beta_N; \bar{\omega})e^{-N\bar{\lambda}(\beta_N)} = 1 + \sum_{k=1}^N \hat{S}_k^N, \quad \text{where} \quad \hat{S}_k^N := \beta_N^k \sum_{(\mathbf{t}, \mathbf{x}) \in \Delta_k^{(N)} \times \mathcal{X}_N^k} p_N^k(\mathbf{t}, \mathbf{x}) \hat{\omega}_N^k(\mathbf{t}, \mathbf{x}).
$$
Similarly, the modified partition function is expanded as $\mathfrak{Z}_N(\beta_N; \bar{\omega}) = 1 + \sum_{k=1}^N \bar{S}_k^N$, where $\bar{S}_k^N$ is defined analogously by replacing $\hat{\omega}_N^k$ with $\bar{\omega}_N^k$.

First, for a fixed truncation level $m \ge 1$, we consider the tail sum  $\sum_{k=m+1}^N \hat{S}_k^N$. By the orthogonality of the discrete chaos expansion in different time steps, the $L^2(\mathbb{Q})$-norm of the tail is the sum of the variances of each order $k > m$:
$$
\mathbb{E}_{\mathbb{Q}} \left[ \left( \sum_{k=m+1}^N \hat{S}_k^N \right)^2 \right] = \sum_{k=m+1}^N \mathbb{E}_{\mathbb{Q}}[(\hat{S}_k^N)^2].
$$

By Lemma \ref{lemma:env_covariance_bounds}, the covariance of the $\hat{\omega}$ is bounded by the original covariance $\gamma$, i.e., $\mathbb{E}_{\mathbb{Q}}[\hat{\omega}(i,x)\hat{\omega}(i,z)] \le C\gamma(x-z)$. Following the rescaling $x \to N^{1/2}x$ and the discrete-to-continuous transition framework developed in Section 3, we bound the $k$-th order variance as:
$$
\mathbb{E}_{\mathbb{Q}}[(\hat{S}_k^N)^2] \le C_1^k \beta_N^{2k} N^k N^{k(H-1)} \int_{\Delta_k(1)} \int_{\mathbb{R}^{2k}} \hat{p}_N^k(\mathbf{t}, \mathbf{x}) \hat{p}_N^k(\mathbf{t}, \mathbf{z}) \prod_{i=1}^k K(x_i - z_i) \,\mathrm{d}\mathbf{x} \,\mathrm{d}\mathbf{z} \,\mathrm{d}\mathbf{t},
$$
where $K(z) = H(2H-1)|z|^{2H-2}$ is the fractional kernel.

Due to Lemma \ref{lem:3.1}, the integrals are uniformly controlled by
$$
\mathbb{E}_{\mathbb{Q}}[(\hat{S}_k^N)^2] \le (C_2 \beta_N N^{H/2})^{2k} \int_{\Delta_k(1)} \prod_{i=1}^k (t_i - t_{i-1})^{H-1} \,\mathrm{d}\mathbf{t} \le (C_2 \beta_N N^{H/2})^{2k} \frac{\Gamma^k(H)}{\Gamma(kH + 1)}.
$$
Given the scaling $\beta_N N^{H/2} \to \beta$, there exists a universal constant $A > 0$ such that for sufficiently large $N$, the tail variance is bounded by:
$$
\mathbb{E}_{\mathbb{Q}} \left[ \left( \sum_{k=m+1}^N \hat{S}_k^N \right)^2 \right] \le \sum_{k=m+1}^\infty \frac{A^{2k} \Gamma^k(H)}{\Gamma(kH + 1)}.
$$

Since $H > 0$, the tail remainder vanishes as the truncation level increases
$$
\lim_{m \to \infty} \limsup_{N \to \infty} \mathbb{E}_{\mathbb{Q}} \left[ \left( \sum_{k=m+1}^N \hat{S}_k^N \right)^2 \right] = 0.
$$
This justifies that the higher-order terms of the normalized chaos expansion are uniformly negligible in $L^2(\mathbb{Q})$.

Therefore, it suffices to show that for any fixed $1 \le k \le m$, the difference of the $k$-th chaos terms vanishes:
\begin{equation}\label{eq:limit_difference_k}
\lim_{N\to\infty} \mathbb{E}_{\mathbb{Q}}[(\hat{S}_k^N - \bar{S}_k^N)^2] = 0.
\end{equation}

To prove \eqref{eq:limit_difference_k}, we employ a telescoping sum to replace the environment variables one coordinate at a time. For $1 \le j \le k$, we define the intermediate difference terms by
$$
S_{k,j}^N := \beta_N^k \sum_{(\mathbf{t}, \mathbf{x}) \in \Delta_k^{(N)} \times \mathcal{X}_N^k} p_N^k(\mathbf{t}, \mathbf{x}) \left( \prod_{i=1}^{j-1} \hat{\omega}(t_i, x_i) \right) \Delta\hat{\omega}(t_j, x_j) \left( \prod_{i=j+1}^k \bar{\omega}(t_i, x_i) \right).
$$
Since $\hat{S}_k^N - \bar{S}_k^N = \sum_{j=1}^k S_{k,j}^N$, by Minkowski's inequality, it suffices to prove that $\|S_{k,j}^N\|_2 \to 0$ for each $j$.

Recalling that $\mathbb{E}_{\mathbb{Q}}[\bar{\omega}\bar{\omega}] \le 2\gamma$, $\mathbb{E}_{\mathbb{Q}}[\hat{\omega}\hat{\omega}] \le C\gamma$, and $\mathbb{E}_{\mathbb{Q}}[\Delta\hat{\omega}\Delta\hat{\omega}] \le C\beta_N^{\theta-2}\gamma$, we obtain
\begin{align*}
\mathbb{E}_{\mathbb{Q}}[(S_{k,j}^N)^2] &\le \beta_N^{2k} \sum_{\mathbf{t} \in \Delta_k^{(N)}} \sum_{\mathbf{x}, \mathbf{z} \in \mathcal{X}_N^k} p_N^k(\mathbf{t}, \mathbf{x}) p_N^k(\mathbf{t}, \mathbf{z}) \\
&\quad \times C^k \left( \prod_{i=1}^{j-1} \gamma(x_i - z_i) \right) \big(\mathcal{O}(\beta_N^{\theta-2})\gamma\big) \left( \prod_{i=j+1}^k \gamma(x_i - z_i) \right).
\end{align*}

To bound this discrete sum, we convert it into an integral over the continuous simplex $\Delta_k(1) \times \mathbb{R}^{2k}$. We substitute $\hat{p}_N^k(\mathbf{t}, \mathbf{x}) = (N^{1/2}/2)^k p_N^k(\mathbf{t}, \mathbf{x})$. For the covariance, we utilize the scaling relation $\gamma(N^{1/2}(x_i - z_i)) \le C N^{H-1} K(x_i - z_i)$.
We obtain
\begin{align*}
\mathbb{E}_{\mathbb{Q}}[(S_{k,j}^N)^2] &\le C_1^k N^k \beta_N^{2k} \mathcal{O}(\beta_N^{\theta-2}) N^{k(H-1)} \int_{\Delta_k(1)} \int_{\mathbb{R}^{2k}} \hat{p}_N^k(\mathbf{t}, \mathbf{x}) \hat{p}_N^k(\mathbf{t}, \mathbf{z}) \\
&\quad \times \left( \prod_{i=1}^k \left[ N^{H-1} K(x_i - z_i) \right] \right) \,\mathrm{d}\mathbf{x} \,\mathrm{d}\mathbf{z} \,\mathrm{d}\mathbf{t} \\
&\le C_2^k N^k \beta_N^{2k+\theta-2} N^{k(H-1)} \int_{\Delta_k(1)} \prod_{i=1}^k (t_i - t_{i-1})^{H-1} \,\mathrm{d}\mathbf{t}.
\end{align*}
The integral $\int_{\Delta_k(1)} \prod (t_i - t_{i-1})^{H-1} \,\mathrm{d}\mathbf{t}$ is bounded by $\frac{\Gamma^{k}(H)}{\Gamma(kH + 1)}$, which is a finite constant independent of $N$. Substituting the scaling $\beta_N = \beta N^{-H/2}$, we have the convergence rate
\begin{align*}
N^k (\beta N^{-H/2})^{2k+\theta-2} N^{kH-k} &= \beta^{2k+\theta-2} N^k N^{-kH - \frac{\theta-2}{2}H} N^{kH-k} \\
&= \beta^{2k+\theta-2} N^{- \frac{\theta-2}{2}H}.
\end{align*}
Recall from Lemma \ref{lemma:env_covariance_bounds} that we chose $\theta > 2$. Consequently, the variance is bounded by a strictly decaying power of $N$:
$$
\mathbb{E}_{\mathbb{Q}}[(S_{k,j}^N)^2] \le \mathcal{O}\big(N^{- \frac{\theta-2}{2}H}\big) \xrightarrow{N \to \infty} 0.
$$
Therefore, $\|\hat{S}_k^N - \bar{S}_k^N\|_2 \to 0$, confirming that $Z_N(\beta_N; \bar{\omega})e^{-N\bar{\lambda}(\beta_N)}$ and $\mathfrak{Z}_N(\beta_N; \bar{\omega})$ have the identical asymptotic distribution. With the $L^2$-equivalence established, we know that the truncated normalized partition function converges in distribution to the continuous Wiener chaos:
$$
Z_N(\beta_N; \bar{\omega})e^{-N\bar{\lambda}(\beta_N)} \xrightarrow{d} Z_{\sqrt{2}\beta}(1,\cdot).
$$

To apply Proposition \ref{truncated-error}, we must verify the following positivity condition:
\begin{equation}\label{eq:prop22_condition}
\lim_{\delta \downarrow 0} \limsup_{N\to\infty} \mathbb{Q}\big(Z_N(\beta_N; \bar{\omega}) < \delta \mathbb{E}_{\mathbb{Q}}[Z_N(\beta_N; \bar{\omega})]\big) = 0.
\end{equation}
By the Portmanteau Theorem for weak convergence, for any fixed $\delta > 0$, the discrete probability is bounded by the distribution of the continuous limit:
$$
\limsup_{N\to\infty} \mathbb{Q}\left( \frac{Z_N(\beta_N; \bar{\omega})}{\mathbb{E}_{\mathbb{Q}}[Z_N(\beta_N; \bar{\omega})]} < \delta \right) \le \mathbb{Q}\big(Z_{\sqrt{2}\beta}(1, \cdot) \le \delta\big).
$$
Taking the limit as $\delta \downarrow 0$, the right-hand side converges to $\mathbb{Q}\big(Z_{\sqrt{2}\beta}(1, \cdot) \le 0\big)$. Therefore, verifying condition \eqref{eq:prop22_condition} is equivalent to proving that the Wiener chaos is strictly positive almost surely. 

It is a well-established fact that the continuous Wiener chaos $Z_{\sqrt{2}\beta}(t, x)$ coincides with the solution $u(t, x)$ to the following stochastic heat equation (SHE) with multiplicative fractional noise:
$$
\partial_t u = \frac{1}{2} \Delta u + \sqrt{2} \beta u \dot{W}, \quad u(0, x) = \delta_0(x),
$$
where the noise $\dot{W}$ possesses the spatial covariance structure $\mathbb{E}_{\mathbb{Q}}[\dot{W}(t,x) \dot{W}(s,y)] = \delta(t-s) |x-y|^{2H-2}$.

By applying Theorem 1.4 in \cite{Chen2016RegularityAS} (setting the parameters $\alpha = 2$, $\delta = 0$, and the function $\rho(u) = \sqrt{2} \beta u$), one establishes that the solution $u(t, x)$ to this SHE is strictly positive almost surely for all $t > 0$. Consequently, $\mathbb{Q}\big(Z_{\sqrt{2}\beta}(1, \cdot) \le 0\big) = 0$, which satisfies the condition required in \eqref{eq:prop22_condition}.

Having this positivity condition in hand, we apply Proposition \ref{truncated-error}:
$$
    \left| \ln Z_N(\beta_N; \omega) - \ln Z_N(\beta_N; \bar{\omega}) \right| \xrightarrow{\mathbb{Q}} 0.
$$
Applying the Slutsky's Theorem yielding
\begin{equation}
\label{eq:original_log_conv}
    \ln Z_N(\beta_N; \omega) - N\bar{\lambda}(\beta_N) \xrightarrow{d} \ln Z_{\sqrt{2}\beta}(1, \cdot).
\end{equation}

To complete the proof of Theorem 1.3(a), we expand the term
\(N\bar{\lambda}(\beta_N)\).  Set
\[
    \theta := \frac{2}{H}.
\]
Because \(H\in(1/2,1)\), we have \(\theta\in(2,4)\).  Moreover, since
\(\alpha > 3/H\), we have
\[
    \theta = \frac{2}{H} < \frac{3}{H} < \alpha.
\]
Also, recalling that \(r = \frac32 - H\), we obtain
\[
    r\theta
    =
    \left(\frac32 - H\right)\frac{2}{H}
    =
    \frac{3}{H} - 2
    >
    1.
\]
In the present regime \(\alpha > \frac{6}{3-2r} = \frac{3}{H}\),
the truncation is \(k_N=\beta_N^{-1}\), therefore the conditions for applying the Lemma \ref{lem:cgf_expansion} are satisfied. Hence it applies with
\[
    q=1,
    \qquad
    \theta=\frac{2}{H},
    \qquad
    p=\lfloor \theta \rfloor
    =
    \left\lfloor \frac{2}{H} \right\rfloor
    =
    p'.
\]

Since \(\mathbb{E}_{\mathbb{Q}}[\bar{\xi}]=0\), it gives
\[
    \bar{\lambda}(\beta_N)
    =
    \sum_{y=-\infty}^{\infty}
    \log\left(
    1+
    \sum_{j=2}^{p'}
    \frac{(\beta_N \psi_y)^j}{j!}
    \mathbb{E}_{\mathbb{Q}}[\bar{\xi}^j]
    \right)
    +
    o\!\left(\beta_N^{2/H}\right).
\]
Multiplying by \(N\), and using
\[
    N\beta_N^{2/H}
    =
    (\beta_N N^{H/2})^{2/H}
    \longrightarrow
    \beta^{2/H},
\]
we obtain
\[
    N\,o\!\left(\beta_N^{2/H}\right)=o(1).
\]
Therefore
\[
    N\bar{\lambda}(\beta_N)
    =
    N\sum_{y=-\infty}^{\infty}
    \log\left(
    1+
    \sum_{j=2}^{p'}
    \frac{(\beta_N \psi_y)^j}{j!}
    \mathbb{E}_{\mathbb{Q}}[\bar{\xi}^j]
    \right)
    +
    o(1).
\]

Now define
\[
    X_y
    :=
    \sum_{j=2}^{p'}
    \frac{(\beta_N \psi_y)^j}{j!}
    \mathbb{E}_{\mathbb{Q}}[\bar{\xi}^j].
\]
Since \(H>1/2\), we have \(2/H<4\), hence
\[
    p'=\left\lfloor \frac{2}{H}\right\rfloor \le 3.
\]
Because \(\sup_y \psi_y<\infty\) and \(\beta_N\to0\), it follows that for all
large \(N\),
\[
    |X_y|
    \le
    C\big(\beta_N^2\psi_y^2+\beta_N^3\psi_y^3\big)
    \le
    C\beta_N^2\psi_y^2.
\]
In particular, \(\sup_y |X_y|\to0\), and hence
\[
    \log(1+X_y)=X_y+\mathcal{O}(X_y^2)
\]
uniformly in \(y\).  Therefore
\[
    N\bar{\lambda}(\beta_N)
    =
    N\sum_{y=-\infty}^{\infty} X_y
    +
    N\sum_{y=-\infty}^{\infty}\mathcal{O}(X_y^2)
    +
    o(1).
\]
It remains to show that the quadratic correction is negligible.  Using the
bound above and \(\sum_y \psi_y^4<\infty\), we obtain
\[
    \sum_{y=-\infty}^{\infty} X_y^2
    \le
    C\beta_N^4 \sum_{y=-\infty}^{\infty}\psi_y^4
    \le
    C\beta_N^4.
\]
Hence, it
\[
    N\sum_{y=-\infty}^{\infty}\mathcal{O}(X_y^2)
    =
    \mathcal{O}(N\beta_N^4).
\]
Finally, since \(\beta_N N^{H/2}\to\beta\), we have
\[
    N\beta_N^4 = \mathcal{O}(N^{1-2H}) \to 0
\]
because \(H>1/2\).

Consequently, the quadratic logarithmic correction vanishes, and we conclude
that
\[
    N\bar{\lambda}(\beta_N)
    =
    N\sum_{y=-\infty}^{\infty} X_y + o(1)
    =
    N\sum_{j=2}^{p'}
    \frac{\beta_N^j}{j!}
    \mathbb{E}_{\mathbb{Q}}[\bar{\xi}^j]
    \sum_{y=-\infty}^{\infty}\psi_y^j
    +o(1),
\]
Substitute this explicit polynomial back into
\eqref{eq:original_log_conv}.

Finally, We briefly discuss the case where $\beta_N N^{H/2} \to 0$ as $N \to \infty$, under the same tail condition  $\alpha > \frac{6}{3-2r}$.

When \(\beta=0\) in \eqref{beta_N}, the preceding lemmas and theorems remain valid. The limiting behavior of the free energy in this regime can be intuitively seen from the modified partition function. Recall the modified partition function $\mathfrak{Z}_N(\beta_N; \bar{\omega}) = 1 + \sum_{k=1}^N \overline{S}_k^N$. If we normalize it by $\beta_N N^{H/2}$, the $L^2$ bounds established in Section 3 imply that the variance of the scaled $k$-th order chaos term behaves asymptotically as
\begin{equation}
    \mathbb{E}_{\mathbb{Q}}\left[ \left( \frac{\overline{S}_k^N}{\beta_N N^{H/2}} \right)^2 \right] \sim \mathcal{O}\left( (\beta_N N^{H/2})^{2k-2} \right).
\end{equation}
For $k = 1$, the variance converges to a positive constant $\sigma_H^2$. However, for all higher-order terms ($k \ge 2$), the factor $(\beta_N N^{H/2})^{2k-2}$ guarantees that their scaled variances vanish. Consequently, we can obtain 
\begin{equation}
    \frac{\mathfrak{Z}_N(\beta_N; \bar{\omega}) - 1}{\beta_N N^{H/2}} = \frac{\overline{S}_1^N}{\beta_N N^{H/2}} + o_{L^2}(1) \approx \frac{1}{N^{H/2}} \sum_{i=1}^N \sum_{x\in\mathbb{Z}} p_N(i,x) \bar{\omega}(i,x).
\end{equation}
Since the right-hand side is a linear combination of moving average environment variables, the Lindeberg--Feller theorem for triangular arrays implies its convergence in distribution to a Gaussian distribution $\mathcal{N}(0, \sigma_H^2)$. Here, the variance can be explicitly evaluated as follows:
\begin{equation}
    \sigma_H^2 = 2 \int_0^1 \int_{\mathbb{R}^2} \rho(t,x)\rho(t,y)K(x-y) \,dx \,dy \,dt.
\end{equation}

Combining this with Proposition~\ref{truncated-error},
we obtain
\[
        \frac{
        \log Z_N(\beta_N;\omega)
        -
        \log Z_N(\beta_N;\bar\omega)
        }
        {\beta_NN^{H/2}}
        \xrightarrow{\mathbb Q}0,
\]
provided that \(\beta_NN^{H/2}\gg N^{-a}\) for some
\(a<A_*(\alpha,H)\). Hence, it remains only to identify \(N\bar\lambda(\beta_N)\) at this scale.

By Lemma~\ref{lem:cgf_expansion}, applied with
\(\theta=2/H\), we have
\[
        N\bar\lambda(\beta_N)
        =
        N\sum_{j=2}^{\lfloor2/H\rfloor}
        \frac{\beta_N^j}{j!}
        \mathbb E_{\mathbb Q}[\bar\xi^j]
        \sum_{y\in\mathbb Z}\psi_y^j
        +o(\beta_NN^{H/2}).
\]
Finally, since
\[
        \frac{Z_N(\beta_N;\bar\omega)e^{-N\bar\lambda(\beta_N)}-1}
        {\beta_NN^{H/2}}
        \xrightarrow{d}\mathcal N(0,\sigma_H^2),
\]
and \(\beta_NN^{H/2}\to0\), Taylor's formula gives
\[
        \frac{
        \log Z_N(\beta_N;\bar\omega)-N\bar\lambda(\beta_N)}
        {\beta_NN^{H/2}}
        \xrightarrow{d}\mathcal N(0,\sigma_H^2).
\]\hfill $\square$

\section{proof of Theorem 1.4}
\begin{proof}
Put
\[
        a_N:=\beta_NN^{H/2}.
\]
By assumption \(a_N\to0\).  We first prove the fluctuation result for the
truncated normalized partition function and, at the same time, verify the
positive condition required in Proposition~\ref{truncated-error}.

Recall that
\[
        W_N
        :=
        Z_N(\beta_N;\bar\omega)e^{-N\bar\lambda(\beta_N)}
        =
        \mathbb E_{\mathbb P_N}
        \left[
        \prod_{i=1}^N(1+\beta_N\hat\omega(i,S_i))
        \right],
\]
where
\[
        \hat\omega(i,x)
        =
        \frac{e^{\beta_N\bar\omega(i,x)-\bar\lambda(\beta_N)}-1}{\beta_N}.
\]
Expanding the product gives the discrete chaos decomposition
\[
        W_N
        =
        1+\sum_{k=1}^N\beta_N^k\widehat S_{N,k},
\]
with
\[
        \widehat S_{N,k}
        =
        \sum_{\mathbf i\in\Delta_k^{(N)}}
        \sum_{\mathbf x\in\mathbb Z^k}
        p_N^k(\mathbf i,\mathbf x)
        \prod_{\ell=1}^k\hat\omega(i_\ell,x_\ell).
\]
The first-order term, divided by \(a_N\), is
\[
        \frac{\beta_N\widehat S_{N,1}}{a_N}
        =
        N^{-H/2}
        \sum_{i=1}^N\sum_{x\in\mathbb Z}
        p_N(i,x)\hat\omega(i,x).
\]

We first replace \(\hat\omega\) by the linear truncated environment
\(\bar\omega\).  By Lemma~\ref{lemma:env_covariance_bounds},
\[
        \left|
        \mathbb E_{\mathbb Q}
        [\Delta\hat\omega(i,x)\Delta\hat\omega(i,z)]
        \right|
        \le
        \varepsilon_N\gamma(x-z),
        \qquad
        \varepsilon_N\downarrow0,
\]
where \(\Delta\hat\omega=\hat\omega-\bar\omega\).  Therefore
\begin{equation}
\begin{aligned}
&\mathbb E_{\mathbb Q}
\left[
\left(
N^{-H/2}
\sum_{i=1}^N\sum_{x\in\mathbb Z}
p_N(i,x)\Delta\hat\omega(i,x)
\right)^2
\right]                                                     \\
&\quad\le
\varepsilon_N
N^{-H}
\sum_{i=1}^N
\sum_{x,z\in\mathbb Z}
p_N(i,x)p_N(i,z)\gamma(x-z)
\le
C\varepsilon_N
\longrightarrow0.
\end{aligned}
\label{L2_env_error}
\end{equation}

It remains to prove the central limit theorem for
\[
        \mathcal L_N
        :=
        N^{-H/2}
        \sum_{i=1}^N\sum_{x\in\mathbb Z}
        p_N(i,x)\bar\omega(i,x).
\]
Using
\[
        \bar\omega(i,x)
        =
        \sum_{y\in\mathbb Z}\psi_{y-x}\bar\xi(i,y),
\]
write
\[
        \mathcal L_N
        =
        \sum_{i=1}^N X_{N,i},
        \qquad
        X_{N,i}
        :=
        N^{-H/2}
        \sum_{y\in\mathbb Z}a_{N,i}(y)\bar\xi(i,y),
\]
where
\[
        a_{N,i}(y):=
        \sum_{x\in\mathbb Z}p_N(i,x)\psi_{y-x}.
\]
The variables \(X_{N,1},\ldots,X_{N,N}\) are independent and centered.  Also,
if \(\sigma_N^2:=\mathbb E_{\mathbb Q}[\bar\xi^2]\), then
\(\sigma_N^2\to1\) and
\[
\begin{aligned}
        \sum_{y\in\mathbb Z}a_{N,i}(y)^2
        &=
        \sum_{x,z\in\mathbb Z}
        p_N(i,x)p_N(i,z)
        \sum_{y\in\mathbb Z}\psi_{y-x}\psi_{y-z}              \\
        &=
        \sum_{x,z\in\mathbb Z}
        p_N(i,x)p_N(i,z)\gamma(x-z)
        \le
        C(1+i)^{H-1}.
\end{aligned}
\]
Consequently,
\[
\begin{aligned}
        \operatorname{Var}_{\mathbb Q}(\mathcal L_N)
        &=
        \sigma_N^2
        N^{-H}
        \sum_{i=1}^N
        \sum_{x,z\in\mathbb Z}
        p_N(i,x)p_N(i,z)\gamma(x-z)                           \\
        &\longrightarrow
        2\int_0^1\int_{\mathbb R^2}
        \rho(t,x)\rho(t,y)K(x-y)\,dx\,dy\,dt
        =
        \sigma_H^2,
\end{aligned}
\]
where \(K(x)=H(2H-1)|x|^{2H-2}\).

We verify Lyapunov's condition.  Choose \(q\in(2,\alpha)\).  Since
\(\mathbb E_{\mathbb Q}[|\bar\xi|^q]\) is uniformly bounded, Lemma~B.1 gives
\[
        \left\|
        \sum_y a_{N,i}(y)\bar\xi(i,y)
        \right\|_{L^q(\mathbb Q)}
        \le
        C_q
        \max\{\|a_{N,i}\|_2,\|a_{N,i}\|_q\}
        \le
        C_q\|a_{N,i}\|_2,
\]
since \(q>2\).  Hence
\[
\begin{aligned}
        \sum_{i=1}^N
        \mathbb E_{\mathbb Q}[|X_{N,i}|^q]
        &\le
        C_q
        N^{-qH/2}
        \sum_{i=1}^N(1+i)^{q(H-1)/2}                         \\
        &\le
        C_q
        \left(
        N^{-qH/2}
        +
        N^{1-q/2}
        \right)
        \longrightarrow0.
\end{aligned}
\]
Therefore the Lindeberg--Feller theorem for triangular arrays ( \cite{Durrett_2019}) yields
\[
        \mathcal L_N
        \xrightarrow{d}
        \mathcal N(0,\sigma_H^2).
\]
Combining this with \eqref{L2_env_error}, we obtain the same limit for the
first-order chaos term:
\[
        \frac{\beta_N\widehat S_{N,1}}{a_N}
        \xrightarrow{d}
        \mathcal N(0,\sigma_H^2).
\]

We next show that all higher-order terms vanish after division by
\(a_N\).  By the covariance bound
\[
        \left|
        \mathbb E_{\mathbb Q}
        [\hat\omega(i,x)\hat\omega(i,z)]
        \right|
        \le
        C\gamma(x-z)
\]
in Lemma~\ref{lemma:env_covariance_bounds} and the same computation as in
Section~3 gives,
\[
\begin{aligned}
        \mathbb E_{\mathbb Q}
        \left[
        \left(
        \frac1{a_N}
        \sum_{k=2}^N\beta_N^k\widehat S_{N,k}
        \right)^2
        \right]
        &\le
        C
        \sum_{k=2}^\infty
        a_N^{2k-2}
        \frac{C^k\Gamma(H)^k}{\Gamma(kH+1)}
        \longrightarrow0,
\end{aligned}
\]
Thus
\begin{equation}\label{eq:truncated-fluctuation-theorem14}
        \frac{
        Z_N(\beta_N;\bar\omega)e^{-N\bar\lambda(\beta_N)}-1
        }
        {a_N}
       \stackrel{d} {\longrightarrow}
        \mathcal N(0,\sigma_H^2).
\end{equation}

In particular,
\[
        Z_N(\beta_N;\bar\omega)e^{-N\bar\lambda(\beta_N)}
       \stackrel{\mathbb Q} {\longrightarrow}1.
\]
Since
\[
        \mathbb E_{\mathbb Q}[Z_N(\beta_N;\bar\omega)]
        =
        e^{N\bar\lambda(\beta_N)},
\]
this implies the positivity condition required by
Proposition~\ref{truncated-error}:
\[
        \lim_{\delta\downarrow0}\limsup_{N\to\infty}
        \mathbb Q
        \left(
        Z_N(\beta_N;\bar\omega)
        <
        \delta\,
        \mathbb E_{\mathbb Q}[Z_N(\beta_N;\bar\omega)]
        \right)
        =
        0.
\]
Therefore Proposition~\ref{truncated-error}, applied with
\(\tau_N=a_N\), gives
\[
        \frac{
        \log Z_N(\beta_N;\omega)
        -
        \log Z_N(\beta_N;\bar\omega)
        }
        {a_N}
        \stackrel{\mathbb Q}{\longrightarrow}0.
\]

We now pass from the normalized partition function to its logarithm.  From
\eqref{eq:truncated-fluctuation-theorem14},
\[
        Z_N(\beta_N;\bar\omega)e^{-N\bar\lambda(\beta_N)}-1
        =
        O_{\mathbb Q}(a_N).
\]
Since \(a_N=\beta_NN^{H/2}\to0\), Taylor's formula gives
\[
        \frac{
        \log Z_N(\beta_N;\bar\omega)-N\bar\lambda(\beta_N)
        }
        {a_N}
        -
        \frac{
        Z_N(\beta_N;\bar\omega)e^{-N\bar\lambda(\beta_N)}-1
        }
        {a_N}
         \stackrel{\mathbb Q}{\longrightarrow}0.
\]
Consequently,
\[
        \frac{
        \log Z_N(\beta_N;\bar\omega)-N\bar\lambda(\beta_N)
        }
        {a_N}
        \stackrel{d}{\longrightarrow}
        \mathcal N(0,\sigma_H^2).
\]

It remains to expand \(N\bar\lambda(\beta_N)\).
Let \(p'\) be the truncation order; namely,
\[
        1\le p'\le \left\lfloor\frac2H\right\rfloor,
        \qquad
        p'<\alpha,
        \qquad\text{and}\quad
        \frac{N\beta_N^{p'+1}}{a_N}\longrightarrow0.
\]

Choose \(\theta\in(2,\alpha)\) such that
\[
        \lfloor\theta\rfloor\ge p',
        \qquad
        r\theta>1,
        \qquad
        \frac{N\beta_N^\theta}{a_N}\longrightarrow0.
\]
This \(\theta\) exists, which makes Lemma \eqref{lem:cgf_expansion} applicable. Indeed, \(r>1/2\), so \(r\theta>1\) for every
\(\theta>2\).  The assumption
\(N\beta_N^{p'+1}/a_N\to0\) handles the case \(p'+1<\alpha\), while if
\(p'+1\ge\alpha\) one chooses \(\theta<\alpha\) sufficiently close to
\(\alpha\); then
\(
        \frac{N\beta_N^\theta}{a_N}\to0
\)
because \(\beta_N=\beta l(N^{3/2})^{-1}=N^{-3/(2\alpha)+o(1)}\) and
\[
        \frac{N\beta_N^\alpha}{a_N}
        =
        N^{1-H/2}\beta_N^{\alpha-1}
        =
        N^{(3/\alpha-H-1)/2+o(1)}
        \longrightarrow0,
\]
using \(\alpha>2\) and \(H>1/2\). 

Applying Lemma~\ref{lem:cgf_expansion} with \(q=1\) and 
the \(\theta\) chosen last paragraph, and using \(\mathbb E_{\mathbb Q}[\bar\xi]=0\), gives
\[
        \bar\lambda(\beta_N)
        =
        \sum_{y\in\mathbb Z}
        \log\left(
        1+
        \sum_{j=2}^{\lfloor\theta\rfloor}
        \frac{(\beta_N\psi_y)^j}{j!}
        \mathbb E_{\mathbb Q}[\bar\xi^j]
        \right)
        +
        o(\beta_N^\theta).
\]
After multiplication by \(N/a_N\), the last error is \(o(1)\).  As before,
if
\[
        Y_{N,y}:=
        \sum_{j=2}^{\lfloor\theta\rfloor}
        \frac{(\beta_N\psi_y)^j}{j!}
        \mathbb E_{\mathbb Q}[\bar\xi^j],
\]
then \(|Y_{N,y}|\le C\beta_N^2\psi_y^2\) and
\[
        \frac{N}{a_N}\sum_y Y_{N,y}^2
        \le
        C\frac{N\beta_N^4}{a_N}
        =
        C a_N^3N^{1-2H}
        \longrightarrow0.
\]
Thus the logarithm may be replaced by its linear part at the scale \(a_N\).
The terms with orders \(j>p'\) are negligible.  Indeed, for large \(N\),
\(\beta_N\le1\), while \(\sum_y|\psi_y|^j<\infty\) for every \(j\ge2\), and
therefore
\[
        \frac{N\beta_N^j}{a_N}\le
        \frac{N\beta_N^{p'+1}}{a_N}
        \longrightarrow0
        \qquad (j\ge p'+1).
\]
Therefore
\[
        \frac{
        N\bar\lambda(\beta_N)
        -
        N\sum_{j=2}^{p'}
        \frac{\beta_N^j}{j!}
        \mathbb E_{\mathbb Q}[\bar\xi^j]
        \sum_{y\in\mathbb Z}\psi_y^j
        }
        {a_N}
        \longrightarrow0.
\]

Consequently, we obtain
\[
        \frac{1}{\beta_NN^{H/2}}
        \left(
        \log Z_N(\beta_N;\omega)
        -
        N\sum_{j=2}^{p'}
        \frac{\beta_N^j}{j!}
        \mathbb E_{\mathbb Q}[\bar\xi^j]
        \sum_{y\in\mathbb Z}\psi_y^j
        \right)
        \xrightarrow{d}
        \mathcal N(0,\sigma_H^2).
\]
This completes the proof.
\end{proof}

\begin{appendices}
\section{Proof of Lemma 1.7}
\begin{proof}
Without loss of generality, we set \(x=0\).  Let
\[
        Z_y:=\xi(i,y)\mathbf 1_{\{|\xi(i,y)|\le k_N\}},
        \qquad
        c_N:=\mathbb E_{\mathbb Q}[Z_y],
\]
so that the centered truncated variable is
\[
        \bar\xi(i,y)=Z_y-c_N .
\]
Since the variables \(\{\xi(i,y)\}_{y\in\mathbb Z}\) are independent under
\(\mathbb Q\), the log-cumulant generating function factorizes as
\[
\begin{aligned}
\Lambda_N(q)
&=
\log\mathbb E_{\mathbb Q}
\left[
\exp\left\{
q\beta_N\sum_{y\in\mathbb Z}\psi_y\bar\xi(i,y)
\right\}
\right]                                                     \\
&=
\sum_{y\in\mathbb Z}
\log\mathbb E_{\mathbb Q}
\left[
e^{q\beta_N\psi_y\bar\xi(i,y)}
\right]                                                     \\
&=
\sum_{y\in\mathbb Z}
\log\left(
e^{-q\beta_N\psi_yc_N}
\mathbb E_{\mathbb Q}
\left[e^{q\beta_N\psi_y Z_y}\right]
\right).
\end{aligned}
\]
Put
\[
        \lambda_{N,y}:=q\beta_N\psi_y .
\]
Since \(q\) is fixed, \(\sup_y\psi_y<\infty\), and \(\beta_N\to0\), we have
\(\sup_y|\lambda_{N,y}|\to0\).  Let
\[
        \varphi_p(u):=e^u-\sum_{j=0}^p\frac{u^j}{j!}.
\]
We first estimate the Taylor remainder for the uncentered truncated variable.
We shall use the elementary bound
\[
        |\varphi_p(u)|\le C_p |u|^{p+1} e^{u_+},
        \qquad u\in\mathbb R,
\]
where \(u_+=\max\{u,0\}\).  Since \(|Z_y|\le k_N\), it gives the bound \(e^{|\lambda_{N,y}|k_N}\). Then we split at \(a|\lambda_{N,y}|^{-1}\), with \(a>0\) fixed.

On the region \(\{|Z_y|\le a|\lambda_{N,y}|^{-1}\}\), the usual Taylor
bound gives
\[
\begin{aligned}
&\mathbb E_{\mathbb Q}
\left[
|\varphi_p(\lambda_{N,y}Z_y)|
\mathbf 1_{\{|Z_y|\le a|\lambda_{N,y}|^{-1}\}}
\right]                                                     \\
&\qquad\le
C|\lambda_{N,y}|^{p+1}
\mathbb E_{\mathbb Q}
\left[
|\xi|^{p+1}
\mathbf 1_{\{|\xi|\le a|\lambda_{N,y}|^{-1}\}}
\right].
\end{aligned}
\]
By Karamata's theorem applied to the tail, and by the finiteness of moments below \(\alpha\), this
is bounded by
\[
        C\left(
        |\lambda_{N,y}|^{p+1}
        +
        |\lambda_{N,y}|^\alpha
        L(|\lambda_{N,y}|^{-1})
        \right).
\]
Here and below \(L\) may change from line to line, but is always slowly
varying. This unified bound covers both cases \(p+1<\alpha\) and
\(p+1\ge\alpha\).

On the complementary region
\(\{|Z_y|>a|\lambda_{N,y}|^{-1}\}\), the same remainder bound and
\(|Z_y|\le k_N\) yield
\[
\begin{aligned}
&\mathbb E_{\mathbb Q}
\left[
|\varphi_p(\lambda_{N,y}Z_y)|
\mathbf 1_{\{|Z_y|>a|\lambda_{N,y}|^{-1}\}}
\right]                                                     \\
&\qquad\le
C e^{|\lambda_{N,y}|k_N}
|\lambda_{N,y}|^{p+1}
\mathbb E_{\mathbb Q}
\left[
|\xi|^{p+1}
\mathbf 1_{\{a|\lambda_{N,y}|^{-1}<|\xi|\le k_N\}}
\right]                                                     \\
&\qquad\le
C e^{|\lambda_{N,y}|k_N}
\left(
|\lambda_{N,y}|^{p+1}
+
|\lambda_{N,y}|^\alpha
L(|\lambda_{N,y}|^{-1})
\right).
\end{aligned}
\]
Consequently, uniformly in \(y\),
\begin{equation}\label{eq:lem17-raw-remainder}
\left|
\mathbb E_{\mathbb Q}
\left[e^{\lambda_{N,y}Z_y}\right]
-
\sum_{j=0}^p
\frac{\lambda_{N,y}^j}{j!}
\mathbb E_{\mathbb Q}[Z_y^j]
\right|
\le
C e^{|\lambda_{N,y}|k_N}
\left(
|\lambda_{N,y}|^{p+1}
+
|\lambda_{N,y}|^\alpha L(|\lambda_{N,y}|^{-1})
\right).
\end{equation}

We now pass from \(Z_y\) to the centered variable \(\bar\xi=Z_y-c_N\).
Since \(p<\alpha\), all moments of \(Z_y\) up to order \(p\) are uniformly
bounded.  Also \(|c_N|\le \mathbb E_{\mathbb Q}[|\xi|\mathbf 1_{\{|\xi|\le
k_N\}}]\le C\).  Expanding \(e^{-\lambda_{N,y}c_N}\) up to order \(p\) and
using the Cauchy product of the two Taylor polynomials, the polynomial part is
exactly the Taylor polynomial in the centered moments:
\[
\sum_{j=0}^p
\frac{\lambda_{N,y}^j}{j!}
\mathbb E_{\mathbb Q}[(Z_y-c_N)^j].
\]
The terms lost by truncating the exponential
\(e^{-\lambda_{N,y}c_N}\) are \(O(|\lambda_{N,y}|^{p+1})\), and are absorbed
by the right-hand side of \eqref{eq:lem17-raw-remainder}.  Therefore
\begin{equation}\label{eq:lem17-centered-expansion}
\mathbb E_{\mathbb Q}
\left[e^{\lambda_{N,y}\bar\xi}\right]
=
1+
\sum_{j=1}^p
\frac{\lambda_{N,y}^j}{j!}
\mathbb E_{\mathbb Q}[\bar\xi^j]
+
E_{N,y},
\end{equation}
where
\begin{equation}\label{eq:lem17-local-error}
|E_{N,y}|
\le
C e^{|\lambda_{N,y}|k_N}
\left(
|\lambda_{N,y}|^{p+1}
+
|\lambda_{N,y}|^\alpha L(|\lambda_{N,y}|^{-1})
\right).
\end{equation}

Set
\[
        M_{N,y}:=
        \sum_{j=1}^p
        \frac{\lambda_{N,y}^j}{j!}
        \mathbb E_{\mathbb Q}[\bar\xi^j].
\]
Because \(\mathbb E_{\mathbb Q}[\bar\xi]=0\), the first non-zero term starts
at order two; in particular
\[
        \sup_y |M_{N,y}|\to0.
\]
Moreover, by \eqref{eq:lem17-local-error} and Potter bounds,
\[
\begin{aligned}
        \sup_y |E_{N,y}|
        &\le
        C e^{q\beta_Nk_N\sup_z\psi_z}
        \left(
        \beta_N^{p+1}
        +
        \beta_N^\alpha L(\beta_N^{-1})
        \right)
        \to0 .
\end{aligned}
\]
Hence, for all large \(N\), we can set
\[
        |1+M_{N,y}|\ge\frac12,
        \qquad
        |1+M_{N,y}+E_{N,y}|\ge\frac14,
        \qquad y\in\mathbb Z.
\]
Using the mean value theorem for the logarithm,
\[
\left|
\log(1+M_{N,y}+E_{N,y})
-
\log(1+M_{N,y})
\right|
\le C|E_{N,y}|.
\]
Therefore
\begin{equation}\label{eq:lem17-log-error-sum}
\begin{aligned}
&\left|
\Lambda_N(q)
-
\sum_{y\in\mathbb Z}
\log\left(
1+
\sum_{j=1}^p
\frac{(q\beta_N\psi_y)^j}{j!}
\mathbb E_{\mathbb Q}[\bar\xi^j]
\right)
\right|                                                     \\
&\qquad\le
C
\sum_{y\in\mathbb Z}
e^{q\beta_Nk_N\sup_z\psi_z}
\left[
(\beta_N\psi_y)^{p+1}
+
(\beta_N\psi_y)^\alpha
L\big((\beta_N\psi_y)^{-1}\big)
\right].
\end{aligned}
\end{equation}
We now make the last sum independent of \(y\).  Choose
\(\varepsilon>0\) so small that \(r(\alpha-\varepsilon)>1\).  Potter bounds
give, uniformly in \(0<\psi_y\le\sup_z\psi_z\),
\[
        L\big((\beta_N\psi_y)^{-1}\big)
        \le
        C_\varepsilon
        \psi_y^{-\varepsilon}L(\beta_N^{-1})
\]
for all large \(N\).  Since \(p+1>2\), \(\alpha-\varepsilon>1/r\), and
\(\psi_y\asymp |y|^{-r}\), we have
\[
        \sum_y \psi_y^{p+1}<\infty,
        \qquad
        \sum_y \psi_y^{\alpha-\varepsilon}<\infty .
\]
Thus \eqref{eq:lem17-log-error-sum} gives
\begin{equation}\label{eq:lem17-global-error}
\begin{aligned}
&\left|
\Lambda_N(q)
-
\sum_{y\in\mathbb Z}
\log\left(
1+
\sum_{j=1}^p
\frac{(q\beta_N\psi_y)^j}{j!}
\mathbb E_{\mathbb Q}[\bar\xi^j]
\right)
\right|                                                     \\
&\qquad\le
C
e^{q\beta_Nk_N\sup_z\psi_z}
\left(
\beta_N^{p+1}
+
\beta_N^\alpha L(\beta_N^{-1})
\right).
\end{aligned}
\end{equation}

It remains to verify that the right-hand side is \(o(\beta_N^\theta)\).
Since \(p=\lfloor\theta\rfloor\), we have \(p+1>\theta\), hence
\[
        \beta_N^{p+1}=o(\beta_N^\theta).
\]
For the regularly varying term, fix
\(\varepsilon\in(0,\alpha-\theta)\).  By Potter bounds,
\[
        \beta_N^\varepsilon L(\beta_N^{-1})\to0.
\]
Hence
\[
        \beta_N^\alpha L(\beta_N^{-1})
        =
        \beta_N^\theta
        \beta_N^{\alpha-\theta-\varepsilon}
        \big(\beta_N^\varepsilon L(\beta_N^{-1})\big).
\]
Therefore the desired conclusion follows once
\[
        \left(
        \beta_N^{p+1-\theta}
        +
        \beta_N^{\alpha-\theta-\varepsilon}
        \right)
        e^{q\beta_Nk_N\sup_z\psi_z}
        \to0
\]
for some such \(\varepsilon>0\). In the truncation \(k_N=\beta_N^{-1}\), the exponential
factor is bounded.  In the truncation,
\[
        k_N
        =
        \beta_N^{-1}
        \frac{l(N^{3/2}(\log N)^\eta)}{l(N^{3/2})},
        \qquad \eta\in(1,\alpha),
\]
and therefore, by regular variation of \(l\) with index \(1/\alpha\),
\[
        \beta_Nk_N
        =
        \frac{l(N^{3/2}(\log N)^\eta)}{l(N^{3/2})}
        =
        (\log N)^{\eta/\alpha+o(1)}.
\]
Since \(\eta/\alpha<1\),
\[
        \exp\{q\beta_Nk_N\sup_z\psi_z\}
        =
        \exp\{C(\log N)^{\eta/\alpha+o(1)}\}
        =
        N^{o(1)}.
\]
On the other hand,
\[
        \beta_N
        =
        \beta\,l(N^{3/2})^{-1}
        =
        N^{-3/(2\alpha)+o(1)}.
\]
Combining this with
\eqref{eq:lem17-global-error}, we obtain
\[
\frac1{\beta_N^\theta}
\left|
\Lambda_N(q)
-
\sum_{y\in\mathbb Z}
\log\left(
1+
\sum_{j=1}^p
\frac{(q\beta_N\psi_y)^j}{j!}
\mathbb E_{\mathbb Q}[\bar\xi^j]
\right)
\right|
\longrightarrow0 .
\]
This completes the proof.
\end{proof}

\begin{corollary}\label{cor:variance_estimate}
Recall the non-linear environment variable as $\hat{\omega}(i,x) := \frac{1}{\beta_N}\left(e^{\beta_N \bar{\omega}(i,x) - \bar{\lambda}(\beta_N)} - 1\right)$, where $\bar{\lambda}(\beta_N) = \Lambda_N(1)$. As $N \to \infty$, its variance converges asymptotically to the variance of the correlated environment:
$$
\mathbb{E}_{\mathbb{Q}}[\hat{\omega}(i,x)^2] \approx \bar{\gamma}(0).
$$
\end{corollary}

\begin{proof}
By definition, $\mathbb{E}_{\mathbb{Q}}[\hat{\omega}(i,x)] = 0$. The variance is exactly its second moment, which can be expressed via the cumulant generating function evaluated at $q=2$ and $q=1$:
$$
\mathbb{E}_{\mathbb{Q}}[\hat{\omega}^2] = \frac{1}{\beta_N^2} \left( \mathbb{E}_{\mathbb{Q}}[e^{2\beta_N \bar{\omega}(i,x)}] e^{-2\bar{\lambda}(\beta_N)} - 1 \right) = \frac{1}{\beta_N^2} \left( \exp\big(\Lambda_N(2) - 2\Lambda_N(1)\big) - 1 \right).
$$
Applying Lemma \ref{lem:cgf_expansion} with $p=2$ (since $\theta > 2$), and noting that the truncated noise is centered ($\mathbb{E}_{\mathbb{Q}}[\bar{\xi}] = 0$), the expansion reduces to:
$$
\Lambda_N(q) \approx \sum_{y=-\infty}^{\infty} \log \left( 1 + \frac{1}{2} q^2 \beta_N^2 \psi_y^2 \mathbb{E}_{\mathbb{Q}}[\bar{\xi}^2] \right) \approx \frac{q^2 \beta_N^2}{2} \sigma_{\bar{\xi}}^2 \sum_{y=-\infty}^{\infty} \psi_y^2.
$$
Substituting $q=2$ and $q=1$ into the exponent difference yields:
$$
\Lambda_N(2) - 2\Lambda_N(1) \approx \frac{4 \beta_N^2}{2} \sigma_{\bar{\xi}}^2 \sum_{y} \psi_y^2 - 2 \frac{\beta_N^2}{2} \sigma_{\bar{\xi}}^2 \sum_{y} \psi_y^2 = \beta_N^2 \sigma_{\bar{\xi}}^2 \sum_{y} \psi_y^2.
$$
Recalling the definition of the spatial covariance $\bar{\gamma}(x-z) = \sum_{y} \psi_{y-x}\psi_{y-z} \sigma_{\bar{\xi}}^2$, we identify $\bar{\gamma}(0) = \sigma_{\bar{\xi}}^2 \sum_{y} \psi_y^2$. Therefore,
$$
\mathbb{E}_{\mathbb{Q}}[\hat{\omega}^2] \approx \frac{1}{\beta_N^2} \left( \exp\big(\beta_N^2 \bar{\gamma}(0)\big) - 1 \right) \approx \bar{\gamma}(0),
$$
which concludes the verification.
\end{proof}
\section{Proof of Bound for Linear Environment Variable's Moments}

\begin{lemma} \label{lemma:rosenthal_linear}
Under conditions \textbf{(A)} and \textbf{(B)}, for any $p \ge 2$, there exists an absolute constant $K_p \le 14.7 \frac{p}{\ln p}$ such that the $L^p(\mathbb{Q})$-norm of the linear environment variable satisfies:
\begin{equation}
    \|\bar{\omega}(i,x)\|_p \le K_p \max \left\{ \left( \sum_{y=-\infty}^{\infty} \psi_y^2 \right)^{1/2}, \left( \sum_{y=-\infty}^{\infty} |\psi_y|^p \mathbb{E}_{\mathbb{Q}}[|\bar{\xi}|^p] \right)^{1/p} \right\}.
\end{equation}
\end{lemma}

\begin{proof}
Without loss of generality, we set $x=0$ and denote the truncated noise as $\bar{\xi}_y := \bar{\xi}(i, y)$ for simplicity. The linear environment variable is then $\bar{\omega}(i,0) = \sum_{y=-\infty}^{\infty} \psi_y \bar{\xi}_y$. Since the sequence $\{\bar{\xi}_y\}$ consists of independent, centered random variables, we can apply the martingale techniques and moment inequalities developed by Johnson, Schechtman, and Zinn (1985) \cite{Johnson-inequality}.

By a standard symmetrization argument (see the Remark following Theorem 4.1 in \cite{Johnson-inequality}), we can assume for the moment that the variables $\bar{\xi}_y$ are symmetric, which ultimately introduces an additional factor of 2 to the final $L^p$ bound. Under this symmetry assumption, any symmetric truncation preserves the centered and orthogonal properties of the sequence.

We define a sequence of truncation events based on the $L^2$-norm of the linear environment to separate the extreme values from the bulk fluctuations:
$$
    A_y = \left\{ |\psi_y \bar{\xi}_y| \ge \|\bar{\omega}\|_2 \right\}, \quad y \in \mathbb{Z}.
$$
By the Minkowski inequality, we can decompose the $L^p$-norm of $\bar{\omega}$ into two parts:
\begin{equation} \label{eq:rosenthal_split}
    \|\bar{\omega}\|_p = \left\| \sum_{y=-\infty}^{\infty} \psi_y \bar{\xi}_y \right\|_p \le \left\| \sum_{y=-\infty}^{\infty} \psi_y \bar{\xi}_y I_{A_y^c} \right\|_p + \left\| \sum_{y=-\infty}^{\infty} \psi_y \bar{\xi}_y I_{A_y} \right\|_p.
\end{equation}

For the first term in \eqref{eq:rosenthal_split}, the variables are bounded. We utilize the Orlicz norm $\|\cdot\|_\psi$ associated with the function $\psi(t) = t \ln(1+t)$. Combining Prokhorov's arcsinh inequality with the properties of the Orlicz space (see Proposition 3.6 in \cite{Johnson-inequality}), we obtain:
$$
    \left\| \sum_{y=-\infty}^{\infty} \psi_y \bar{\xi}_y I_{A_y^c} \right\|_p \le e^{1/p} \psi^{-1}(p) \left\| \sum_{y=-\infty}^{\infty} \psi_y \bar{\xi}_y I_{A_y^c} \right\|_\psi.
$$
Further bounding the Orlicz norm using Corollary 3.5 from the same reference yields:
$$
    \left\| \sum_{y=-\infty}^{\infty} \psi_y \bar{\xi}_y I_{A_y^c} \right\|_\psi \le 2\left(\frac{4+e}{e}\right) \max \left\{ \left\| \sum_{y=-\infty}^{\infty} \psi_y \bar{\xi}_y I_{A_y^c} \right\|_2, \max_y \|\psi_y \bar{\xi}_y I_{A_y^c}\|_\infty \right\}.
$$
By the definition of the set $A_y^c$, the supremum norm is trivially bounded by $\|\bar{\omega}\|_2$. Additionally, because we established the symmetry assumption, the truncated variables $\bar{\xi}_y I_{A_y^c}$ remain centered and orthogonal. Consequently, the $L^2$-norm of the truncated sum is bounded by the $L^2$-norm of the full sum: $\left\| \sum_y \psi_y \bar{\xi}_y I_{A_y^c} \right\|_2 \le \|\bar{\omega}\|_2$. Thus, the first term is entirely controlled by the $L^2$-norm:
\begin{equation} \label{eq:term1_bound}
    \left\| \sum_{y=-\infty}^{\infty} \psi_y \bar{\xi}_y I_{A_y^c} \right\|_p \le 2\left(\frac{4+e}{e}\right) e^{1/p} \psi^{-1}(p) \|\bar{\omega}\|_2.
\end{equation}

For the second term in \eqref{eq:rosenthal_split}, we apply the Rosenthal-type inequality specifically derived for sums of independent non-negative random variables (Theorem 2.5 in \cite{Johnson-inequality}):
$$
    \left\| \sum_{y=-\infty}^{\infty} \psi_y \bar{\xi}_y I_{A_y} \right\|_p \le \left\| \sum_{y=-\infty}^{\infty} |\psi_y \bar{\xi}_y| I_{A_y} \right\|_p \le \frac{2p}{\ln p} \max \left\{ \left\| \sum_{y=-\infty}^{\infty} |\psi_y \bar{\xi}_y| I_{A_y} \right\|_1, \left( \sum_{y=-\infty}^{\infty} \|\psi_y \bar{\xi}_y I_{A_y}\|_p^p \right)^{1/p} \right\}.
$$
Notice that the $L^1$-norm can be bounded using the definition of the indicator function $I_{A_y}$:
$$
    \left\| \sum_{y=-\infty}^{\infty} |\psi_y \bar{\xi}_y| I_{A_y} \right\|_1 \le \sum_{y=-\infty}^{\infty} \mathbb{E}_{\mathbb{Q}} \left[ \frac{|\psi_y \bar{\xi}_y|^2}{\|\bar{\omega}\|_2} \right] = \frac{\sum_{y=-\infty}^{\infty} \mathbb{E}_{\mathbb{Q}}[|\psi_y \bar{\xi}_y|^2]}{\|\bar{\omega}\|_2} = \|\bar{\omega}\|_2.
$$
Therefore, the second term is bounded by:
\begin{equation} \label{eq:term2_bound}
    \left\| \sum_{y=-\infty}^{\infty} \psi_y \bar{\xi}_y I_{A_y} \right\|_p \le \frac{2p}{\ln p} \max \left\{ \|\bar{\omega}\|_2, \left( \sum_{y=-\infty}^{\infty} |\psi_y|^p \mathbb{E}_{\mathbb{Q}}[|\bar{\xi}_y|^p] \right)^{1/p} \right\}.
\end{equation}

Combining \eqref{eq:term1_bound} and \eqref{eq:term2_bound}, we extract the common maximums to obtain the global bound:
$$
    \|\bar{\omega}\|_p \le K_p \max \left\{ \|\bar{\omega}\|_2, \left( \sum_{y=-\infty}^{\infty} |\psi_y|^p \mathbb{E}_{\mathbb{Q}}[|\bar{\xi}_y|^p] \right)^{1/p} \right\},
$$
where the absolute prefactor is exactly $K_p = 2\left(\frac{4+e}{e}\right) e^{1/p} \psi^{-1}(p) + \frac{2p}{\ln p}$. By analyzing the asymptotic behavior of the inverse Orlicz function $\psi^{-1}(p)$ associated with $\psi(t) = t \ln(1+t)$, one can verify that this prefactor is bounded by $7.35 \frac{p}{\ln p}$ for symmetric variables. Incorporating the standard symmetrization factor of $2$ discussed earlier, we yield $K_p \le 14.7 \frac{p}{\ln p}$ for all $p \ge 2$. Finally, evaluating the variance gives $\|\bar{\omega}\|_2 = (\sum_y \psi_y^2 \mathbb{E}_{\mathbb{Q}}[\bar{\xi}^2])^{1/2} \le (\sum_y \psi_y^2)^{1/2}$ since $\mathbb{E}_{\mathbb{Q}}[\bar{\xi}^2] \le 1$. This completes the proof.
\end{proof}

\end{appendices}
\bibliographystyle{alpha}  

\bibliography{references}

\end{document}